\documentclass[a4paper, reqno, 12pt]{amsart}
\usepackage[utf8]{inputenc} 
\usepackage[T1, T2A]{fontenc}
\usepackage[english]{babel}
\usepackage{amsmath, amssymb, amsfonts, amsthm, amscd, mathrsfs,stmaryrd}
\usepackage{enumitem}
\usepackage[cmtip,arrow]{xy}
\usepackage{pb-diagram, pb-xy}
\usepackage[hidelinks]{hyperref}
\usepackage{longtable}
\setlist[enumerate]{label = (\alph*), ref=(\text{\alph*)}}
\setlist[itemize]{nolistsep}
\usepackage{epsfig}
\usepackage{graphicx}
\renewcommand{\phi}{\varphi}
\renewcommand{\ge}{\geqslant}
\renewcommand{\le}{\leqslant}
\newcommand{\CC}{\mathbb{C}}
\newcommand{\UU}{\mathbb{U}}
\newcommand{\KK}{\mathbb{K}}
\renewcommand{\AA}{\mathbb{A}}
\newcommand{\QQ}{\mathbb{Q}}
\newcommand{\ZZ}{\mathbb{Z}}
\newcommand{\GG}{\mathbb{G}}
\newcommand{\PP}{\mathbb{P}}
\newcommand{\FF}{\mathbb{F}}
\newcommand{\SAS}{\mathbb{S}}
\newcommand{\mm}{\mathfrak{m}}
\renewcommand{\gg}{\mathfrak{g}}
\def\Of{{\mathcal{O}}}
\def\Uf{{\mathcal{U}}}
\def\Ff{{\mathcal{F}}}
\def\RRR{{\mathfrak R}}
\def\SSS{{\mathfrak S}}
\def\UUU{{\mathfrak U}}
\def\fg{{\mathfrak g}}
\def\fh{{\mathfrak h}}
\def\fb{{\mathfrak b}}
\def\fn{{\mathfrak n}}
\def\ft{{\mathfrak t}}
\def\fq{{\mathfrak q}}
\def\fl{{\mathfrak l}}
\def\fp{{\mathfrak p}}
\newcommand{\NN}{\mathbb{Z}_{>0}}
\newcommand{\Zgezero}{\mathbb{Z}_{\geqslant 0}}
\newcommand{\pa}{\partial}
\DeclareMathOperator{\Ker}{Ker}
\DeclareMathOperator{\GL}{GL}
\DeclareMathOperator{\SL}{SL}
\DeclareMathOperator{\SO}{SO}
\DeclareMathOperator{\Sp}{Sp}
\DeclareMathOperator{\Cl}{Cl}
\DeclareMathOperator{\codim}{codim}
\DeclareMathOperator{\gl}{\mathfrak{gl}}
\DeclareMathOperator{\Lie}{Lie}
\DeclareMathOperator{\Span}{Span}
\DeclareMathOperator{\Mat}{Mat}
\DeclareMathOperator{\Aut}{Aut}
\DeclareMathOperator{\PGL}{PGL}
\DeclareMathOperator{\PSL}{PSL}
\def\PSp{\mathop{\rm PSp}}
\def\PSO{\mathop{\rm PSO}}
\DeclareMathOperator{\End}{End}
\DeclareMathOperator{\id}{id}

\DeclareMathOperator{\Soc}{Soc}
\DeclareMathOperator{\WDiv}{WDiv}
\DeclareMathOperator{\Pic}{Pic}
\DeclareMathOperator{\Spec}{Spec}
\DeclareMathOperator{\Hom}{Hom}
\DeclareMathOperator{\Gr}{Gr}
\DeclareMathOperator{\Lag}{Lag}

\DeclareMathOperator{\dg}{deg}
\DeclareMathOperator{\ddiv}{div}
\DeclareMathOperator{\Sym}{Sym}
\theoremstyle{plain}
\newtheorem{lemma}{Lemma}[section]%
\newtheorem{proposition}[lemma]{Proposition}
\newtheorem{theorem}[lemma]{Theorem}
\newtheorem{corollary}[lemma]{Corollary}
\theoremstyle{definition}
\newtheorem{definition}[lemma]{Definition}
\newtheorem{construction}[lemma]{Construction}
\newtheorem{example}[lemma]{Example}
\newtheorem{conjecture}[lemma]{Conjecture}
\newtheorem{problem}[lemma]{Problem}
\theoremstyle{remark}
\newtheorem{remark}[lemma]{Remark}

\begin{document}

\title[Equivariant completions of affine spaces]{Equivariant completions of affine spaces}
\author{Ivan Arzhantsev}
\address{HSE University, Faculty of Computer Science, Pokrovsky Boulevard 11, Moscow, 109028 Russia}
\email{arjantsev@hse.ru}

\author{Yulia Zaitseva}
\address{HSE University, Faculty of Computer Science, Pokrovsky Boulevard 11, Moscow, 109028 Russia}
\email{yuliazaitseva@gmail.com}

\thanks{Supported by the Russian Science Foundation grant 19-11-00172.}



\begin{abstract}
We survey recent results on open embeddings of the affine space $\CC^n$ into a complete algebraic variety $X$ such that the action of the vector group $\GG_a^n$ on $\CC^n$ by translations extends to an action of $\GG_a^n$ on $X$. We begin with Hassett-Tschinkel correspondence describing equivariant embeddings of $\CC^n$ into projective spaces and give its generalization for embeddings into projective hypersurfaces. Further sections deal with embeddings into flag varieties and their degenerations, complete toric varieties, and Fano varieties of certain types. 
\end{abstract}

\maketitle

\tableofcontents


\section*{Introduction}

The survey is devoted to the study of completions of the affine space $\CC^n$ by an algebraic variety $X$ such that the action of the vector group $\GG_a^n$ on $\CC^n$ by translations can be extended to a regular action $\GG_a^n\times X\to X$. To obtain such a completion means to construct an effective regular action of the commutative unipotent group $\GG_a^n$ on a complete algebraic variety $X$ with an open orbit. We call an effective regular action $\GG_a^n\times X\to X$ with an open orbit an \emph{additive action} on $X$. One more interpretation comes from the theory of group embeddings. Let $G$ be a linear algebraic group. A \emph{group embedding} is an embedding of $G$ as an open subset into an algebraic variety $X$ such that the actions of $G$ on $G$ by left and right translations can be extended to a regular action of the group $G\times G$ on $X$. In these terms, we are going to study group embeddings of a commutative unipotent group. 

The story began with the work of Hirzebruch. In \cite[Section~3.2]{Hir1954} the author considered complex analytic compactifications of the affine space $\CC^n$. Problem~26 asks to determine all complex analytic compactifications of $\CC^2$, and Problem~27 raises the same question for all $\CC^n$ under the restriction that the compactification has the second Betti number~1. These problems initiated the study of open embeddings of affine spaces both in analytic and algebraic categories. For more information on algebraic compactifications of affine spaces, see e.g. \cite{Fur1993, PZ2018, CPPZ2021} and references therein. 

Clearly, an algebraic variety $X$ that contains an open subset $U$ isomorphic to an affine space possesses some specific properties. In particular, $X$ is rational, every invertible regular function on $X$ is constant, and the divisor class group $\Cl(X)$ is a free finitely generated abelian group. More precisely, the group $\Cl(X)$ is freely generated by classes of irreducible components of the complement $X\setminus U$. At the same time, the class of all compactifications of affine spaces is too wide, and it is natural to study compactifications satisfying some extra conditions.

The first variant is to consider algebraic manifolds $X$ in a naive sense, that is $X$ can be covered by open subsets $U_1,\ldots,U_m$ such that each $U_i$ is isomorphic to an affine space. Manifolds of this type were considered by Gromov in \cite[Section~3.5.D]{Gr1989}. In~\cite[Section~6.4]{For2011} such manifolds are called \emph{manifolds of class~$\mathcal{A}_0$}. They appear in connection with the Oka principle and algebraic ellipticity. It is known that class~$\mathcal{A}_0$ includes smooth projective rational surfaces, smooth complete toric varieties, flag varieties and, more generally, smooth complete spherical varieties. Moreover, this class is closed under taking blowing-up of points. In~\cite[Theorem~A.1]{APS2014} it is proved that any smooth complete rational variety with a torus action of complexity $1$ belongs to class~$\mathcal{A}_0$. A wider class is the class of uniformly rational varieties. A variety $X$ is \emph{uniformly rational} if every point in $X$ admits a Zariski open neighbourhood isomorphic to a Zariski open subset of the affine space. Some recent results on uniformly rational varieties can be found in~\cite{LP2019}.

The second variant is to involve algebraic group actions. Namely, if an algebraic group $G$ acts on the affine space $\CC^n$, we may study open embeddings of $\CC^n$ into complete varieties $X$ such that the action of $G$ on $\CC^n$ extends to an action of $G$ on $X$. Taking $G=\GG_a^n$ with the action $\GG_a^n\times\CC^n\to\CC^n$ by parallel translations, we come to the theory of additive actions. This is the subject of the present survey.

One more motivation to investigate equivariant completions of affine spaces comes from Arithmetic Geometry. In their study of Manin's Conjecture on the distribution of rational points on algebraic varieties, Chambert-Loir and Tschinkel~\cite{CLT2002} gave asymptotic formulas for the number of rational points of bounded height on smooth projective equivariant compactifications of the vector group.
More generally, asymptotic formulas for the number of rational points of bounded height on quasi-projective equivariant embeddings of the vector group are obtained in~\cite{CLT2012}. The limited volume of the survey does not allow us to discuss these results. We recommend the reader articles~\cite{CLT2002,CLT2012,Pey2002,ShT2016,TT2012} and references therein. 

It is natural to compare the theory of additive actions with the theory of toric varieties. At the first glance, two theories should be similar since the formulations of the problems are almost the same: in the toric case we study open equivariant embeddings of the group $\GG_m^n$, and in the theory of additive actions we just replace the multiplicative group $\GG_m$ of the ground field by the additive group $\GG_a$. But it turns out that toric geometry and the theory of additive actions have almost nothing in common. Let us stay a bit more on this. 

The theory of toric varieties plays an important role in modern Algebra, Combinatorics, Geometry, and Topology. It is caused by a beautiful description of toric varieties in terms of rational polyhedral cones and fans of such cones~\cite{CLS2011,Fu1993}. There are several ways to generalize the theory of toric varieties. For example, one may consider arbitrary torus actions on algebraic varieties. Recently a semi-combinatorial description of such actions in terms of so-called polyhedral divisors living on varieties of smaller dimensions was introduced~\cite{AH2006,AHS2008}. Another variant is to restrict the (complex) algebraic torus action on a toric variety to the maximal compact subtorus $(S^1)^n$, to axiomatize this class of $(S^1)^n$-actions, and to consider such actions on wider classes of topological spaces.
This is an active research area called Toric Topology~\cite{BP2002}. Further, one may consider linear algebraic group actions with an open orbit replacing the torus $T$ with a non-abelian connected reductive group $G$. In other words, one may study open equivariant embeddings of homogeneous spaces $G/H$, where $H$ is an algebraic subgroup of $G$. The theory is well-developed in the case when $G/H$ is a spherical homogeneous space, that is a Borel subgroup $B$ in $G$ acts on $G/H$ with an open orbit. Here a description of equivariant embeddings in terms of convex geometry is also available in the framework of the Luna-Vust theory, while it is more complicated than in the toric case~\cite{LV1983,Ti2011}.

Returning to an ``additive analogue'' of toric geometry, i.e. to the case when we replace the acting torus $T$ with the commutative unipotent group $\GG_a^n$, we come across principal differences. Firstly, it is well known that every orbit of an action of a unipotent group on an affine variety is closed~\cite[Section~1.3]{PV1994}. In particular, if a unipotent group acts on an affine variety with an open orbit, then the action is transitive. This means that, in contrast to the toric case, an irreducible algebraic variety with a non-transitive action of a unipotent group $U$  that contains an open $U$-orbit can not be covered by $U$-invariant open affine charts. Secondly, any toric variety contains finitely many $T$-orbits, and if two toric varieties are isomorphic as abstract algebraic varieties, then they are isomorphic in the category of toric varieties \cite[Theorem~4.1]{Be2003}. In the additive case these two properties do not hold: one may consider two actions of $\GG_a^2$ on the projective plane $\PP^2$ given in homogeneous coordinates as
$$
(a_1,a_2)\cdot [z_0:z_1:z_2]=[z_0:z_1+a_1z_0:z_2+a_2z_0]
$$
and
$$
(a_1,a_2)\cdot [z_0:z_1:z_2]=\bigl[z_0:z_1+a_1z_0:z_2+a_1z_1+\Bigl(\frac{a_1^2}{2}+a_2\Bigr)z_0\bigr].
$$
In the first case, there is a line consisting of fixed points, while for the second action there are exactly three $\GG_a^2$-orbits.

At the same time, the absence of analogy with toric geometry is definitely not the end of the theory of additive actions. During the last decades, many general and classification results on varieties with an additive action were obtained and some original methods to deal with this class of actions were developed. The present survey aims to discuss these results and methods. 

Let us describe the content of the paper. In Section~\ref{secpn}, we study additive actions on projective spaces. It is a certain surprise that the space $\CC^n$ can be embedded equivariantly in $\PP^n$ in many different ways. Hassett and Tschinkel~\cite{HaTs1999} observed that such embeddings are in bijection with local commutative associative unital algebras of dimension~$n+1$. This result also follows from a more general correspondence between finite-dimensional commutative associative unital algebras and open equivariant embeddings of commutative linear algebraic groups into projective spaces established by Knop and Lange~\cite{KnLa1984}. We begin with well-known structural theory and classification results on finite-dimensional commutative associative algebras and develop Hassett-Tschinkel correspondence in complete generality. In particular, it includes a nice correspondence with certain subspaces in the polynomial algebra that are invariant under some differential operators with constant coefficients. 

In Section~\ref{secght}, we show how the technique proposed by Hassett and Tschinkel can be applied to the study of additive actions on projective varieties different from projective spaces. 
This started already in~\cite{HaTs1999}, where projective curves, smooth projective surfaces and a special class of smooth projective 3-folds carrying additive actions were described.
Using this technique, we give a proof of Sharoiko's theorem~\cite{Sh2009}. It claims that, in contrast to projective spaces, any non-degenerate projective quadric admits a unique additive action. We also explain how one can describe additive actions on degenerate projective quadrics~\cite{ArSh2011,AP2014} and establish a generalization of Hassett-Tschinkel correspondence to arbitrary projective hypersurfaces in terms of invariant multilinear forms~\cite{AP2014,Ba2013}. In Theorem~\ref{hypGor_prop} we find a correspondence between additive actions on non-degenerate projective hypersurfaces and Gorenstein local algebras. Finally, Theorem~\ref{tnew} generalizes Sharoiko's result and claims that any non-degenerate projective hypersurface admits at most one additive action. Theorem~\ref{hypGor_prop}, Theorem~\ref{tnew} and some other statements in Section~\ref{secght} are original results of this article. 

Section~\ref{secaafv} begins with some general background on varieties with additive actions. Then we show that if a flag variety $G/P$ of a simple linear algebraic group $G$ admits an additive action then the parabolic subgroup $P$ is maximal. We list all varieties $G/P$ admitting an additive action following~\cite{Ar2011}. Then we discuss a uniqueness result which claims that if a flag variety is not isomorphic to the projective space then it admits at most one additive action. This theorem is proved by Fu-Hwang~\cite{FH2014} and independently by Devyatov~\cite{De2015}. The last part presents a construction due to Feigin~\cite{Fe2012} that degenerates arbitrary flag variety to a variety with an additive action. 

In Section~\ref{aatv}, we study additive actions on toric varieties following~\cite{AR2017}. It is proved that if a complete toric variety admits an additive action, then it admits an additive action normalized by the acting torus. Moreover, we show that any two normalized additive actions are equivalent and give a combinatorial criterion of the existence of a normalized additive action on a toric variety. These results are based on the theory of Cox rings and Demazure roots of toric varieties. Also, we present two results of Dzhunusov. The first one is a classification of additive actions on complete toric surfaces~\cite{Dz2019}, and the second one is a criterion of the uniqueness of an additive action on a complete toric variety~\cite{Dz2020}. 

In the last section, we discuss recent classification results on additive actions on generalized del Pezzo surfaces, Fano 3-folds, and varieties with high index due to Fu, Huang, Hwang, Montero, and Nagaoka \cite{FH2014,FH2017,FM2018,HM2018,Nag2021}. The special subsection is devoted to Euler-symmetric projective varieties introduced by Fu and Hwang. Every Euler-symmetric variety admits an additive action. Moreover, for wide classes of varieties including toric varieties and flag varieties the condition to be Euler-symmetric is equivalent to the existence of an additive action. 

We end the text with a list of open problems and possible directions for further research. 

\smallskip

{\it Acknowledgment}. \ The authors are grateful to Anthony Iarrobino and Joachim Jelisiejew for useful comments and references to works on local Artinian algebras. Discussions with Anton Shafarevich helped us a lot to understand the results of Fu and Hwang on Euler-symmetric varieties. Special thanks are due to the referees for many valuable suggestions and corrections that helped us to improve the text. 


\section{Equivariant embeddings into projective spaces} 
\label{secpn}
In this section, we study additive actions on projective spaces. In 1999, Hassett and Tschinkel~\cite{HaTs1999} established a remarkable correspondence between such actions and commutative associative local Artinian unital algebras. This correspondence led to classification results and allowed to employ new methods that were later generalized to some other classes of projective varieties. The main goal of this section is to introduce all objects and concepts that are needed to establish Hassett-Tschinkel correspondence, formulate the correspondence in complete generality and with detailed proofs, and discuss related results and corollaries. We work over an algebraically closed field $\KK$ of characteristic zero. 

In subsection~\ref{localg_subsect}, we begin with basic facts on finite-dimensional commutative associative algebras. Any finite-dimensional commutative associative algebra is a direct sum of local ones. So finite-dimensional local algebras are important building blocks in many problems of algebra and geometry, sometimes compatible with finite simple groups or finite fields. Although the classification of local algebras of small dimensions is known for many years, it is not easy to find it in explicit form in the literature. In Table~\ref{table_localg6}, we list all local algebras up to dimension~6~\footnote{Starting from dimension~7, the number of isomorphy classes of such algebras becomes infinite.}. We also introduce the Hilbert-Samuel sequence of a local algebra and define Gorenstein local algebras. 

Subsection~\ref{SuTy_subsect} is devoted to results of Suprunenko and Tyshkevich~\cite{SuTy1968}. We explain how information on maximal commutative nilpotent subalgebras of a matrix algebra can be used to study abstract commutative algebras and groups. In particular, one can deduce the classification of local algebras in Table~\ref{table_localg6} from the classification results in~\cite{SuTy1968}. The book contains many important facts and observations that are useful for our purposes, but it is not easy to extract them from the text. We hope that the subsection with unified formulations and, where it is possible, short proofs, 
may help the reader to understand better the results of Suprunenko and Tyshkevich. 

In subsection~\ref{subsecKLT} we prove a result of Knop and Lange~\cite{KnLa1984}. This result establishes a bijective correspondence between effective actions of commutative linear algebraic groups on the projective space $\PP^n$ with an open orbit and commutative associative unital algebras $A$ of dimension $n+1$. Also, we characterize the actions with finitely many orbits. 

Subsection~\ref{IVsubsection} contains preparatory results on a duality between subspaces of the polynomial algebra $\KK[x_1,\ldots,x_n]$ and of the algebra 
$\KK[\frac{\pa}{\pa x_1}, \ldots, \frac{\pa}{\pa x_n}]$ of differential operators with constant coefficients. In general, the duality is not bijective, but it defines a bijection being restricted to finite-dimensional subspaces in $\KK[x_1,\ldots,x_n]$ and subspaces in $\KK[\frac{\pa}{\pa x_1}, \ldots, \frac{\pa}{\pa x_n}]$ of finite codimension. Moreover, let us define a generating subspace in $\KK[x_1,\ldots,x_n]$ as a translation invariant subspace that generates the algebra $\KK[x_1,\ldots,x_n]$. It turns out that the duality provides a bijection between generating subspaces of dimension $m$ and non-degenerate ideals of codimension $m$ in $\KK[\frac{\pa}{\pa x_1}, \ldots, \frac{\pa}{\pa x_n}]$ supported at the origin.

Following Hassett and Tschinkel~\cite{HaTs1999}, in subsection~\ref{subsecHTC} we establish a correspondence between 
\begin{enumerate}
\item faithful cyclic representations $\rho\colon \GG_a^n \to \GL_m(\KK)$;
\item pairs $(A, U)$, where $A$ is a local commutative associative unital algebra of dimension $m$ with maximal ideal $\mm$, and $U \subseteq \mm$ is a subspace of dimension $n$ generating the algebra $A$;
\item non-degenerate ideals $I \subseteq \KK[S_1, \ldots, S_n]$ of codimension $m$ supported at the origin;
\item generating subspaces $V \subseteq \KK[x_1, \ldots, x_n]$ of dimension~$m$.
\end{enumerate} 
We give complete proofs including arguments for `up to isomorphism' statements that are usually ignored in the literature. An effective algorithm that finds the generating subspace corresponding to a pair $(A,U)$ is given. We illustrate the theory by explicit computations in low-dimensional cases. Also, it is shown that the $\GG_a^n$-modules $A$ and $V$ are dual to each other. 

In subsection~\ref{subsecAA} we show that restricting either Knop-Lange theorem to the case of a unipotent group or Hassett-Tschinkel correspondence to the case $m=n+1$, we obtain a bijection between additive actions on $\PP^n$ and local commutative associative unital algebras $A$ of dimension $n+1$. In this case, we come to a remarkable class of generating subspaces which we call basic subspaces. 
Such a subspace represents an automorphism of the open orbit $\GG_a^n$ in $\PP^n$ that conjugates an additive action to the standard action by translations in the automorphism group of the affine space. 
It is shown that there is a unique additive action on $\PP^n$ with finitely many orbits, and additive actions of modality one are described. Finally, we observe that an additive action has a unique fixed point if and only if the corresponding local algebra is Gorenstein. 


\subsection{Finite-dimensional algebras}
\label{localg_subsect}
In this subsection we recall basic structural and classification results on Artinian commutative algebras or, equivalently, finite-dimensional commutative associative unital algebras over the ground field $\KK$; see \cite[Chapter~8]{AtMa1969} for more information. Hereafter \emph{algebra} means finite-dimensional commutative associative unital algebra. The base field $\KK$ is embedded into an algebra as the linear span of the unit.

\begin{definition}
An algebra $A$ is called \emph{local} if it contains a unique maximal ideal $\mm$. 
\end{definition}

\begin{lemma}
\label{localg_lem}
An algebra $A$ is local if and only if $A$ is the direct sum of its subspaces $\KK\oplus\mm$, where $\mm$ is an ideal consisting of nilpotent elements. 
\end{lemma}

\begin{proof}
Let $A = \KK\oplus\mm$. The ideal $\mm$ is maximal since its codimension equals one. Any element of $A \setminus \mm$ is the sum of an invertible scalar and a nilpotent element, whence is invertible and can not belong to any proper ideal. Thus the ideal $\mm$ is a unique maximal ideal. 

Conversely, let $A$ be a local algebra with maximal ideal $\mm$. Let us show that any $a \in \mm$ is nilpotent. Since $A$ is finite-dimensional, we have the equality of ideals $(a^k) = (a^{k+1})$ for some $k \in \NN$, that is $a^k = a^{k+1}b$ and $a^k(ab-1) = 0$ for some $b \in A$. Note that $ab - 1 \notin \mm$, therefore $ab - 1$ does not belong to any proper ideal and so it is invertible. This implies $a^k = 0$. 

Denote by $L_a\colon A \to A$ the operator of multiplicaton by $a \in A$. Let $\lambda$ be an eigenvalue of $L_a$. Then $L_{a-\lambda\cdot1}$ is non-invertible, whence $a-\lambda\cdot1$ is non-invertible and belongs to the maximal ideal $\mm$. Together with $\KK \cap \mm = 0$ this implies $A = \KK\oplus\mm$. 
\end{proof}

The following lemma is a particular case of~\cite[Theorem 8.7]{AtMa1969}. 

\begin{lemma}
\label{directloc_lem}
Every algebra is the direct sum of its local ideals. 
\end{lemma}

\begin{proof}
As above, denote by $L_a\colon A \to A$ the operator of multiplication by $a \in A$. Recall that the generalized eigenspace of an operator $L\in \End(V)$ with respect to an eigenvalue $\lambda$ is the subspace $V^\lambda = \{v\in V \mid (L - \lambda\id_V)^kv = 0 \text{ for some }k\in \NN\}$. Let us prove that $A$ is the direct sum of its ideals $V_i$ lying in a generalized eigenspace of $L_a$ for any $a \in A$. Indeed, take some $a \in A$ and consider the generalized eigenspace decomposition $A = \bigoplus V_i^\prime$ with respect to $L_a$. All the generalized eigenspaces are ideals since $A$ is commutative. Repeating the decomposition procedure for those $V_i^\prime$ which do not lie in a generalized eigenspace of $L_b$ for some $b \in A$, we obtain the desired decomposition. 

The components $\varepsilon_i \in V_i$ of the unit in $A$ are units in $V_i$. By construction of $V_i$, for any $a_i \in V_i$ there is $\lambda \in \KK$ such that $\left(L_{a_i} - \lambda\id_A\right)\bigr|_{V_i} = L_{a_i - \lambda\varepsilon_i}\bigr|_{V_i}$ acts on $V_i$ nilpotently. Applying this operator to $\varepsilon_i \in V_i$ we obtain that $a_i - \lambda\varepsilon_i$ is nilpotent in $V_i$. So the algebra $V_i$ is local by Lemma~\ref{localg_lem}. 
\end{proof}

Let $A$ be a local algebra and $\mm$ be its maximal ideal. Consider the following series of ideals in $A$: 
\[A \supset \mm \supset \mm^2 \supset \ldots \supset \mm^{l-1} \supset \mm^l = 0.\]
The number $l$ is called the \emph{length} of the algebra $A$. Denote $r_i := \dim \mm^i - \dim \mm^{i+1}$. In particular, $r_0 = 1$. The sequence $r_0, r_1, r_2, \ldots, r_{l-1}$ is called the \emph{Hilbert-Samuel sequence} of the algebra $A$. 

The \emph{socle} of $A$ is the ideal $\Soc A = \{a \in A \mid \mm a = 0\}$. The algebras with $\dim \Soc A = 1$ are called \emph{Gorenstein}. Note that $\mm^{l-1} \subseteq \Soc A$, but the inclusion can be strict. So $A$ is Gorenstein if and only if $\mm^{l-1} = \Soc A$ and $\dim \mm^{l-1} = r_{l-1} = 1$. 

\smallskip

\begin{theorem}
\label{localg6_prop}
For $m \le 6$, the number of isomorphism classes of local algebras of dimension $m$ is finite. For $m \ge 7$, there are infinite series of non-isomorphic local algebras. The number of such classes is the following:
\[\begin{array}{c|c|c|c|c|c|c|c}
m & 1 & 2 & 3 & 4 & 5 & 6 & \ge 7\\\hline
& 1 & 1 & 2 & 4 & 9 & 25 & \infty
\end{array}\]
\end{theorem}

The local algebras of dimension at most $6$ are listed in the table below. Gorenstein algebras are marked with ``G''. It is observed in~\cite{HaTs1999} that this result can be extracted from the 1968 book of Suprunenko and Tyshkevich~\cite{SuTy1968}, see 2) -- 5) in the next subsection for details. The same classification was obtained independently and by other methods in the 1980 article of Mazolla~\cite[Section~2]{Ma1980}, where schemes parameterizing commutative nilpotent associative multiplications on the affine space are studied. One more approach to such a classification can be found in \cite{Poo2008}. 

\pagebreak

\begin{longtable}{|c|c|c|c|}
\hline
№ & Local algebra $A$ & $r_0, r_1, \ldots, r_{l-1}$ & \\
\hline
\multicolumn{4}{l}{\quad $\dim A = 1$ \parbox[b][15pt]{0pt}{}} \\
\hline
1& $\KK$ & 1 & G\\
\hline
\multicolumn{4}{l}{\quad $\dim A = 2$ \parbox[b][15pt]{0pt}{}} \\
\hline
2& $\KK[x_1]/(x_1^2)$ & 1, 1 & G\\
\hline
\multicolumn{4}{l}{\quad $\dim A = 3$ \parbox[b][15pt]{0pt}{}} \\
\hline
3& $\KK[x_1]/(x_1^3)$ & 1, 1, 1 & G\\
\hline
4& $\KK[x_1,x_2]/(x_1^2,x_1x_2,x_2^2)$ & 1, 2 &\\
\hline
\multicolumn{4}{l}{\quad $\dim A = 4$ \parbox[b][15pt]{0pt}{}} \\
\hline
5& $\KK[x_1]/(x_1^4)$ & 1, 1, 1, 1 & G\\
\hline
6& $\KK[x_1,x_2]/(x_1x_2, x_1^2-x_2^2)$ & 1, 2, 1 & G\\
\hline
7& $\KK[x_1,x_2]/(x_1^3,x_1x_2,x_2^2)$ & 1, 2, 1 &\\
\hline
8& $\KK[x_1,x_2,x_3]/(x_i^2,x_ix_j)$ & 1, 3 &\\
\hline
\multicolumn{4}{l}{\quad $\dim A = 5$ \parbox[b][15pt]{0pt}{}} \\
\hline
9& $\KK[x_1]/(x_1^5)$ & 1, 1, 1, 1, 1 & G\\
\hline
10 & $\KK[x_1,x_2]/(x_1x_2,x_1^3-x_2^2)$ & 1, 2, 1, 1 & G\\
\hline
11 & $\KK[x_1,x_2]/(x_1^3,x_2^3,x_1x_2)$ & 1, 2, 2 &\\
\hline
12 & $\KK[x_1,x_2]/(x_1^4,x_2^2,x_1x_2)$ & 1, 2, 1, 1 &\\
\hline
13 & $\KK[x_1,x_2]/(x_1^3,x_2^2,x_1^2x_2)$ & 1, 2, 2 &\\
\hline
14 & $\KK[x_1,x_2,x_3]/(x_1x_2,x_1x_3,x_2x_3,x_1^2-x_2^2,x_1^2-x_3^2)$ & 1, 3, 1 & G\\
\hline
15 & $\KK[x_1,x_2,x_3]/(x_1^2,x_1x_2,x_1x_3,x_2x_3,x_2^2-x_3^2)$ & 1, 3, 1 &\\
\hline
16 & $\KK[x_1,x_2,x_3]/(x_1^3,x_2^2,x_3^2,x_1x_2,x_1x_3,x_2x_3)$ & 1, 3, 1 &\\
\hline
17 & $\KK[x_1,x_2,x_3,x_4]/(x_i^2,x_ix_j)$ & 1, 4 &\\
\hline
\multicolumn{4}{l}{\quad $\dim A = 6$ \parbox[b][15pt]{0pt}{}} \\
\hline
18 & $\KK[x_1]/(x_1^6)$ & \!1, 1, 1, 1, 1, 1\! & G\\
\hline
19 & $\KK[x_1,x_2]/(x_1x_2,x_1^4-x_2^2)$ & 1, 2, 1, 1, 1 & G\\
\hline
20 & $\KK[x_1,x_2]/(x_1x_2,x_1^3-x_2^3)$ & 1, 2, 2, 1 & G\\
\hline
21& $\KK[x_1,x_2]/(x_1^3,x_2^2)$ & 1, 2, 2, 1 & G\\
\hline
22& $\KK[x_1,x_2]/(x_1^5,x_1x_2,x_2^2)$ & 1, 2, 1, 1, 1 &\\
\hline
23& $\KK[x_1,x_2]/(x_1^4,x_1x_2,x_2^3)$ & 1, 2, 2, 1 &\\
\hline
24& $\KK[x_1,x_2]/(x_1^3,x_1^2x_2,x_1x_2^2,x_2^3)$ & 1, 2, 3 &\\
\hline
25& $\KK[x_1,x_2]/(x_1^4,x_1^2x_2,x_1^3-x_2^2)$ & 1, 2, 2, 1 &\\
\hline
26& $\KK[x_1,x_2]/(x_1^4,x_1^2x_2,x_2^2)$ & 1, 2, 2, 1 &\\
\hline
27& $\KK[x_1,x_2,x_3]/(x_1^2,x_2^2,x_3^2,x_1x_2-x_1x_3)$ & 1, 3, 2 &\\
\hline
28& $\KK[x_1,x_2,x_3]/(x_2^2,x_3^2,x_1x_2,x_1^2-x_2x_3)$ & 1, 3, 2 &\\
\hline
29& $\KK[x_1,x_2,x_3]/(x_1^2,x_2^2,x_3^2,x_2x_3)$ & 1, 3, 2 &\\
\hline
30& $\KK[x_1,x_2,x_3]/(x_1^2,x_2^2,x_1x_3,x_2x_3,x_1x_2-x_3^3)$ & 1, 3, 1, 1 & G\\
\hline
31& $\KK[x_1,x_2,x_3]/(x_1^2-x_3^3,x_2^2,x_1x_2,x_1x_3,x_2x_3)$ & 1, 3, 1, 1 &\\
\hline
32& $\KK[x_1,x_2,x_3]/(x_1^3,x_2^2,x_3^2,x_1x_2,x_1x_3)$ & 1, 3, 2 &\\
\hline
33& $\KK[x_1,x_2,x_3]/(x_1^2,x_2^2,x_3^2,x_1x_2-x_1x_3-x_2x_3)$ & 1, 3, 2 &\\
\hline
34& $\KK[x_1,x_2,x_3]/(x_1^3,x_2^2,x_1x_3,x_2x_3,x_1x_2-x_3^2)$ & 1, 3, 2 &\\
\hline
35& $\KK[x_1,x_2,x_3]/(x_1^4,x_2^2,x_3^2,x_1x_2,x_1x_3,x_2x_3)$ & 1, 3, 1, 1 &\\
\hline
36& $\KK[x_1,x_2,x_3]/(x_1^3,x_2^3,x_3^2,x_1x_2,x_1x_3,x_2x_3)$ & 1, 3, 2 &\\
\hline
37& $\KK[x_1,x_2,x_3]/(x_1^3,x_2^2,x_3^2,x_1^2x_2,x_1x_3,x_2x_3)$ & 1, 3, 2 &\\
\hline
38& $\KK[x_1,x_2,x_3,x_4]/(x_i^2-x_j^2, x_ix_j, i\ne j)$ & 1, 4, 1 & G\\
\hline
39& $\KK[x_1,x_2,x_3,x_4]/(x_1^2,x_2^2,x_4^2,x_1x_3,x_1x_4,x_2x_3,x_2x_4,x_3x_4, x_1x_2-x_3^2)$ & 1, 4, 1 &\\
\hline
40& $\KK[x_1,x_2,x_3,x_4]/(x_i^2,x_1x_3,x_1x_4,x_2x_3,x_2x_4,x_3x_4)$ & 1, 4, 1 &\\
\hline
41& $\KK[x_1,x_2,x_3,x_4]/(x_1^3,x_2^2,x_3^2,x_4^2,x_ix_j,i\ne j)$ & 1, 4, 1 &\\
\hline
42& $\KK[x_1,x_2,x_3,x_4,x_5]/(x_i^2,x_ix_j)$ & 1, 5 &\\
\hline
\caption{Local algebras of dimension at most $6$ \parbox[b][15pt]{0pt}{}}
\label{table_localg6}
\end{longtable}

There are many classification results on Gorenstein local algebras, see e.g. \cite{Cas2010, EV2011, Jel2017}. In general, local algebras and their Hilbert-Samuel sequences are studied intensively in connection with punctual Hilbert schemes and collections of commuting nilpotent matrices, see e.g. \cite{Iar1977,Iar1987, Iar1994, Nak1999, BI2010} and references therein. 


\subsection{Suprunenko-Tyshkevich classification}
\label{SuTy_subsect}

In this subsection we present and discuss some results of book~\cite{SuTy1968}. This monograph deals with collections of commuting matrices in the matrix algebra $\Mat_m(\KK)$. Our goal is to demonstrate applications of these results to the study of abstract commutative algebras and groups. In particular, a classification of maximal commutative nilpotent subalgebras of $\Mat_m(\KK)$ for $m \le 6$ leads to the classification of local algebras of dimension at most $6$, see Theorem~\ref{localg6_prop}. 

Let us start with a short historical overview. There are an immeasurably large number of results and publications on maximal commutative subalgebras and subgroups in various contexts and under various constraints. The earliest one is the paper by Frobenius~\cite{Fr1896}. In the period c.1920-35, Kravchuk studied a canonical form of maximal commutative subalgebras called the Kravchuk Normal Form by the authors of~\cite{SuTy1968} and obtained many results on criteria for conjugacy via this form, see~\cite[Sections~2.5 and 2.6]{SuTy1968}. 

For the dimension function of a commutative subalgebras of~$\Mat_m(\KK)$, it dates back to Schur's work~\cite{Sc1905}, where the upper bound $[\frac m4]^2 + 1$ for the field $\KK = \CC$ was established. Jacobson~\cite{Ja1944} extended this result to an arbitrary field. In~\cite{Ge1961}, Gerstenhaber proved that the dimension of the algebra generated by two commuting matrices in~$\Mat_m(\KK)$ is at most~$m$, see also~\cite{BaHa1990, Wa1990, LaLa1991} for other proofs of this fact and more discussion. In~\cite{LaLa1991, GuSe2000}, the dimension bounds for algebras generated by a pair and a triple of elements were studied. The dual problem on minimal dimension was discussed in~\cite{Co1965, La1985}. It occurs that there are maximal commutative subalgebras of~$\Mat_m(\KK)$ of dimension smaller than~$m - 1$. Various constructions of maximal commutative subalgebras of~$\Mat_m(\KK)$ can be found in~\cite{BrCa1993, Br1997, So2003}. 

As was observed by Handelman~\cite{Ha1973}, relations between maximal commutative subalgebras and maximal commutative subgroups were established for the first time by Charles~\cite{Ch11953, Ch21953, Ch1954}. Such relations were studied systematically in~\cite{SuTy1968}. Let us present the corresponding results.

\smallskip

As above, all algebras are supposed to be finite-dimensional, commutative, and associative. If an algebra is not said to be nilpotent we also suppose that it has a unit. All results are formulated over an algebraically closed field $\KK$ of characteristic zero. 

\medskip

\emph{1) Local algebras and indecomposable subalgebras.} Let us introduce some notation. A~set $A$ of elements in $\Mat_m(\KK)$ is called \emph{decomposable} if $\KK^m$ is the direct sum of proper subspaces that are invariant under the tautological action of $A$ on $\KK^m$; otherwise $A$ is called \emph{indecomposable}. 

In \cite[Section~2.2]{SuTy1968} (see Theorem~2.2 and the text below) it is proved that any maximal commutative subalgebra of $\Mat_m(\KK)$ is the direct sum of indecomposable maximal commutative subalgebras of $\Mat_{m_i}(\KK)$ for some $m_1 + \ldots + m_r = m$.

An algebra $A$ is an indecomposable maximal commutative subalgebra of $\Mat_m(\KK)$ if and only if $A = \KK \oplus \mm$, where $\KK$ is the subalgebra of scalar matrices and $\mm$ is a
maximal commutative nilpotent subalgebra of $\Mat_m(\KK)$, see \cite[Theorems~2.3,~2.4]{SuTy1968}. Together with Lemma~\ref{localg_lem}, it implies that the set of indecomposable maximal commutative subalgebras of $\Mat_m(\KK)$ coincides with the set of local maximal commutative subalgebras of $\Mat_m(\KK)$. 

\medskip

\emph{2) Classification of nilpotent subalgebras.} In Section~3.3 of~\cite{SuTy1968}, the classification of maximal commutative nilpotent subalgebras of the algebra $\Mat_m(\KK)$ for $m \le 6$ up to conjugation is given. The number of conjugacy classes of such subalgebras is the following: 
\[\begin{array}{c|c|c|c|c|c|c|c}
m & 1 & 2 & 3 & 4 & 5 & 6 & \ge7\\\hline
& 1 & 1 & 3 & 7 & 18 & 57 & \infty
\end{array}\]
For a nilpotent algebra $\mm$, denote by $l$ its index of nilpotency, i.e. $\mm^l = 0$ and $\mm^{l-1} \ne 0$. The classification is derived from the following cases: a classification of maximal commutative nilpotent subalgebras of $\Mat_m(\KK)$ with 
$l = 2$ (Section~2.3, Theorem~2.7), 
$l = m$ (Section~2.4, Theorem~2.8), 
$l = m - 1$ (Section~3.1, Theorem~3.1),
$l = m - 2$ (Section~3.2, Theorem~3.2) for an arbitrary $m$ and commutative nilpotent algebras of dimension $5$ with $l = 3$ (Section~2.9, Theorem~2.18 and Section~3.3) up to conjugation. 

\medskip

\emph{3) Regular subgroups and subalgebras.} 
Let us call a commutative subgroup $G \subseteq \GL_{n+1}(\KK)$ \emph{regular} if the tautological action of $G$ on $\KK^{n+1}$ has an open orbit, i.e. there exists $v \in \KK^{n+1}$ with the open orbit $Gv \subseteq \KK^{n+1}$. A commutative subalgebra $A \subseteq \Mat_{n+1}(\KK)$ is \emph{regular} if there is a cyclic vector $v \in \KK^{n+1}$, that is $Av = \KK^{n+1}$. A commutative nilpotent subalgebra $\mm \subseteq \Mat_{n+1}(\KK)$ is called \emph{regular} if there is a vector $v \in \KK^{n+1}$ with $\dim \mm v = n$; in this case, we call such vector $v$ cyclic as well. 

\smallskip

\begin{lemma}
\label{openorb_lem}
Let $G$ be a commutative algebraic group acting effectively on an irreducible algebraic variety $X$ with an open orbit. Then $G$ is connected and $\dim G = \dim X$. 
\end{lemma}

\begin{proof}
Let $Gx_0 \subseteq X$ be an open orbit. Since $G$ is commutative, the stabilizers of all the points in $Gx_0$ coincide. Any element of $G$ that acts trivially on $Gx_0$ acts trivially on $X$ as well. Therefore, by the effectivity of the action, the stabilizer of~$x_0$ is trivial, and the mapping $G \hookrightarrow X$ defined by $g \mapsto gx_0$ is an equivariant open embedding. This implies the assertion. 
\end{proof}

\begin{lemma}
Every regular subgroup $G \subseteq \GL_{n+1}(\KK)$ / regular subalgebra $A \subseteq \Mat_{n+1}(\KK)$ / regular nilpotent subalgebra $\mm \subseteq \Mat_{n+1}(\KK)$ is maximal among commutative subgroups of $\GL_{n+1}(\KK)$ / commutative subalgebras of $\Mat_{n+1}(\KK)$ /commutative nilpotent subalgebras of $\Mat_{n+1}(\KK)$. Moreover, $G$ is connected, $\dim G = \dim A = n+1$, and $\dim \mm = n$. 
\end{lemma}

\begin{proof}
From Lemma~\ref{openorb_lem} applied to the tautological action of $G$ on $\KK^{n+1}$ we conclude that $G$ is connected and has dimension $n+1$. 
Any commutative subgroup $\widetilde G$ with $\widetilde G \supseteq G$ is regular as well, whence $G$ and $\widetilde G$ are two connected algebraic groups of the same dimension $n+1$ and $\widetilde G = G$. This implies maximality. 

If $A$ is a regular subalgebra of $\Mat_{n+1}(\KK)$ with a cyclic vector $v$, then the map $A \to \KK^{n+1}$, $a \mapsto av$, is a surjection. Any $a \in A$ in the kernel of this map equals zero since $a\KK^{n+1} = aAv = Aav = 0$ holds. Thus, $A$ is isomorphic to $\KK^{n+1}$ as a vector space. The maximality can be proved as above. 

For a regular nilpotent subalgebra $\mm \subseteq \Mat_{n+1}(\KK)$, consider the direct sum $\KK \oplus \mm$ with the subspace of scalar matrices. It is a regular unital subalgebra. Indeed, let $\dim \mm v = n$ for some $v \in \KK^{n+1}$; then $\dim (\KK + \mm)v = n + 1$ since $v \notin \mm v$ by nilpotency of $\mm$. 
\end{proof}

\medskip

\emph{4) Regular representations.} 
Let us discuss a connection between abstract commutative algebras and commutative subalgebras of $\Mat_{n+1}(\KK)$. 
Any algebra $A$ of dimension $n + 1$ has the \emph{regular representation} $R\colon A \to \End(A)$ defined by the operators of multiplication. Different identifications $\phi\colon A \overset{{}_\sim}{\to} \KK^{n+1}$ give conjugate subalgebras $R'(A)$ of $\Mat_{n+1}(\KK)$, see the diagram below. We say that a subalgebra $A$ comes from the regular representation if $A = R'(A)$ for some identification $A \cong \KK^{n+1}$. 
\[
\begin{diagram}
\node{A} \arrow{e,t}{R(a)} \arrow{s,lr}{\phi}{\wr} \node{A} \arrow{s,lr}{\phi}{\wr} \\
\node{\KK^{n+1}} \arrow{e,t}{R'(a)} \node{\KK^{n+1}}
\end{diagram} \quad\quad\quad
\begin{diagram}
\node{1} \arrow{e,t}{R(A)} \arrow{s,lr}{\phi}{\wr} \node{A} \arrow{s,lr}{\phi}{\wr} \\
\node{\phi(1)} \arrow{e,t}{R'(A)} \node{\KK^{n+1}}
\end{diagram}
\]
The regular representation of an algebra $A$ is faithful provided $A$ has a unit. If $\mm$ is a nilpotent algebra of dimension $n$, we can add an element $e$ and construct a unital algebra $A = \KK e \oplus \mm$ of dimension $n+1$ defined by relations $e^2 = e$ and $ae=ea=a$ for any $a \in \mm$. The regular representation of $A$ induces a faithful respesentation of $\mm$ in $\Mat_{n+1}(\KK)$, which is called \emph{regular} as well. 

\begin{lemma}
\label{algsubalg_lem}
A commutative subalgebra / commutative nilpotent subalgebra of $\Mat_{n+1}(\KK)$ comes from the regular representation if and only if it is a regular subalgebra / regular nilpotent subalgebra. In particular, there is a bijection between isomorphism classes of commutative algebras of dimension~$n+1$ / commutative nilpotent algebras of dimension~$n$ and conjugacy classes of regular subalgebras / regular nilpotent subalgebras of $\Mat_{n+1}(\KK)$. 
\end{lemma}

\begin{proof}
First consider unital algebras. Any subalgebra $R'(A)$ of $\Mat_{n+1}(\KK)$ coming from the regular representation is regular with a cyclic vector $v = \phi(1)$ since $R'(A)\phi(1) = \phi(A)$; see the diagrams above. Conversely, if $A$ is a regular subalgebra with $Av = \KK^{n+1}$, $v \in \KK^{n+1}$, then $A$ comes from its regular representation via the identification $\phi(a)= av$. 

Let a nilpotent subalgebra $R'(\mm)$ come from the regular representation. Then ${\KK \oplus R'(\mm)}$ is a regular subalgebra of $\Mat_{n+1}(\KK)$, and $R'(\mm)$ is regular with the same cyclic vector $v = \phi(1)$ since $R'(\mm)\phi(1) = \phi(\mm)$. Conversely, if $\mm \subseteq \Mat_{n+1}(\KK)$ is a regular nilpotent subalgebra, then $A = \KK \oplus \mm$ is a regular subalgebra and comes from its regular representation by the arguments given above. 
\end{proof}

\medskip

\emph{5) Classification results on abstract algebras.} According to the above, the classification of local algebras of dimension $n+1$ is equivalent to the classification of images of the regular representations of their maximal nilpotent ideals, i.e. regular nilpotent subalgebras of $\Mat_{n+1}(\KK)$. Thus if we want to get a classification of local algebras of dimension at most~$6$ up to isomorphism, we have to choose those subalgebras from the list of \cite[Section~3.3]{SuTy1968} (see 2)) which are regular. 
Moreover, Theorem~2.15 says that a maximal commutative nilpotent subalgebra of $\Mat_{n+1}(\KK)$ is regular if and only if its so called first Kravchuk number $\nu = n + 1 - \dim \mm\KK^{n+1}$ equals one, i.e. $\dim \mm\KK^{n+1} = n$. 
Thus Table~\ref{table_localg6} can be obtained from results of \cite[Section~3.3]{SuTy1968}. 

\begin{example}
Consider $n+1 = 4$. By the classification of \cite[Section~3.3]{SuTy1968}, there are $7$ maximal commutative nilpotent subalgebras of $\Mat_4(\KK)$, see 2): 
\[\begin{aligned}
l &= 2: \quad (1)
\left\{\begin{pmatrix}
0&0&0&0\\
a&0&0&0\\
b&0&0&0\\
c&0&0&0
\end{pmatrix}\right\}, \quad (2)
\left\{\begin{pmatrix}
0&0&0&0\\
0&0&0&0\\
a&b&0&0\\
c&d&0&0
\end{pmatrix}\right\}, \quad (3)
\left\{\begin{pmatrix}
0&0&0&0\\
0&0&0&0\\
0&0&0&0\\
c&b&a&0
\end{pmatrix}\right\}
\\ l &= 3: \quad (4)
\left\{\begin{pmatrix}
0&0&0&0\\
a&0&0&0\\
b&a&0&0\\
c&0&0&0
\end{pmatrix}\right\}, \quad (5)
\left\{\begin{pmatrix}
0&0&0&0\\
a&0&0&0\\
b&a&0&c\\
0&0&0&0
\end{pmatrix}\right\}, \quad (6)
\left\{\begin{pmatrix}
0&0&0&0\\
a&0&0&0\\
b&a&0&c\\
c&0&0&0
\end{pmatrix}\right\}
\\l &= 4: \quad (7)
\left\{\begin{pmatrix}
0&0&0&0\\
a&0&0&0\\
b&a&0&0\\
c&b&a&0
\end{pmatrix}\right\}, \quad\quad\quad a, b, c, d \in \KK. 
\end{aligned}\]
Subalgebras (1), (4), (6), and (7) are regular with cyclic vector $v = (1, 0, 0, 0)$. They correspond to four commutative algebras no.~8, 7, 6, and 5 of dimension $4$ from Table~\ref{table_localg6}. For subalgebras (2), (3), and (5) the first Kravchuk number equals 2, 3, 2 respectively, so they are not regular. 
\end{example}

\medskip

\emph{6) Infinite series.} While there is a finite number of nilpotent algebras of dimension $n$ and index of nilpotency $2$, $n-2$, $n-1, n$, there exist infinitely many non-isomorphic nilpotent algebras of dimension $6$ and index of nilpotency~$3$. It follows that there is an infinite number of local algebras of dimension at least $7$. More precisely, consider the algebras with Hilbert-Samuel sequence $(1, 4, 2)$. Since the index of nilpotency of the maximal ideal of such an algebra equals~$3$, the multiplication is determined by a bilinear symmetric map $\mm/\mm^2 \times \mm/\mm^2 \to \mm^2$. 
We have $\dim \mm/\mm^2 = 4$ and $\dim \mm^2 = 2$, so such maps form a space of dimension $20 = 4(4+1)/2 \cdot 2$. 
An isomorphism between such algebras corresponds to a change of coordinates in $\mm/\mm^2$ and $\mm^2$, i.e. we consider the maps up to the action of the group $\GL(4) \times \GL(2)$. It has dimension $20 = 4^2 + 2^2$ and acts on the space of such maps with a one-dimensional inefficiency kernel. Since $19 < 20$, it follows that there are infinitely many generic pairwise non-isomorphic algebras of this type. A disscussion about algebras of similar type can be found in \cite[Example~3.6]{HaTs1999} and the text before and after it. For more information on Hilbert-Samuel sequences corresponding to infinitely many non-isomorphic local algebras, see~\cite{Lo2017}. 

Let us give an explicit example. For $n = 7$, consider the algebras $A_\alpha$ of the form 
\[A_\alpha = \KK[x_1, x_2, x_3, x_4]\,/\,(x_1^2+x_3^2 - 2x_2^2, x_4^2 - x_2^2 - \alpha(x_3^2-x_2^2), x_ix_j, i \ne j).\]
It is shown in \cite[Section~2.8]{SuTy1968} that in any isomorphy class of algebras of the form $A_\alpha$ there is a finite number of algebras. For $n > 7$, we can add variables $x_5, \ldots, x_{n-3}$ to the algebra $A_\alpha$ with $x_ix_k = 0$ for any $1 \le i \le n-3$, $5 \le k \le n-3$ and obtain an infinite series of pairwise non-isomorphic algebras of dimension $n$. 


\subsection{Knop-Lange theorem}
\label{subsecKLT}

In this section we study actions of arbitrary connected commutative linear algebraic groups on projective spaces with an open orbit. It is well known that such a group $G$ is isomorphic to $\GG_m^r \times \GG_a^s$ for some $r, s \in \Zgezero$, see~\cite[Theorem~15.5]{Hu1975}. The numbers $r$ and $s$ are called the \emph{rank} and the \emph{corank} of $G$, respectively. 

\begin{definition}
\label{act_equivdef}
Actions $\alpha_i\colon G_i \times X_i \to X_i$ of algebraic groups~$G_i$ on algebraic varieties~$X_i$, $i=1,2$, are said to be \emph{equivalent} if there are a group isomorphism $\psi\colon G_1 \to G_2$ and a variety isomorphism $\phi\colon X_1 \to X_2$ such that $\phi\circ\alpha_1= \alpha_2\circ(\psi\times\phi)$.
\end{definition}

The following theorem is proved in~\cite[Proposition 5.1]{KnLa1984}.

\begin{theorem}
\label{KnopLange_theor}
There is a bijection between the following objects: 
\smallskip
\begin{enumerate}
	\item effective actions of connected commutative linear algebraic groups $G$ on $\PP^n$ with an open orbit;
	\item commutative associative unital algebras $A$ of dimension $n+1$.
\smallskip
\end{enumerate}
The bijection is considered up to equivalence of actions and algebra isomorphisms. Moreover, if $G$ is of rank $r$ then $A$ contains exactly $r+1$ maximal ideals. 
The number of isomorphism classes is given in Table~2. 
\end{theorem}

\begin{table}[ht]
$\begin{array}{c|c|c|c|c|c|c|c|c|c|c|c|c}
\dim A& 1 & 2 & 3 & 4 & 5 & 6 &     7   &    8   &    9   &    10  &    11  & 12 \\ \hline 
r = 0 & 1 & 1 & 2 & 4 & 9 & 25 & \infty & \infty & \infty & \infty & \infty & \dots \\
r = 1 &   & 1 & 1 & 3 & 6 & 16 &   42   & \infty & \infty & \infty & \infty & \dots \\
r = 2 &   &   & 1 & 1 & 3 & 7  &   18   &   49   & \infty & \infty & \infty & \dots \\
r = 3 &   &   &   & 1 & 1 & 3  &    7   &   19   &   51   & \infty & \infty & \dots \\
r = 4 &   &   &   &   & 1 & 1  &    3   &    7   &   19   &    52  & \infty & \dots \\
r = 5 &   &   &   &   &   & 1  &    1   &    3   &   7    &    19  &   52   & \dots \\ 
\cdots&   &   &   &   &   &    & \ddots & \ddots & \ddots & \ddots & \ddots & \ddots\\\hline
\text{total} & 1 & 2 & 4 & 9 & 20 & 53  & \infty & \infty & \infty & \infty & \infty & \dots
\end{array}$
\caption{The number of algebras of small dimension \parbox[b][15pt]{0pt}{}}
\end{table}

\begin{proof} 
$(b) \to (a)$ The group of invertible elements $A^\times$ of the algebra $A$ is a connected commutative linear algebraic group that is open in $A$. The factor group ${G =\PP(A^\times):= A^\times / \,\KK^\times}$ by the subgroup of invertible scalars $\KK^\times \cdot 1$ is a connected commutative linear algebraic group. It acts in a canonical way on $\PP(A) = \PP^n$ with an open orbit isomorphic to $\PP(A^\times)$. 

\emph{Equivalence.} An algebra isomorphism $\phi\colon A_1 \to A_2$ induces a group isomorphism $\PP(A_1^\times) \to \PP(A_2^\times)$ and a variety isomorphism $\PP(A_1) \to \PP(A_2)$. They determine the equivalence between the actions of $\PP(A_i^\times)$ on~$\PP(A_i)$, $i = 1,2$. 

\smallskip
$(a) \to (b)$ Lemma~\ref{openorb_lem} implies $\dim G = n$. Since $G$ acts on $\PP^n$ effectively, we can consider $G$ as a subgroup of $\Aut(\PP^n) = \PGL_{n+1}(\KK)$. 

Denote by $\pi\colon \GL_{n+1}(\KK) \to \PGL_{n+1}(\KK)$ the canonical projection and let $H := \pi^{-1}(G)$. 
Let us prove that $H$ is a connected commutative linear algebraic group of dimension $n+1$. First note that $H$ contains the group $\KK^\times$ of invertible scalar matrices since $G \ni 1$. Then $\dim H = \dim G + \dim \Ker\pi\bigr|_H = n + 1$. 
Further, $H$ is connected as $\pi(H) = G$ and $\Ker \pi\bigr|_H$ are connected. 
Finally, let us prove that $H$ is commutative. Consider the commutant $[H, H]$ of the group $H$. Since $G$ is commutative, we have $[H, H] \subseteq \Ker\pi\bigr|_H = \KK^\times$. On the other hand, $[H, H]$ is connected as the commutant of a connected group, whence $[H, H] = \{1\}$ or $[H, H] = \KK^\times$. The latter is impossible since the commutant consists of matrices with determinant~$1$. It follows that $[H, H]$ is trivial and $H$ is commutative. 

Consider $\GL_{n+1}(\KK)$ as an open subset of $\Mat_{n+1}(\KK)$ and denote by $A$ the associative subalgebra of $\Mat_{n+1}(\KK)$ generated by $H$. Clearly, $A$ is a commutative unital algebra. Let us prove that $\dim A = n+1$. 

Note that the tautological action of $H \subseteq \GL_{n+1}(\KK)$ on $\KK^{n+1}$ has an open orbit. 
The group of invertible elements $A^\times \subseteq \GL_{n+1}(\KK)$ is open in~$A$. It is commutative, acts effectively on~$\KK^{n+1}$, and the action has an open orbit since the action of $H \subseteq A^\times$ has. By Lemma~\ref{openorb_lem} we obtain $\dim A^\times = \dim \KK^{n+1} = n+1$, whence $\dim A = n+1$. Moreover, $H = A^\times$ since $H$ is an algebraic subgroup of~$A^\times$ of the same dimension. 

\emph{Equivalence.} 
Let $\psi\colon G_1 \to G_2$ and $\phi\colon \PP^n \to \PP^n$ determine the equivalence of two actions. 
Since $\phi \in \PGL_n(\KK)$, there is $\Phi\in\GL_{n+1}(\KK)$ that induces $\phi$ on~$\PP(\KK^{n+1})$. The isomorphism of vector spaces $\Phi$ induces an isomorphism of operator algebras $\Psi\colon\Mat_{n+1}(\KK) \to \Mat_{n+1}(\KK)$, $\Psi(X) = \Phi X\Phi^{-1}$. Considering $G_i$ as the subgroups of $\PGL_n(\KK)$ and setting $H_i = \pi^{-1}(G_i)$, $i = 1,2$, we obtain that $\Psi(H_1) = \Phi \pi^{-1}(G_1) \Phi^{-1} = \pi^{-1}(\phi G_1 \phi^{-1}) = \pi^{-1}(G_2) = H_2$, whence $\Psi(A_1) = A_2$ is the desired algebra isomorphism. 

\smallskip

Let us check that two constructed maps are inverse to each other. Let $A$ be an algebra as in $(b)$. Then we have an action of the group $G = A^\times / \,\KK^\times$ on $\PP(A)$ as in $(a)$. We can consider $G$ as a subgroup of $\PGL(A)$. According to $(a) \to (b)$, this action corresponds to the associative subalgebra of $\Mat_{n+1}(\KK)$ generated by $\pi^{-1}(A^\times / \,\KK^\times) = A^\times$, which coincides with $A$. 

Conversely, let $G$ act on $\PP^n$ with an open orbit. We have an algebra $A$ as in $(a) \to (b)$, in particular, $A^\times = H = \pi^{-1}(G)$. Then $A^\times / \,\KK^\times$ coincides with $G$ in $\PGL_{n+1}(\KK)$. 

\smallskip

For the second assertion, note that if $A = \KK \oplus \mm$ is local, its group of invertible elements equals $A^\times = \KK^\times \oplus \mm = \KK^\times \times (1+\mm)$, where $(1+\mm, \times) \cong (\mm, +) \cong \GG_a^n$ via exponential map and $\KK^\times \cong \GG_m$. Since any commutative algebra $A$ is a sum of local algebras by Lemma~\ref{directloc_lem}, it follows that the rank of the group $A^\times$ equals the number of its local summands, which is equal to the number of maximal ideals. By construction, the rank of $A^\times = H$ is one more than the rank of~$G$. 

\smallskip

The number of isomorphism classes of algebras of dimension $n+1$ can be found by direct computations using the number of local algebras of a fixed dimension, which is given in Table~\ref{table_localg6}. More precisely, any algebra of dimension $n+1$ decomposes into a sum of local algebras, and this decomposition is defined by unordered tuples of local algebras of dimensions $m_1, \ldots, m_r$, where $n+1 = m_1 + \ldots + m_r$. 
\end{proof}

\begin{remark} 
In~\cite[Proposition~5.1]{KnLa1984}, the first assertion of Theorem~\ref{KnopLange_theor} is proved for an arbitrary ground field $\KK$. 
\end{remark}

\begin{remark}
Theorem~2.1 of Suprunenko and Tyshkevich~\cite{SuTy1968} states that there is a one-to-one correspondence between maximal commutative subalgebras of $\Mat_{n+1}(\KK)$ and maximal commutative subgroups of $\GL_{n+1}(\KK)$. More precisely, for a subalgebra $A \subseteq \Mat_{n+1}(\KK)$ and a subgroup $H \subseteq \GL_{n+1}(\KK)$ the bijection is defined by $A \mapsto A^\times$ and $\Span H \mapsfrom H$. Let us reformulate the proof of Knop-Lange theorem in these terms. 

It is easy to see that the correspondence of Theorem~2.1 restricts to the bijection between regular subalgebras of $\Mat_{n+1}(\KK)$ and regular subgroups of $\GL_{n+1}(\KK)$. On the one hand, regular subalgebras of $\Mat_{n+1}(\KK)$ correspond to abstract algebras of dimension $n+1$ by Lemma~\ref{algsubalg_lem}. On the other hand, the arguments from the proof of Knop-Lange theorem show that regular subgroups $H \subseteq \GL_{n+1}(\KK)$ are in bijection with commutative subgroups $G \subseteq \PGL_{n+1}(\KK) = \Aut(\PP^n)$ such that the corresponding action of the group $G$ on $\PP^n$ has an open orbit: the correspondence is given by $G = \pi(H)$ and $H = \pi^{-1}(G)$, where $\pi$ is the canonical projection $\pi\colon \GL_{n+1}(\KK) \to \PGL_{n+1}(\KK)$. Thus we obtain the bijection between $G$-actions on $\PP^n$ with an open orbit and algebras of dimension $n+1$. 
\end{remark}

Now we come to a description of orbits of a commutative group acting on $\PP^n$ in terms of the corresponding algebra. 

\begin{corollary}
\label{orb_as_id_cor}
The correspondence of Theorem~\ref{KnopLange_theor} determines a bijection between $G$-orbits on~$\PP^n$ and nonzero principal ideals in the algebra~$A$. 
\end{corollary}

\begin{proof}
Let us establish first a bijection between $G$-orbits on~$\PP^n$ and association classes of nonzero elements in the algebra~$A$. If for $a, b \in A$ there exists $c \in A^\times$ such that $a = cb$, then $[b] \in \PP(A)$ is obtained from $[a] \in \PP(A)$ by the action of $[c] \in A^\times / \,\KK^\times$. Conversely, if $[a] = [c] \cdot [b]$ for $a, b \in A$, $c \in A^\times$, then $a = \lambda cb$, $\lambda \in \KK^\times$, whence $a$ and $b$ are associated.

It remains to notice that a generator of a principal ideal is defined up to association. 
\end{proof}

For the following statement, see~\cite[Proposition 3.5]{HaTs1999}. 
\begin{corollary}
\label{fin_orb_Gan_cor}
There is a unique action of $\GG_a^n$ on $\PP^n$ with finitely many orbits. It corresponds to the truncated polynomial algebra $A=\KK[S]\,/\,(S^{n+1})$. 
\end{corollary}

\begin{proof}
By Corollary~\ref{orb_as_id_cor}, we have to investigate local $(n+1)$-dimensional algebras~$A$ with finite number of principal ideals. First note that the algebra $\KK[S]\,/\,(S^{n+1})$ is local and has finite number of principal ideals $(S^k)$, $0 \le k \le n + 1$. 
Let us prove the converse statement by induction on~$n$. Let $A$ be a local algebra of dimension~$n+1$ with finitely many principal ideals. The set of fixed points in $\PP^n = \PP(A)$ coincides with $\PP(\Soc A)$, so $\dim \Soc A = 1$. Notice that $\Soc A$ is an ideal in~$A$, so we can consider the factor-algebra $A\,/\,\Soc A$. It is $n$-dimensional and has a finite number of principal ideals as well, so by inductive hypothesis it is isomorphic to $\KK[s]\,/\,(s^{n})$. Let $S + \Soc A \in A\,/\,\Soc A$ corresponds to~$s$. Then $A$ is the direct sum of the vector spaces $\Soc A$ and $\langle S^k, \, 0 \le k \le n-1\rangle$. Moreover, it follows that $S^n \in \Soc A$, whence $S^{n+1} = 0$. If $S^n = 0$, then $S^{n-1}\cdot S = 0$ and $S^{n-1}\Soc A = 0$ imply $S^{n-1}\mm = 0$, a contradiction with $S^{n-1} \notin \Soc A$. Thus $A = \langle S^k, \, 0 \le k \le n\rangle$. 
\end{proof}

For positive integers $n$ and $r$, we denote by $p_r(n)$ the number of partitions $n = n_1 + \ldots + n_r$ with $n_1 \ge \ldots \ge n_r \ge 1$. 

\begin{corollary}
Let $G$ be a connected commutative linear algebraic group of dimension $n$ and rank $r$. Then there exist precisely $p_r(n)$ effective actions of $G$ on $\PP^n$ with finite number of orbits. The corresponding algebras $A$ are precisely the algebras of the form $\KK[S]\,/\,(f(S))$, where $f(S)$ is a polynomial of degree $n$ with precisely $r$ distinct roots. 
\end{corollary}

\begin{proof}
By Corollary~\ref{orb_as_id_cor}, the number of $G$-orbits in~$\PP^n$ is equal to the number of principal ideals in the corresponding algebra $A$. Let $A = A_1 \oplus \ldots \oplus A_n$ be the decomposition into the sum of local ideals, see Lemma~\ref{directloc_lem}. Principal ideals in~$A$ are precisely the sums of principal ideals in~$A_i$, so the number of principal ideals in $A$ is finite if and only if it is finite for every local summand. By Corollary~\ref{fin_orb_Gan_cor}, this holds if and only if every $A_i$ is isomorphic to $\KK[S]\,/\,(S^{n_i})$, where $n_i = \dim A_i$. Hence the algebra $A$ is of the required form and is uniquely determined by dimensions $n_1,\ldots, n_r$. 
\end{proof}

\begin{example}
Consider the algebra $A = \KK^{n+1}$ with the coordinate-wise multiplication. 
Then $A^\times = (\KK^\times)^{n+1}$, and the group $A^\times / \,\KK^\times$ is isomorphic to $\GG_m^n$: an element ${(t_1, \ldots, t_n) \in \GG_m^n}$ corresponds to the class of $(1, t_1, \ldots, t_n) \in A^\times$ and acts via multiplication on the classes of elements $(z_0, \ldots, z_n) \in A$: 
\[(t_1, \ldots, t_n) \cdot [z_0 : z_1 : \ldots : z_n] = [z_0 : t_1z_1 : \ldots t_nz_n].\]
It is an action of $\GG_m^n$ on $\PP^n$ with an open orbit $\{z_i \ne 0, \; 0 \le i \le n\}$. The other orbits are parameterized by the set of indices $0 \le i \le n$ such that $z_i = 0$, so there are $2^{n+1} - 1$ orbits for this action. 
\end{example}

\begin{example}
\label{KnL_S1S2_ex}
Consider the local algebra $A = \KK[S_1, S_2]\,/\,(S_1^2, S_1S_2, S_2^2)$, $\mm = \langle S_1, S_2\rangle$. Let us find the corresponding action of $A^\times / \,\KK^\times$ on $\PP(A)$. 

Since $A^\times / \,\KK^\times = (1 + \mm, \times) \cong (\mm, +) \cong \GG_a^2$ via exponential map, the action of an element $(x_1, x_2) \in \GG_a^2$ is given by the multiplication by the class of $\exp(x_1S_1+x_2S_2) \in A^\times$. Applying this to $[z_0:z_1:z_2] \in \PP^2$ identified with the class of $z_0 + z_1S_1+z_2S_2 \in A$, we obtain
\begin{multline*}
(x_1, x_2)\cdot [z_0:z_1:z_2] = \exp(x_1S_1+x_2S_2)(z_0 + z_1S_1+z_2S_2) = \\ = (1+x_1S_1+x_2S_2)(z_0 + z_1S_1+z_2S_2) = z_0 + (z_1+x_1z_0)S_1 + (z_2+x_2z_0)S_2 = \\ = [z_0 : z_1+x_1z_0 : z_2+x_2z_0]. 
\end{multline*}
It is an action of $\GG_a^2$ on $\PP^2$ with an open orbit $\{z_0 \ne 0\}$. The other orbits are the fixed points, which form the line $\{z_0 = 0\}$, so there are infinitely many orbits in this case. 
\end{example}

\begin{example}
\label{KnL_S3_ex}
Consider the remaining local algebra of dimension~3: $A = \KK[S]\,/\,(S^3)$ with $\mm = \langle S, S^2\rangle$. As above, the action of $(x_1, x_2) \in \GG_a^2$ on $[z_0:z_1:z_2] \in \PP^2$ is given by
\begin{multline*}
(x_1, x_2)\cdot [z_0:z_1:z_2] = \exp(x_1S+x_2S^2)(z_0 + z_1S+z_2S^2) = \\ = \left(1+x_1S+\Bigl(x_2+\frac{x_1^2}{2}\Bigr)S^2\right)(z_0 + z_1S+z_2S^2) = z_0 + (z_1+x_1z_0)S + \\ + \left(z_2+x_1z_1 + \Bigl(x_2+\frac{x_1^2}{2}\Bigr)z_0\right)S^2 = \Bigl[z_0 : z_1+x_1z_0 : z_2+x_1z_1+\Bigl(x_2+\frac{x_1^2}{2}\Bigr)z_0\Bigr]. 
\end{multline*}
It is an action of $\GG_a^2$ on $\PP^2$ with an open orbit $\{z_0 \ne 0\}$. The other orbits are $\{z_0 = 0, z_1 \ne 0\}$ and $\{z_0 = z_1 = 0\}$, so there are $3$ orbits for this action. 
\end{example}


\subsection{Polynomials and differential operators}
\label{IVsubsection}
We begin with some auxiliary definitions and bijections required for Hassett-Tschinkel correspondence. Similar results are explained in~\cite{IaKa1999} with the reference to~\cite{Ma1916}. Let $\KK$ be a field of characteristic zero. Fix $n \in \NN$ and consider two polynomial algebras $\KK[x_1, \ldots, x_n]$ and $\KK[S_1, \ldots, S_n]$. If we identify $S_i$ with $\frac{\pa}{\pa x_i}$, $1 \le i \le n$, then $\KK[S_1, \ldots, S_n]$ can be considered as the polynomial algebra $\KK[\frac{\pa}{\pa x_1}, \ldots, \frac{\pa}{\pa x_n}]$ of differential operators with constant coefficients. 
\begin{construction}
\label{IVconstr}
Consider the pairing between $\KK[x_1, \ldots, x_n]$ and $\KK[S_1, \ldots, S_n]$:
\begin{equation}
\label{pairing}
\KK[S_1, \ldots, S_n] \times \KK[x_1, \ldots, x_n] \,\to \,\KK, \quad
(g, f) \; \mapsto \; g[f]\bigr|_{(0,\ldots,0)} =: \langle g \mid f \rangle.
\end{equation}
In particular, $\langle S_1^{i_1}\ldots S_n^{i_n} \mid x_1^{j_1}\ldots x_n^{j_n}\rangle$ equals $i_1!\ldots i_n!$ if $i_k = j_k$, $1 \le k \le n$, and $0$ otherwise. 
The pairing is non-degenerate:
\begin{itemize}
\item $f \in \KK[x_1, \ldots, x_n]$ with $\langle g \mid f \rangle = 0 \quad \forall g \in \KK[S_1, \ldots, S_n]$ implies $f = 0$;
\item $g \in \KK[S_1, \ldots, S_n]$ with $\langle g \mid f \rangle = 0 \quad \forall f \in \KK[x_1, \ldots, x_n]$ implies $g = 0$.
\end{itemize}
Moreover, it induces the perfect pairing $\KK[S_1, \ldots, S_n]_{\le d} \times \KK[x_1, \ldots, x_n]_{\le d} \to \KK$ between polynomials and differential operators of total degree at most $d$ since these vector spaces are of finite dimension and the restriction of the pairing is non-degenerate as well. 

For a subspace $V \subseteq \KK[x_1, \ldots, x_n]$, one can define the subspace 
\[I_V = \{g \in \KK[S_1, \ldots, S_n] \colon \langle g \mid f \rangle = 0 \;\, \forall f\in V\},\]
and for a subspace $I \subseteq \KK[S_1, \ldots, S_n]$ one can consider
\[V_I = \{f \in \KK[x_1, \ldots, x_n] \colon \langle g \mid f \rangle = 0 \;\, \forall g \in I\}.\]
\end{construction}

\begin{example}
Let $V = \langle x_1^2\rangle \subseteq \KK[x_1]$. Then $I_V$ consists of elements $g = \sum_{i \ge 0}\alpha_iS_1^i$ with $\langle g \mid x_1^2\rangle = 2!\alpha_2 = 0$, i.e. $I_V = \langle S_1^i, \; i \ne 2\rangle$. Conversely, for $I = \langle S_1^i, \; i \ne 2\rangle \subseteq \KK[S_1]$ we obtain $V_I = \langle x_1^2\rangle$ since any $f = \sum_{i \ge 0}\alpha_ix_1^i \in V_I$ satisfies $\langle S_1^i \mid f \rangle = i!\alpha_i = 0$ for all $i \ne 2$. 
\end{example}

\begin{example}
Consider the ideal $I = (S_1^2 - 1) \subseteq \KK[S_1]$, i.e. $I = \langle S_1^{i+2} - S_1^{i}, \; i \ge 0\rangle$. Any $f = \sum_{i \ge 0}\alpha_ix_1^i \in V_I$ satisfies $\langle S_1^{i+2} - S_1^i \mid f \rangle = (i+2)!\alpha_{i+2} - i!\alpha_i = 0$ for all $i \ge 0$. Then
\[\begin{aligned}
0!\alpha_0 = 2!\alpha_2 = 4!\alpha_4 = \ldots,\\
1!\alpha_1 = 3!\alpha_3 = 5!\alpha_5 = \ldots,
\end{aligned}\]
whence $f = 0$ since it can not contain infinitely many nonzero coefficients. Thus $V_I = \{0\}$. It follows that the correspondences of Construction~\ref{IVconstr} between subspaces in $\KK[x_1, \ldots, x_n]$ and $\KK[S_1, \ldots, S_n]$ are not bijective. 
\end{example}

\begin{lemma}
\label{IVbij}
For fixed $d,m\in \Zgezero$, Construction~\ref{IVconstr} defines a bijection between
\begin{enumerate}
	\item subspaces $V \subseteq \KK[x_1, \ldots, x_n]_{\le d}$ with $\dim V = m$;
	\item subspaces $I \subseteq \KK[S_1, \ldots, S_n]$ with $I \supseteq \KK[S_1, \ldots, S_n]_{>d}$ and 
	$\codim_{\KK[S_1, \ldots, S_n]} I = m$.
\end{enumerate}
\end{lemma}
\begin{proof}
It is easy to see that $I_V \supseteq \KK[S_1, \ldots, S_n]_{>d}$. Note that $\dim V = \codim_{\KK[S_1, \ldots, S_n]} I_V$ since the pairing between $\KK[x_1, \ldots, x_n]_{\le d}$ and $\KK[S_1, \ldots, S_n]_{\le d}$ is perfect. 
Since $V \subseteq V_{(I_V)}$ and $\dim V = \codim I_V = \dim V_{(I_V)}$, we obtain $V = V_{(I_V)}$. Analogously $I = I_{(V_I)}$. 
\end{proof}

Now we are going to precise the constructed correspondence in a series of lemmas. The main result of this subsection is formulated in Proposition~\ref{IV}.

\smallskip

Notice that there is a canonical action of the group $\GG_a^n$ on the linear span $\langle x_1, \ldots, x_n\rangle$ by translations. It can be extended to the action of $\GG_a^n$ on $\KK[x_1, \ldots, x_n]$: a group element $\beta = (\beta_1, \ldots, \beta_n) \in \GG_a^n$ maps a polynomial $f(x) = f(x_1, \ldots, x_n)$ to 
$$
f(x+\beta) = f(x_1 + \beta_1, \ldots, x_n + \beta_n).
$$ 

We recall Taylor's theorem:
$f(x + \beta) = \sum\limits_{i_1,\ldots,i_n} \frac{\beta_1^{i_1}\ldots\beta_n^{i_n}}{i_1!\ldots i_n!}\frac{\pa^{i_1+\ldots+i_n}f(x)}{\pa x_1^{i_1}\ldots\pa x_n^{i_n}}$.
It follows that 
$$
f(x + \beta) = \exp(\beta_1S_1 + \ldots \beta_nS_n)[f(x)].
$$ 

\begin{definition}
A subspace $V \subseteq \KK[x_1, \ldots, x_n]$ is called \emph{translation invariant} if the following equivalent conditions hold:
\begin{enumerate}
\item[1)] $V$ is invariant under $S_i = \frac{\pa}{\pa x_i}$ for every $1 \le i \le n$;
\item[2)] $V$ is invariant under the $\GG_a^n$-action by translations.
\end{enumerate} 
\end{definition}

\begin{proof}[Proof of equivalence]
The assertion follows from the fact that the subspace $V$ is $\GG_a^n$-invariant if and only if it is ($\Lie \GG_a^n$)-invariant. 
\end{proof}

\begin{example}
\label{IV_dual_S1S2_ex}
Consider the vector subspace $V = \langle1, x_1, x_2\rangle \subseteq \KK[x_1, x_2]$. It is invariant under $\frac{\pa}{\pa x_1}$ and $\frac{\pa}{\pa x_2}$. On the other hand, it is invariant under translations: 
the corresponding representation of $(\beta_1, \beta_2) \in \GG_a^2$ in $V$ is given by 
$\begin{pmatrix}
1 & \beta_1 & \beta_2 \\
0 & 1 & 0 \\
0 & 0 & 1
\end{pmatrix}$
in the basis $1, x_1, x_2$. 
\end{example}

\begin{example}
\label{IV_dual_S3_ex}
Let $V = \bigl\langle1, x_1, x_2 + \frac{x_1^2}{2}\bigr\rangle \subseteq \KK[x_1, x_2]$. It is translation invariant according to both definitions. Since $(\beta_1, \beta_2) \in \GG_a^2$ applied to basis vectors $1, x_1, x_2 + \frac{x_1^2}{2}$ gives $1, x_1 + \beta_1$ and $x_2 + \beta_2 + \frac{(x_1+\beta_1)^2}{2} = x_2 + \frac{x_1^2}{2} + \beta_1x_1+\beta_2+\frac{\beta_1^2}{2}$ respectively, the corresponding representation of $\GG_a^2$ in $V$ is given by
$\begin{pmatrix}
1 & \beta_1 & \beta_2 + \frac{\beta_1^2}{2} \\
0 & 1 & \beta_1 \\
0 & 0 & 1
\end{pmatrix}$.
\end{example}

\begin{lemma}
\label{IVbij_id_ti}
Lemma~\ref{IVbij} defines a bijection between translation invariant subspaces of $\KK[x_1, \ldots, x_n]$ and ideals in $\KK[S_1, \ldots, S_n]$. Moreover, in this case we have 
\begin{equation}
\label{IVconstr_upd}
\begin{aligned}
V_I = \{f \in V \mid g[f] = 0 \;\; \forall g \in I\}; \\ I_V = \{g \in I \mid g[f] = 0 \;\; \forall f \in V\}. 
\end{aligned}
\end{equation}
\end{lemma}

\begin{proof}
Let $I$ be an ideal and $f \in V_I$, that is $\langle g \mid f\rangle = 0$ for any $g \in I$. Since $\widetilde gg \in I$ for any $\widetilde g \in \KK[S_1, \ldots, S_n]$, it follows that $0 = \langle\widetilde gg \mid f\rangle = \langle\widetilde g \mid g[f]\rangle$, whence by nondegeneracy of $\langle\,\cdot\mid\cdot\,\rangle$ we obtain $g[f] = 0$. 
This implies the first formula in~\eqref{IVconstr_upd}, and thus $V_I$ is $\frac{\pa}{\pa x_i}$-invariant for any $1 \le i \le n$. 

Conversely, let $V$ be a translation invariant subspace and $g \in I_V$. Since $\widetilde g[f] \in V$ for any $\widetilde g \in \KK[S_1, \ldots, S_n]$ and $f \in V$, it follows that $0 = \langle g \mid \widetilde g[f]\rangle = \langle \widetilde g \mid g[f]\rangle$, whence $g[f] = 0$. Then we obtain the second formula in~\eqref{IVconstr_upd}, which implies that $I_V$ is an ideal. 
\end{proof} 

\begin{example}
\label{IV_S13_ex}
The translation invariant vector subspace $V = \langle 1, x_1, x_1^2\rangle \subseteq \KK[x_1]$ corresponds to the ideal $I = (S_1^3) \subseteq \KK[S_1]$. 
\end{example}

\begin{definition}
Let us call a subspace $V \subseteq \KK[x_1, \ldots, x_n]$ \emph{non-degenerate} if no nonzero operator from $\langle S_1, \ldots, S_n\rangle$ annihilates $V$. A subspace $I \subseteq \KK[S_1, \ldots, S_n]$ is called \emph{non-degenerate} if $I \cap \langle S_1, \ldots, S_n\rangle = 0$. 
\end{definition}

The following lemma is straightforward. 

\begin{lemma}
\label{IV_nondeg}
The bijection in Lemma~\ref{IVbij} restricts to a bijection between non-degenerate subspaces in $\KK[x_1, \ldots, x_n]$ and $\KK[S_1, \ldots, S_n]$. 
\end{lemma}

\begin{definition}
Let us call a subspace $V \subseteq \KK[x_1, \ldots, x_n]$ \emph{generating} if one of the following equivalent conditions hold:
\begin{enumerate}
\item[1)] $V$ is translation invariant and non-degenerate;
\item[2)] $V$ is translation invariant and generates $\KK[x_1, \ldots, x_n]$ as an algebra.
\end{enumerate}
\end{definition}

\begin{proof}[Proof of equivalence]
Let $V$ be translation invariant and generate $\KK[x_1, \ldots, x_n]$. There is no nonzero operator from $\langle S_1, \ldots, S_n\rangle$ annihilating $V$ since it would annihilate $\KK[x_1, \ldots, x_n]$ otherwise. 

Conversely, let a translation invariant and non-degenerate subspace $V$ generate a subalgebra $A \subsetneq \KK[x_1, \ldots, x_n]$. Denote $W = A \cap \langle x_1, \ldots, x_n\rangle$. Choosing appropriate variables in $\KK[x_1, \ldots, x_n]$, we can assume that $W = \langle x_1, \ldots, x_k\rangle$ for some $k < n$. 
Note that $\KK[x_1, \ldots, x_k] \subseteq A$ since it is generated by $W \subseteq A$. Let us prove that $A = \KK[x_1, \ldots x_k]$. Assume the converse, let $f$ be a polynomial of minimal degree in $A \setminus \KK[x_1, \ldots, x_k]$. Since $V$ is invariant under translations, $A$ is translation invariant as well. Then polynomials $\frac{\pa f}{\pa x_i}$ belong to $A$ and are of degree less than that of $f$, whence $\frac{\pa f}{\pa x_i} \in \KK[x_1, \ldots, x_k]$ for every $1 \le i \le n$. 

Let $f = \sum\limits_j b_j x_n^j$, $b_j \in \KK[x_1, \ldots, x_{n-1}]$. Since $\frac{\pa f}{\pa x_n} = \sum\limits_j jb_j x_n^{j-1}$ is an element of $\KK[x_1, \ldots, x_k]$, we have $f = b_1x_n + b_0$. 
For every $1 \le i < n$, the polynomial $\frac{\pa f}{\pa x_i} = \frac{\pa b_1}{\pa x_i} x_n + \frac{\pa b_0}{\pa x_i}$ does not contain $x_n$ as well, whence $\frac{\pa b_1}{\pa x_i} = 0$ for any $i$, that is $b_1 \in \KK$. Thus $x_n$ occurs in $f$ only in a linear term. The same holds for $x_{k+1}, \ldots, x_{n-1}$, that is $f$ is a sum of a linear polynomial in $x_{k+1}, \ldots, x_n$ and an element $f_0 \in \KK[x_1, \ldots, x_k]$. Since $f, f_0 \in A$, this linear polynomial belongs to $W$. But $W = \langle x_1, \ldots, x_k\rangle$, whence the linear polynomial is equal to zero, that is $f = f_0 \in \KK[x_1, \ldots, x_k]$, a contradiction. 
Thus, $A = \KK[x_1, \ldots x_k]$. Then $\frac{\pa}{\pa x_n}$ annihilates $A$ and hence $V$, which contradicts nondegeneracy of $V$. 
\end{proof}

\smallskip

Consider the canonical action of the group $\GL_n(\KK)$ on the vector space $\langle x_1, \ldots, x_n\rangle$: $x \mapsto \phi x$, $x \in \langle x_1, \ldots, x_n\rangle$, $\phi \in \GL_n(\KK)$. It induces the action of $\GL_n(\KK)$ on the algebra $\KK[x_1, \ldots, x_n]$: $(\phi f)(x_1, \ldots, x_n) := f(\phi x_1, \ldots, \phi x_n)$. Define the action of $\GL_n(\KK)$ on $\KK[S_1, \ldots, S_n]$ as follows: for $g \in \KK[S_1, \ldots, S_n]$ and $\phi \in \GL_n(\KK)$, let $(\phi g)[f] = g[\phi^{-1}f]$ for any $f \in \KK[x_1, \ldots, x_n]$. 

\begin{definition}
\label{IV_equivdef}
We say that subspaces $V_1, V_2 \subseteq \KK[x_1, \ldots, x_n]$ (resp. $I_1, I_2 \subseteq \KK[S_1, \ldots, S_n]$) are \emph{$\GL$-equivalent}, if there exists $\phi \in \GL_n(\KK)$ such that $\phi V_1 = V_2$ (resp. $\phi I_1 = I_2$).
\end{definition}

\begin{lemma}
\label{IV_GL}
The bijection in Lemma~\ref{IVbij} is well defined on classes of $\GL$-equivalence.
\end{lemma}

\begin{proof}
Let $\phi V_1 = V_2$. Then 
\begin{multline*}
I_{V_2} = \{h \in \KK[S_1, \ldots, S_n]\colon \langle h \mid \phi f\rangle = 0 \;\; \forall f \in V_1\} = \\ = 
\{h \in \KK[S_1, \ldots, S_n]\colon \langle\phi^{-1}h \mid f\rangle = 0 \;\; \forall f \in V_1\} = \\ = 
\{\phi g \in \KK[S_1, \ldots, S_n]\colon \langle g \mid f\rangle = 0 \;\; \forall f \in V_1\} = \phi I_{V_1}.
\end{multline*}
In the same way, $\phi I_1 = I_2$ implies $\phi V_1 = V_2$. 
\end{proof}

Let us say that an ideal $I \subseteq \KK[S_1, \ldots, S_n]$ is \emph{supported at the origin} if $I$ contains some powers of $S_i$ for every $1 \le i \le n$. 
It can be easily checked that an ideal $I$ is supported at the origin if and only if $I$ contains $\KK[S_1, \ldots, S_n]_{>d}$ for some~$d$. 

\smallskip

From Lemmas~\ref{IVbij}-\ref{IV_GL} we obtain the following result. 

\begin{proposition}
\label{IV}
Let $m\in\Zgezero$. Formulae~\eqref{IVconstr_upd} give a bijection between classes of $\GL$-equivalence of:
\begin{enumerate}
	\item generating subspaces $V \subseteq \KK[x_1, \ldots, x_n]$ of dimension $m$;
	\item non-degenerate ideals $I \subseteq \KK[S_1, \ldots, S_n]$ of codimension $m$ supported at the origin.
\end{enumerate}
\end{proposition}

\begin{example}
\label{IV_S1S2_ex}
The generating subspace $V = \langle 1, x_1, x_2\rangle \subseteq \KK[x_1, x_2]$ corresponds to the ideal $I = (S_1^2, S_1S_2, S_2^2) \subseteq \KK[S_1, S_2]$ as the latter consists of $g = \sum_{i, j \ge 0} \alpha_{ij}S_1^iS_2^j$ with $\alpha_{00} = \alpha_{01} = \alpha_{11} = 0$. 
\end{example}

\begin{example}
\label{IV_S3_ex}
The generating subspace $V = \bigl\langle1, x_1, x_2+\frac{x_1^2}{2}\bigr\rangle \subseteq \KK[x_1, x_2]$ corresponds to the ideal $I = (S_1^2 - S_2, S_1S_2) \subseteq \KK[S_1, S_2]$ since $g = \sum_{i, j \ge 0} \alpha_{ij}S_1^iS_2^j$ belongs to $I$ if and only if $\alpha_{00} = \alpha_{10} = \alpha_{01} + \frac{2!\alpha_{20}}{2} = 0$. 
\end{example}


\subsection{Hassett-Tschinkel correspondence} 
\label{subsecHTC}
In this subsection we describe and study the correspondence given in~\cite[Section~2.4]{HaTs1999}. 

\begin{definition}
\label{repr_equivdef}
Let $G$ be an algebraic group. Representations $\rho_1\colon G \to \GL(V_1)$ and $\rho_2\colon G \to \GL(V_2)$ are said to be \emph{equivalent} if there exist an automorphism $\psi\colon G \to G$ and an isomorphism of vector spaces $\phi\colon V_1 \to V_2$ such that $\phi(\rho_1(g)v) = \rho_2(\psi(g))\phi(v)$ for any $g \in G$, $v\in V_1$.
\end{definition}

\begin{definition}
\label{alg_equivdef}
Consider pairs $(A, U)$, where $A$ is an algebra and $U \subseteq A$ is a subspace. Two such pairs $(A_1, U_1)$ and $(A_2, U_2)$ are \emph{equivalent} if there is an algebra isomorphism $\phi\colon A_1 \to A_2$ with $\phi(U_1) = U_2$.
\end{definition}

We come to the main result of this subsection. 

\begin{theorem}
\label{HaTsch_theorem}
Let $n,m\in\Zgezero$. There are one-to-one correspondences between
\smallskip
\begin{enumerate}
	\item faithful cyclic representations $\rho\colon \GG_a^n \to \GL_m(\KK)$;
	\item pairs $(A, U)$, where $A$ is a local commutative associative unital algebra of dimension~$m$ with maximal ideal $\mm$, and $U \subseteq \mm$ is a subspace of dimension~$n$ generating the algebra~$A$;
	\item non-degenerate ideals $I \subseteq \KK[S_1, \ldots, S_n]$ of codimension $m$ supported at the origin;
	\item generating subspaces $V \subseteq \KK[x_1, \ldots, x_n]$ of dimension~$m$.
\end{enumerate}
\smallskip
These correspondences are given up to equivalences as in Definitions~\ref{IV_equivdef}-\ref{alg_equivdef}.
\end{theorem}

\begin{proof}
$(a) \to (b)$ Here we follow~\cite[Section~1]{ArSh2011}. 
Let $\rho\colon \GG_a^n \to \GL_m(\KK)$ be a faithful representation. The differential gives a representation $d\rho\colon \gg \to \gl_m(\KK)$ of the tangent algebra $\gg = \Lie(\GG_a^n)$. This defines a representation $\tau\colon \UU(\gg) \to \Mat_{m}(\KK)$ of the universal enveloping algebra $\UU(\gg)=\KK[S_1,\ldots,S_n]$. 

Let $A := \tau(\UU(\gg))$ and $U := \tau(\gg)$. The subspace $U$ generates the algebra $A$ since $\gg$ generates $\UU(\gg)$. The group $\GG_a^n$ is commutative, so $\gg$ is a commutative Lie algebra. Thus $\UU(\gg)$ is isomorphic to a polynomial algebra in $n$ variables with maximal ideal $(\gg)$ consisting of polynomials without constant term. The algebra $A$ is a commutative associative unital algebra. Since $\GG_a^n$ is a unipotent group, the image $d\rho(\gg) \subseteq \gl_m(\KK)$ consists of commuting nilpotent matrices. By definition, $\tau\bigr|_\gg = d\rho$, so $(U) = \tau((\gg))$ is a nilpotent ideal in $A$ of codimension one and the algebra $A$ is local. Since $\rho$ is faithful, it follows that $\tau\bigr|_\gg\colon \gg \to U$ is an isomorphism of vector spaces and $\dim U = n$. 

Let $v$ be a cyclic vector, that is $\langle\rho(\GG_a^n)v\rangle = \KK^m$. Note that the subspace $Av = \tau(\UU(\gg))v$ is $\gg$- and $\GG_a^n$-invariant and contains $v$, whence $Av = \KK^m$. Consider $\pi\colon A \to \KK^m$, $a \mapsto av$. Note that $\Ker\pi = 0$. Indeed, if $av = 0$ for some $a \in A$, then $a\KK^m = aAv = Aav = 0$, whence $a = 0$. Thus, $\pi$ is an isomorphism of vector spaces and $\dim A = m$. 

\emph{Equivalence.} Let $\rho_1\colon \GG_a^n \to \GL_m(\KK)$ and $\rho_2\colon \GG_a^n \to \GL_m(\KK)$ be two equivalent representations, that is there are isomorphisms $\phi\colon \KK^m \to \KK^m$ and $\psi\colon \GG_a^n \to \GG_a^n$ such that the first diagram below is commutative for any $g \in \GG_a^n$. If we differentiate it and extend $d\psi\colon \gg \to \gg$ to $\Psi\colon \UU(\gg) \to \UU(\gg)$, we obtain the central part of the second diagram for every $y \in \UU(\gg)$. 
Denote by $v_1$ a cyclic vector of $\rho_1$ and set $v_2 = \phi(v_1)$. Then $v_2$ is a cyclic vector for $\rho_2$. Identifying $A_i$ with $\KK^m$ by corresponding $\pi_i$, $i = 1,2$, and applying the diagram to $1 \in A_1$, we obtain that $\pi_2^{-1}\phi\pi_1$ maps $\tau_1(y)$ to $\tau_2(\Psi(y))$ for any $y \in \UU(\gg)$, which implies that $\pi_2^{-1}\phi\pi_1$ is an algebra isomorphism. The third diagram implies $\pi_2^{-1}\phi\pi_1(U_1) = U_2$, since $d\psi = \Psi\bigr|_\gg$ maps $\gg$ to $\gg$. 
\[
\begin{diagram}
\node{\KK^m} \arrow{e,tb}{\phi}{\sim} \arrow{s,r}{\rho_1(g)} \node{\KK^m} \arrow{s,r}{\rho_2(\psi(g))} \\
\node{\KK^m} \arrow{e,tb}{\phi}{\sim} \node{\KK^m}
\end{diagram}\quad\quad
\begin{diagram}
\node{A_1} \arrow{e,tb}{\pi_1}{\sim} \arrow{s,r}{\substack{mult. by \\ \tau_1(y)}} \node{\KK^m} \arrow{e,tb}{\phi}{\sim} \arrow{s,r}{\tau_1(y)} \node{\KK^m} \arrow{e,tb}{\pi_2^{-1}}{\sim} \arrow{s,r}{\tau_2(\Psi(y))} \node{A_2} \arrow{s,r}{\substack{mult. by \\ \tau_2(\Psi(y))}} \\
\node{A_1} \arrow{e,tb}{\pi_1}{\sim} \node{\KK^m} \arrow{e,tb}{\phi}{\sim} \node{\KK^m} \arrow{e,tb}{\pi_2^{-1}}{\sim} \node{A_2}
\end{diagram}\quad\quad
\begin{diagram}
\node{\!\UU(\gg)\!} \arrow{e,t}{\Psi} \arrow{s,r}{\tau_1} \node{\!\UU(\gg)\!} \arrow{s,r}{\tau_2} \\
\node{A_1} \arrow{e,tb}{\pi_2^{-1}\phi\pi_1}{\sim} \node{A_2}
\end{diagram}
\]

$(b) \to (a)$ Let $A$ be a local algebra with maximal ideal $\mm$, $U \subseteq \mm$ generate $A$, $\dim A = m$, and $\dim U = n$. Since $U$ consists of nilpotent elements, one can consider the subgroup $\exp U \cong \GG_a^n$ in $A^\times$ and its representation $\rho\colon \exp U \to \GL(A)$ which maps $a \in \exp U \subseteq A$ to the operator of multiplication by $a$ in $A$. 

Clearly, $\rho$ is faithful. Let us prove that $\rho$ is cyclic with the cyclic vector $1 \in A$. Let ${W := \langle\exp U\rangle}$. Note that $W$ is $(\exp U)$-invariant, it follows that $W$ is $\Lie(\exp U)$-invariant, that is $W$ is invariant under multiplication by elements in $U$. Since $U$ generates the algebra~$A$, we obtain $W = A$. 

\emph{Equivalence.} Let $\phi\colon A_1 \to A_2$ be an algebra isomorphism with $\phi(U_1) = U_2$. Then $\phi(\exp U_1) = \exp U_2$, and for any $u \in U_1$ we have $\rho_1(\exp u)\circ\phi = \phi\circ\rho_2(\phi(\exp u))$. 

\smallskip
Let us show that two constructed maps are inverse to each other. For a given representation $\rho$ we have $A = \tau(\UU(\gg)) \subseteq \Mat_m(\KK)$ and $U = \tau(\gg) = d\rho(\gg)$. The corresponding representation maps $\exp U$ to the operators of multiplication by $\exp U$ in $A$. It is equivalent to the initial representation since $\exp U \subseteq \Mat_m(\KK)$ coincides with $\exp d\rho(\gg) = \rho(\GG_a^n)$.

Conversely, for given $(A, U)$, let $\rho\colon \exp U \to \GL(A)$ be the corresponding representation. Then $d\rho\colon U \to \gl(A)$ maps $u$ to the operator of multiplication by $u$. Since the image of $\tau$ coincides with the associative algebra generated by $d\rho(U)$ and $U$ generates $A$, we obtain the algebra of operators of multiplication by elements of $A$, which is isomorphic to $A$. 

\smallskip 

$(b) \to (c)$ Denote by $s_1, \ldots, s_n$ a basis of the vector space $U$. Since $U$ generates $A$, the algebra $A$ is the image of a polynomial algebra for projection $\pi\colon \KK[S_1, \ldots, S_n] \to A$, $S_i \mapsto s_i$. Then $A \cong \KK[S_1, \ldots, S_n] \,/\, I$ for some ideal $I \subseteq \KK[S_1, \ldots, S_n]$. Since $s_i$ are nilpotent in $A$, the ideal $I$ contains some powers of all $S_i$. Since $s_i$ form a basis of $U$, it follows that $I \cap \langle S_1, \ldots, S_n\rangle = 0$ and $I$ is non-degenerate. Since $\dim A = m$, we have $\codim I = m$. 

\emph{Equivalence.} First let us check that the above construction does not depend on the choice of basis in $U$. Let $(s_1, \ldots, s_n)$ and $(\tilde s_1, \ldots, \tilde s_n)$ be two bases of $U$ corresponding to ideals $I$ and $\tilde I$; $(s_1, \ldots, s_n) = (\phi\tilde s_1, \ldots, \phi\tilde s_n)$ for some $\phi\in \GL_n(\KK)$. Then $g(S_1, \ldots, S_n) \in I$ if and only if $(\phi g)(S_1, \ldots, S_n) = g(\phi S_1, \ldots, \phi S_n) \in \tilde I$, whence $I$ is equivalent to $\tilde I$. 

Now let $(A_1, U_1)$ be equivalent to $(A_2, U_2)$, that is there is an isomorphism $\phi\colon A_1 \to A_2$ with $\phi(U_1) = U_2$. According to the above, we can choose a basis in $U_2$ as the image of a basis in $U_1$ under $\phi$ and obtain $I_1 = I_2 \subseteq \KK[S_1, \ldots S_n]$. 

\smallskip 

$(c) \to (b)$ For a given ideal $I \subseteq\KK[S_1, \ldots, S_n]$, let $A := \KK[S_1, \ldots, S_n]\,/\,I$, $s_i := S_i + I$, and $U := \langle s_1, \ldots, s_n\rangle$. 

The elements $s_i$ are nilpotent since some powers of $S_i$ belong to $I$. It follows that the ideal $(s_1, \ldots, s_n)$ is nilpotent of codimension one, whence the algebra $A$ is local. As above, $\dim U = n$ since $I$ is non-degenerate and $\dim A = \codim I = m$. 

\emph{Equivalence.} If ideals $I_1$ and $I_2$ are equivalent, we have an automorphism of $\KK[S_1, \ldots, S_n]$ that induces the desired isomorphism of factor-algebras $A_1 \to A_2$. 

Clearly, two constructed maps are inverse to each other. 

\smallskip 

$(c) \leftrightarrow (d)$ See Proposition~\ref{IV}.
\end{proof}

Below we explain a method to compute the generating subspace $V$ corresponding to a given pair $(A, U)$, see~\cite[Proposition~2.11]{HaTs1999}. 

\begin{construction}
\label{AtoVconstr}
Suppose $A$ is a local algebra of dimension $m$ with maximal ideal $\mm$, and a subspace $U \subseteq \mm$ of dimension $n$ generates the algebra $A$, see Theorem~\ref{HaTsch_theorem}~(b). These data define a representation of $A$ as a factor-algebra $A = \KK[S_1, \ldots, S_n]\,/\,I$: for a basis $s_1, \ldots, s_n$ of the subspace $U$, let the ideal $I$ be the kernel of the surjective homomorphism $\pi\colon\KK[S_1, \ldots, S_n] \to A$, $S_i \mapsto s_i$. 

For the sequel we need a basis of the algebra $A$. Consider a homogeneous lexicographic order on $\KK[S_1, \ldots, S_n]$. Let $\mu_1, \ldots, \mu_k$ be monomials that are not leading terms of polynomials from $I$. Let us prove that $\mu_i$ form a basis of $A$. They are linearly independent in $A$ since a linear combination of $\mu_i$ has one of $\mu_i$ as a leading term and can not belong to $I$. Further, consider any element of $A$. It is a linear combination of some monomials; if some of these monomials is not equal to $\mu_i$, then it is the leading term for some $f \in I$ and we can reduce the given element by $f$. In such a way we obtain a representation of the element as a linear combination of $\mu_i$. 

Since $x_1s_1+\ldots+x_ns_n \in U \subseteq \mm$ is nilpotent for any $x_1, \ldots, x_n \in \KK$ and $\mu_i$ form a basis of $A$, we can expand
\[
\exp(x_1s_1+\ldots+x_ns_n) = \sum_{i = 1}^m f_i(x_1, \ldots, x_n)\mu_i.
\]
For $g \in \KK[S_1, \ldots, S_n]$, denote by $g_x$ the same polynomial in variables $\frac{\pa}{\pa x_i}$. One can easily check that
\[
\frac{\pa}{\pa x_i}[\exp(x_1S_1+\ldots+x_nS_n)]=S_i\exp(x_1S_1+\ldots+x_nS_n).
\] 
This leads to the identity
\[
g_x[\exp(x_1S_1 + \ldots + x_nS_n)] = g \exp(x_1S_1 + \ldots + x_nS_n).
\]
Substituting $S_i = s_i$ to this identity, we obtain
\begin{equation}
\label{AtoVconstr_eq}
\sum\limits_{i = 1}^m g_x[f_i(x_1, \ldots, x_n)]\mu_i = \pi(g)\sum\limits_{i = 1}^m f_i(x_1, \ldots, x_n)\mu_i.
\end{equation}
Note that $\{\sum f_i(x_1, \ldots, x_n)\mu_i \mid x_i \in \KK\} = \exp U$ by definition and $\langle\exp U\rangle = A$ by the proof of $(b) \to (a)$ in Theorem~\ref{HaTsch_theorem}. In particular, $f_i$ are linearly independent. Then the right side of~\eqref{AtoVconstr_eq} equals zero for any $x_i \in \KK$ if and only if $\pi(g) = 0$ in $A$, that is $g \in I$. On the other hand, the left side equals zero for any $x_i \in \KK$ if and only if $g_x[f_i] = 0$ for any $1 \le i \le m$. It follows that $f_i \in V$, where $V$ is the generating subspace corresponding to the ideal $I$, see Lemma~\ref{IVbij_id_ti}. So we obtain

\begin{lemma}
\label{AtoVconstr_lem}
The polynomials $f_i$, $1 \le i \le m$, form a basis of the generating subspace~$V$ corresponding to the given pair $(A, U)$. 
\end{lemma}
\end{construction}

\begin{example}
\label{AtoVconstr_S3_ex}
Let us consider a local algebra $A = \KK[S]\,/\,(S^3)$ with maximal ideal $\mm = \langle S, S^2\rangle$. 

\smallskip

1) Take $U = \mm$. According to Construction~\ref{AtoVconstr}, choose a basis $s_1 = S + (S^3), \, s_2 = S^2 + (S^3)$ of $U$ and let $I$ be the kernel of the projection $\pi\colon \KK[S_1, S_2] \to A, \; S_i \mapsto s_i$:
\[\begin{aligned}
I = (S_1^2 &-S_2, S_1S_2), \quad A = \KK[S_1, S_2]\,/\,I,\\
s_1 &= S_1 + I, \;\; s_2 = S_2 + I.
\end{aligned}\]
We will omit $+\,I$ for convenience. Elements $\mu_1 = 1$, $\mu_2 = S_1$, $\mu_3 = S_2$ form a basis of $A$. Since $S_2 = S_1^2$ and $S_1^3 = 0$ in $A$, it follows that
\begin{multline*}
\exp(x_1s_1 + x_2s_2) = \exp(x_1S_1 + x_2S_1^2) = 1 + x_1S_1 + \Bigl(x_2+\frac{x_1^2}{2}\Bigr)S_1^2 = 1 + x_1\mu_1 + \Bigl(x_2+\frac{x_1^2}{2}\Bigr)\mu_2,
\end{multline*}
whence $f_1 = 1$, $f_2 = x_1$, and $f_3 = x_2+\frac{x_1^2}{2}$. By Lemma~\ref{AtoVconstr_lem}, $V = \bigl\langle1, x_1, x_2+\frac{x_1^2}{2}\bigr\rangle$. This agrees with Example~\ref{IV_S3_ex}. 

\smallskip

2) Take $U = \langle S\rangle$. Its basis $s_1 = S + (S^3)$ corresponds to
\[\begin{aligned}
I = (S_1^3) \subseteq \KK[S_1], &\quad A = \KK[S_1]\,/\,I\\
s_1 &= S_1 + I.
\end{aligned}\]
For $\mu_1 = 1$, $\mu_2 = S_1$, $\mu_3 = S_1^2$, we have $\exp(x_1S_1) = 1 + x_1S_1 + \frac{x_1^2}{2}S_1^2$,
whence $V = \langle1, x_1, x_1^2\rangle$ in $\KK[x_1]$. This agrees with Example~\ref{IV_S13_ex}.
\end{example}

\begin{example}
\label{AtoVconstr_S1S2_ex}
In the same way one can see that the algebra $A = \KK[S_1, S_2]\,/\,(S_1^2, S_1S_2, S_2^2)$ with $U = \mm = \langle S_1, S_2\rangle$ corresponds to the generating vector space $\langle1, x_1, x_2\rangle \subseteq \KK[x_1, x_2]$, which agrees with Example~\ref{IV_S1S2_ex}. There is no other subspace $U \subseteq \mm$ generating the algebra~$A$. 
\end{example}

Now we are going to discuss duality properties for modules under consideration. In particular, we provide complete proofs for results mentioned in~\cite[Remark~2.13]{HaTs1999}.
Let us recall that a generating subspace $V$ contains constants, so the action of $\GG_a^n$ on $V$ by translations is linear.

\begin{lemma}
\label{dual_lem}
In notation of Theorem~\ref{HaTsch_theorem}, the dual of a representation $\rho\colon\GG_a^n \to \GL_m(\KK)$ is equivalent to the representation $\tau\colon \GG_a^n \to \GL(V)$ by translations. 
\end{lemma}

\begin{proof}
Let $\langle\,\cdot\mid\cdot\,\rangle$ be the pairing between $\KK[S_1, \ldots, S_n]$ and $\KK[x_1, \ldots, x_n]$ as in Construction~\ref{IVconstr}. Note that 
\begin{equation}
\label{dual_eq}
\bigl\langle \exp(\beta_1S_1 + \ldots + \beta_nS_n) g \mid f(x)\bigr\rangle = \bigl\langle g \mid f(x + \beta)\bigr\rangle
\end{equation}
for any $\beta = (\beta_1, \ldots, \beta_n) \in \GG_a^n$, $f \in \KK[x_1, \ldots, x_n]$, $g \in \KK[S_1, \ldots, S_n]$. Indeed, the left side equals $\langle g \mid \exp(\beta_1S_1 + \ldots + \beta_nS_n)[f(x)]\rangle$, which coincides with $\langle g \mid f(x + \beta)\rangle$ by Taylor's theorem. 
Since $\langle I_V \mid V\rangle = 0$, we can consider $\langle\,\cdot\mid\cdot\,\rangle$ as a pairing between
$A = \KK[S_1, \ldots, S_n]\,/\,I_V$ and $V \subseteq \KK[x_1, \ldots, x_n]$. According to the proof of Theorem~\ref{HaTsch_theorem}, we have $\rho\colon \exp U \to \GL(A)$, where $U = \langle S_1, \ldots, S_n\rangle$, so equation~\eqref{dual_eq} implies 
\[\langle\rho(-\beta)g \mid f\rangle = \langle g \mid \tau(\beta)f\rangle\]
for any $\beta \in \GG_a^n$, $f \in V$, $g \in A$ (we identify $\beta_1S_1 + \ldots \beta_nS_n$ with $-\beta$ for $\exp U \cong \GG_a^n$). It follows that the representations $\rho$ and $\tau$ are dual. 
\end{proof}

\begin{example}
\label{dual_S1S2S3_ex}
Let $A = \KK[S]\,/\,(S^3)$ and $U = \mm = \langle S, S^2\rangle$ as in Example~\ref{AtoVconstr_S3_ex}.1). According to $(b) \to (a)$ of Theorem~\ref{HaTsch_theorem}, the corresponding representation $\rho\colon \GG_a^2 \to \GL_3(\KK)$ is the representation of $\exp U$ in $A$ via multiplication. For an element $x_1S + x_2S^2$ in $U$, we have
\[
\exp(x_1S + x_2S^2) = 1 + x_1S + \bigl(x_2+\frac{x_1^2}{2}\bigr)S^2, 
\]
whence the representation $\rho$ in the basis $1, S, S^2$ of the algebra $A$ is given by
$\rho(x_1, x_2) = \begin{pmatrix}
1 & 0 & 0 \\
x_1 & 1 & 0 \\
x_2 + \frac{x_1^2}{2} & x_1 & 1
\end{pmatrix}$.
For $A = \KK[S_1, S_2]\,/\,(S_1^2, S_1S_2, S_2^2)$ and $U = \mm = \langle S_1, S_2\rangle$ we obtain 
$\rho(x_1, x_2) = \begin{pmatrix}
1 & 0 & 0\\
x_1 & 1 & 0 \\
x_2 & 0 & 1
\end{pmatrix}$. 
This agrees with Lemma~\ref{dual_lem}: the matrices of representation in~$V$ in Examples~\ref{IV_dual_S3_ex} and~\ref{IV_dual_S1S2_ex} are transposed to the above ones. 
\end{example}

In other words, Lemma~\ref{dual_lem} states that $A$ and $V$ are dual $\GG_a^n$-modules. 

\begin{proposition}
\label{LKJH}
In notation of Theorem~\ref{HaTsch_theorem}, the following conditions are equivalent:
\begin{enumerate}
	\item $\GG_a^n$-modules $A$ and $V$ are equivalent;
	\item $\GG_a^n$-module $V$ is cyclic;
	\item the algebra $A$ is Gorenstein.
\end{enumerate}
\end{proposition}

\begin{proof}
$(a) \Rightarrow (b)$ The module $V \cong A$ is cyclic since the algebra $A$ contains a unit. 

\smallskip

$(b) \Rightarrow (a)$ Since the module structure on $V$ is given by translation operators from $\exp U$, $U = \langle S_1, \ldots, S_n\rangle$, and $V$ is cyclic, it follows that there exists a polynomial $f_0 \in V$ such that $V = \langle(\exp U)[f_0]\rangle = (\KK[S_1, \ldots, S_n])[f_0]$. Hence the kernel of the valuation $\pi\colon \KK[S_1, \ldots, S_n] \to V$, $g \mapsto g[f_0]$, is equal to 
\[\Ker \pi = \{g \in \KK[S_1, \ldots, S_n] \mid g[f_0] = 0\} =
 \{g \in \KK[S_1, \ldots, S_n] \mid g[V] = 0\} = I.\]
Thus $\pi$ gives an isomorphism between $A = \KK[S_1, \ldots, S_n]\,/\,I$ and $V$, which is an isomorphism of $\GG_a^n$-modules since the module structure on $A$ is given by $\exp U$ as well. 

\smallskip

$(b) \Leftrightarrow (c)$ 
Invariant one-dimensional subspaces $\langle a\rangle$ in $A$ correspond to invariant hyperplanes $\langle a\rangle^\perp$ in the dual module $V$. Since $\GG_a^n$ is unipotent, a one-dimensional vector space is invariant if and only if it consists of fixed points. Notice that $\Soc A$ is the set of fixed points in $A$. Indeed, $(\exp U)a = a$ if and only if $Ua = 0$, i.e. $\mm a = 0$. 

If $\dim \Soc A > 1$, the corresponding invariant hyperplanes cover $V$. Indeed, any $f \in V$ is contained in $\langle a\rangle^\perp$, where $a \in \Soc A \cap \langle f\rangle^\perp$. So there is no cyclic vector in this case. 

If $\dim \Soc A = 1$, there is a unique invariant hyperplane in $V$. Let us prove that any vector in the complement of this hyperplane is cyclic. It is sufficient to show that any proper invariant subspace in $V$ is contained in an invariant hyperplane. Indeed, for $W \subseteq V$, consider the invariant subspace $W^\perp \subseteq A$; by the Lie-Kolchin theorem there exists an invariant one-dimensional subspace $\langle a\rangle \subseteq W^\perp$, which corresponds to the required hyperplane $\langle a\rangle^\perp \supseteq~W$. 
\end{proof}


\subsection{The case of additive actions} 
\label{subsecAA}
In this subsection we combine the results of two previous subsections. 

\begin{definition}
A generating subspace $V \subseteq \KK[x_1, \ldots, x_n]$ is called \emph{basic} if $\dim V=n+1$. 
\end{definition}

Basic subspaces are minimal generating subspaces of a polynomial algebra. 

\begin{example}
\label{basic4_ex}
One can check that the following vector subspaces in $\KK[x_1, x_2, x_3, x_4]$ are basic:
\[\begin{aligned}
V_1 = \bigl\langle1, x_1, x_2, x_3, x_4\bigr\rangle, &\quad\quad
V_2 = \Bigl\langle1, x_1, x_2, x_3 + \frac{x_1^2}{2}\Bigr\rangle,\\
V_3 = \Bigl\langle1, x_1, x_2, x_3+x_1x_2\Bigr\rangle, &\quad\quad
V_4 = \Bigl\langle1, x_1, x_2+\frac{x_1^2}{2}, x_3 + x_1x_2 + \frac{x_1^3}{6}\Bigl\rangle.
\end{aligned}\] 
\end{example}

Hassett-Tschinkel correspondence for $m = n + 1$ (see \cite[Proposition~2.15]{HaTs1999}) or Knop-Lange theorem for $r = 0$ implies a description of additive actions on projective spaces. In view of correspondence $(b) \to (d)$ in Theorem~\ref{HaTsch_theorem}, the basic subspace is determined just by the algebra $A$ as we have to set $U=\mm$. 

\begin{theorem}
\label{HaTschKnLa_theor}
There are one-to-one correspondences between the following objects:
\begin{enumerate}
	\item additive actions on $\PP^n$, i.e. effective actions $\alpha\colon \GG_a^n \times \PP^n \to \PP^n$ with an open orbit; 
	\item faithful cyclic representations $\rho\colon \GG_a^n \to \GL_{n+1}(\KK)$;
	\item local commutative associative unital algebras $A$ of dimension $n+1$;
	\item basic subspaces $V \subseteq \KK[x_1, \ldots, x_n]$.
\end{enumerate}
\smallskip
These correspondences are considered up to equivalences as in Definitions~\ref{act_equivdef} and~\ref{IV_equivdef}-\ref{alg_equivdef}.
\end{theorem}

From Theorem~\ref{localg6_prop} we obtain the following statement. 

\begin{corollary} \label{caafin} 
The projective space $\PP^n$ admits a finite number of additive actions if and only if $n \le 5$. 
\end{corollary}

\begin{example} \label{eexx} 
According to Table~\ref{table_localg6}, there are two local algebras of dimension~$3$. The corresponding additive $\GG_a^2$-actions on $\PP^2$ are found in Examples~\ref{KnL_S1S2_ex} and~\ref{KnL_S3_ex}, and basic subspaces are given in Examples~\ref{AtoVconstr_S3_ex}.1) and~\ref{AtoVconstr_S1S2_ex}. Faithful cyclic representations are written in Example~\ref{dual_S1S2S3_ex}. We gather the results in the following table:
\renewcommand{\arraystretch}{1.2}
$$\begin{array}{|l|c|c|}\hline
\text{Additive actions} & [z_0 : z_1+\alpha z_0 : z_2+\beta z_0] & [z_0 : z_1+\alpha z_0 : z_2+\alpha z_1+\Bigl(\beta+\frac{\alpha^2}{2}\Bigr)z_0]\\\hline
\text{Representations} & \begin{pmatrix}1 & 0 & 0 \\ \alpha & 1 & 0 \\ \beta & 0 & 1\end{pmatrix} & \begin{pmatrix}1 & 0 & 0 \\ \alpha & 1 & 0 \\ \beta +\frac{\alpha^2}{2} & \alpha & 1\end{pmatrix}\\\hline
\text{Local algebras} & \KK[S_1, S_2]\,/\,(S_1^2, S_1S_2, S_2^2) & \KK[S]\,/\,(S^3) \\\hline
\text{Basic vector subspaces} & \langle1, x_1, x_2\rangle & \bigl\langle1, x_1, x_2+\frac{x_1^2}{2}\bigr\rangle\\\hline
\end{array}$$
In the same way it can be proved that basic subspaces of Example~\ref{basic4_ex} correspond to four local algebras of dimension~$4$ from Table~\ref{table_localg6} and so are the only basic subspaces in this case. 
\end{example}

Recall that by Corollary~\ref{fin_orb_Gan_cor} there is a unique additive action on~$\PP^n$ with finitely many orbits; it corresponds to the local algebra $A=\KK[S]\,/\,(S^{n+1})$. One may look for a generalization of this result. Namely, the \emph{modality} of an action of a connected algebraic group~$G$ on a variety $X$ is the maximal value of minimal codimension of a $G$-orbit in $Y$ over all irreducible $G$-invariant subvarieties $Y$ in $X$. In other words, the modality is the maximal number of parameters in a continuous family of $G$-orbits on $X$. In particular, the modality is zero if and only if the number of $G$-orbits on $X$ is finite.

A classification of additive actions on $\PP^n$ of modality one is obtained in~\cite[Theorem~3.1]{ArSh2011}. Such actions correspond to the following 2-generated pairwise non-isomorphic local algebras:
\[
A_{a,b}=\KK[S_1,S_2]/(S_1^{a+1},S_2^{b+1},S_1S_2), \ a\ge b\ge 1; \quad B_{a,b}=\KK[S_1,S_2]/(S_1S_2, S_1^a-S_2^b), \ a\ge b\ge 2 ;
\]
\[
C_a=\KK[S_1,S_2]/(S_1^{a+1}, S_2^2 - S_1^3), \ a \ge 3; \quad C_a^1=\KK[S_1,S_2]/(S_1^{a+1}, S_2^2 - S_1^3, S_1^aS_2), \ a\ge 3;
\]
\[
C_a^2=\KK[S_1,S_2]/(S_1^{a+1}, S_2^2 - S_1^3, S_1^{a-1}S_2), \ a \ge 3; \quad C_a^3=\KK[S_1,S_2]/(S_2^2 - S_1^3, S_1^{a-2}S_2), \ a\ge 4;
\]
\[
D=\KK[S_1,S_2]/(S_1^3, S_2^2); \quad E=\KK[S_1,S_2]/(S_1^3, S_2^2 ,S_1^2S_2).
\]

\smallskip

Clearly, the maximum possible value of modality of a non-trivial action of a connected algebraic group on an irreducible $n$-dimensional algebraic variety is $n-1$. It follows from Corollary~\ref{orb_as_id_cor} that this maximal value is achieved only for one additive action on~$\PP^n$. This action corresponds to the local algebra $A$ with the condition $\mm^2=0$, i.e, $A=\KK[S_1,\ldots,S_n]/(S_iS_j, \ 1\le i\le j\le n)$. In this case the hyperplane $\PP(\mm)$, which is the complement to the open orbit in $\PP^n$, consists of $\GG_a^n$-fixed points.  

In this case we may consider the blowup $X$ of the projective space $\PP^n$ along a smooth subvariety contained in $\PP(\mm)$ and lift the additive action from $\PP^n$ to $X$. 
This proves the following result providing many projective varieties admitting an additive action. 

\begin{proposition}
\label{proprr}
Let $X$ be the blowup of the projective space $\PP^n$ along a smooth subvariety contained in a hyperplane in $\PP^n$. Then $X$ admits an additive action.
\end{proposition} 

We finish this section with a characterization of Gorenstein local algebras in terms of Hassett-Tschinkel correspondence. For any action of an algebraic group~$G$ on a variety~$X$ there is a closed $G$-orbit. If the variety~$X$ is complete, any closed orbit is complete as well. If $G$ is unipotent, such an orbit is a $G$-fixed point. So, an action of a unipotent group $G$ on a complete variety $X$ has a fixed point. 

\begin{proposition}
\label{GLA}
In notation of Theorem~\ref{HaTschKnLa_theor}, the following conditions are equivalent:
\begin{enumerate}
	\item an additive action on $\PP^n$ has a unique fixed point;
	\item the corresponding local algebra $A$ is Gorenstein. 
\end{enumerate}
\end{proposition}

\begin{proof}
As was observed in the proof of $(b) \Leftrightarrow (c)$ in Proposition~\ref{LKJH}, the set of fixed points of the action of $\GG_a^n$ on $A$ is $\Soc A$. Since a unipotent group has no non-trivial character, 
the set of fixed points of the corresponding additive action on $\PP^n=\PP(A)$ is $\PP(\Soc A)$. So a fixed point is unique if and only if the ideal $\Soc A$ is one-dimensional. By definition, it means that the algebra $A$ is Gorenstein. 
\end{proof} 


\section{Generalizations of Hassett-Tschinkel correspondence} 
\label{secght}

In this section we adapt the method of Hassett and Tschinkel to the study of additive actions on projective varieties~$X$ different from projective spaces. We introduce an induced additive action as an additive action on~$X$ that can be extended to an ambient projective space. It turns out that every such action comes from an additive action on the projective space via the restriction to a subgroup of the acting vector group. So such actions are given by pairs $(A,U)$, where $A$ is a local algebra defining an additive action on the projective space and~$U$ is a subspace in~$\mm$ that represents the subgroup. 

In subsection~\ref{subbus} we consider the case when the projective subvariety~$X$ is a hypersurface and describe a method to write down explicitly the homogeneous equation for~$X$ in terms of the pair $(A,U)$. In particular, the degree of this equation is equal to the maximal number~$d$ such that the ideal $\mm^d$ is not contained in~$U$. These results imply that smooth projective hypersurfaces that admit an additive action are precisely hyperplanes and non-degenerate quadrics. Moreover, if a hypersurface in $\PP^n$ admits an additive action then its degree does not exceed~$n$. 

In the next three subsections we apply the methods of multilinear algebra to study additive actions on projective hypersurfaces. Namely, we consider the $d$-linear form on the algebra~$A$ which is the polarization of the equation defining the hypersurface~$X$ and characterize additive actions on~$X$ in terms of this form. This allows to describe additive actions on non-degenerate and degenerate quadrics, some cubics, and to prove in Theorem~\ref{tnew} that any non-degenerate projective hypersurface admits at most one additive action. Also, it is shown that induced additive actions on non-degenerate hypersurfaces come from Gorenstein local algebras. 

\subsection{Additive actions on projective subvarieties}
Let $X$ be a closed subvariety of dimension~$n$ in a projective space~$\PP^{m-1}$. Throughout this section we assume that $X$ is not contained in any hyperplane of~$\PP^{m-1}$, i.e., the subvariety~$X$ is linearly non-degenerate. In this subsection we introduce the notion of induced additive action and give a variant of the Hassett-Tschinkel correspondence for induced additive actions. 

\begin{definition}
\label{indac_equivdef}
An additive action $\GG_a^n \times X \to X$ is \emph{induced} if it can be extended to an action $\GG_a^n \times \PP^{m-1} \to \PP^{m-1}$. 
Two induced additive actions ${\alpha_i\colon\GG_a^n \times X_i \to X_i}$, ${X_i \subseteq \PP^{m-1}}$, $i=1,2$, are said to be \emph{equivalent} if there are automorphisms of groups ${\psi\colon \GG_a^n \to \GG_a^n}$ and of varieties $\phi\colon \PP^{m-1} \to \PP^{m-1}$ such that $\phi(X_1) = X_2$ and 
${\phi\circ\alpha_1= \alpha_2\circ(\psi\times\phi)}$.
\end{definition}

\begin{example}
\label{hypnonind_ex}
Consider a cuspidal cubic plane curve $X = \{z_0^2z_3 = z_1^3\} \subseteq \PP^2$. Let us show that $X$ admits an additive action, but has no induced additive action. Let us act in the affine chart $\{z_0 \ne 0\}$ by translation on~$\frac{z_1}{z_0}$, i.e. an element $a \in \GG_a$ acts via 
\[[z_0:z_1:z_3]=\left[1 \,:\, \tfrac{z_1}{z_0} \,:\, \bigl(\tfrac{z_1}{z_0}\bigr)^{3} \right] \;\mapsto\; \left[1 \,:\, \tfrac{z_1}{z_0}+a \,:\, \bigl(\tfrac{z_1}{z_0}+a\bigr)^{\!3} \right].\] 
Notice that we have the identity $\frac{z_1}{z_0} + a = \left(\!z_1\bigl(\frac{z_1}{z_0} - a\bigr)\!\right)^{-1}\!\!(z_3-a^2z_1)$, which follows from the equation of~$X$. Substituting this identity, multiplying homogeneous coordinates by $\left(\!z_1\bigl(\frac{z_1}{z_0} - a\bigr)\!\right)^3$ and using the equation of~$X$, we obtain that $a \in \GG_a$ acts on $[z_0:z_1:z_3]$ via
\begin{gather*}
\left[1 : \frac{z_3-a^2z_1}{z_1\bigl(\frac{z_1}{z_0} - a\bigr)} : \left(\!\frac{z_3-a^2z_1}{z_1\bigl(\frac{z_1}{z_0} - a\bigr)}\right)^{\!\!3} \right] 
= \left[z_1^3\bigl(\tfrac{z_1}{z_0} - a\bigr)^{\!3} : z_1^2\bigl(\tfrac{z_1}{z_0} - a\bigr)^{\!2}(z_3-a^2z_1) : (z_3-a^2z_1)^3\right] = \\
= \left[z_0z_3^2 - 3az_1^2z_3 + 3a^2z_0z_1z_3 -z_1^3a^3 \,:\, (z_1z_3 - 2az_0z_3 + z_1^2a^2)(z_3-a^2z_1) \,:\, (z_3-a^2z_1)^3\right],
\end{gather*}
which is well defined at the unique point $[0:0:1] \in X$ not belonging to the affine chart $\{z_0 \ne 0\}$. Thus we obtain an additive action on $X$. However, by Corollary~\ref{hypdeg_cor} below the degree of a curve admitting an induced additive action on~$\PP^2$ is at most~$2$, i.e. $X$ has no induced additive action. 

This example is also treated in \cite[Subsection~4.1]{HaTs1999}. Here the action on $X$ is constructed from the additive action on the normalization~$\PP^1$ of~$X$. 
\end{example}

\begin{remark}
We denote the third coordinate $z_3$ instead of $z_2$ for a good reason, see Remark~\ref{hyp_twisted_rem}.
\end{remark}

Let us consider the case of a smooth hypersurface $X \subseteq \PP^{m-1}$ of degree~$d$. Denote by $\Aut(X)$ the group of (regular) automorphisms of $X$ and by $\Aut_l(X) \subseteq \PGL_m(X)$ the group of linear automorphisms of $X$. By~\cite[Theorem~2]{MM1964}, we have $\Aut(X) = \Aut_l(X)$ if $m \ge 5$ or $d \ne m$. Under these assumptions, if $X$ admits an additive action, then it admits an induced additive action. This theorem covers all smooth hypersurfaces except of cases $(d,m)=(3,3)$ and $(d,m)=(4,4)$. In case $(d,m) = (3,3)$ we have a cubic curve with genus one. It does not admit an additive action since any variety admitting an additive action is rational. By~\cite[Theorem~4]{MM1964}, in the case $(d,m)=(4,4)$ the connected component $\Aut(X)^0$ of the automorphism group is trivial, which implies that there is no additive action as well. By~\cite[Theorem~1]{MM1964}, if $m \ge 4$ and $d \ge 3$, then the group $\Aut_l(X)$ is finite, so $X$ admits no induced additive action. Thus, we have the following result. 

\begin{proposition} \label{newprop}
There is no additive action on smooth hypersurfaces $X \subseteq \PP^{m-1}$ of degree~$d$ for $m \ge 3$ and $d \ge 3$.
\end{proposition}

As we will see in Theorem~\ref{hyp_quadr_nondeg_theor} below, there is a unique additive action on the non-degenerate quadric of any dimension. 

\smallskip 

An irreducible subvariety $X \subseteq \PP^{m-1}$ is called \emph{linearly normal} if the map $H^0(\PP^{m-1},\Of(1)) \to H^0(X, \Of(1))$ is surjective, or, equivalently, the subvariety is not contained in any hyperplane and it is not a linear projection of a subvariety from a bigger projective space. 

\begin{proposition} {\cite[Section~2]{AP2014}}
Let $X$ be linearly normal in $\PP^{m-1}$ and admit an additive action. Then this action is induced.
\end{proposition}

Let us recall that in Section~\ref{secpn} we established the Hassett-Tschinkel correspondence between faithful cyclic $\GG_a^n$-representations, pairs $(A, U)$ of local algebras and their subspaces, ideals and subspaces in polynomial algebras with some conditions up to equivalences; see Theorem~\ref{HaTsch_theorem}, $(a)-(d)$. Now we are ready to add one more item $(e)$ to this theorem; it characterizes induced additive actions on projective subvarieties. 

\begin{theorem}
\label{hypHaTsch_theorem}
Let $n,m\in\Zgezero$. There are one-to-one correspondences between
\smallskip
\begin{enumerate}
	\item faithful cyclic representations $\rho\colon \GG_a^n \to \GL_m(\KK)$;
	\item pairs $(A, U)$, where $A$ is a local commutative associative unital algebra of dimension~$m$ with maximal ideal $\mm$, and $U \subseteq \mm$ is a subspace of dimension~$n$ generating the algebra~$A$;
	\item non-degenerate ideals $I \subseteq \KK[S_1, \ldots, S_n]$ of codimension $m$ supported at the origin;
	\item generating subspaces $V \subseteq \KK[x_1, \ldots, x_n]$ of dimension~$m$;
	\item classes of equivalence of induced additive actions on projective subvarieties of dimension $n$ in $\PP^{m-1}$ not contained in a hyperplane. 
\end{enumerate}
\end{theorem}

\begin{proof}
Let us construct the correspondence from $(a)$ to $(e)$. Consider the canonical projections $p\colon \KK^m \setminus \{0\} \to \PP^{m-1}$ and $\pi\colon \GL_m(\KK) \to \PGL_m(\KK)$. 

A faithful representation $\GG_a^n \to \GL_m(\KK)$ defines a subgroup $\GG_a^n \subseteq \GL_m(\KK)$. For the group $\KK^\times$ of nonzero scalar matrices, we have the direct product $H = \KK^\times \times \GG_a^n$ in $\GL_m(\KK)$. Let $v$ be a cyclic vector in $\KK^m$ and $X$ be the projectivization of the closure of the orbit $Hv \subseteq \KK^m \setminus \{0\}$, i.e.
\[X = p(\overline{Hv}) \subseteq \PP^{m-1}.\]
Let the effective action on $X$ be given by $\pi(H) \subseteq \PGL_m(\KK) = \Aut(\PP^{m-1})$. Note that $\pi(H) \cong \GG_a^n$ since $\Ker \pi = \KK^\times \subseteq H$. Then $p(Hv)$ is an open orbit in $X$, and $X$ is not contained in any hyperplane since $v$ is a cyclic vector. We will see below that the resulting subvariety~$X$ and the additive action on it do not depend on the choice of a cyclic vector~$v$. 

Conversely, let a subvariety $X \subseteq \PP^{m-1}$ admit an induced additive action. Then $X$ is the closure of an orbit of effective action $\GG_a^n \times \PP^{m-1} \to \PP^{m-1}$. Consider $\GG_a^n$ as a subgroup in $\PGL_m(\KK)$ and let $H = \pi^{-1}(\GG_a^n) \subseteq \GL_m(\KK)$. Then $H \cong \KK^\times \times \GG_a^n$, where $\KK^\times$ is a subgroup of scalar matrices as above, and the subgroup $\{1\} \times \GG_a^n \subseteq H$ gives the corresponding faithful representation of~$\GG_a^n$. Let $\langle v\rangle \in \PP^{m-1}$ be a point in the open orbit of $X$ for some $v \in \KK^m$. Since $X$ is not contained in any hyperplane, the same holds for its open orbit $\GG_a^n\langle v\rangle$, whence an orbit $Hv = p^{-1}(\GG_a^n\langle v\rangle) \subseteq \KK^m$ is not contained in any hyperplane of $\KK^m$ and $v$ is a cyclic vector for $\GG_a^n$. 

\smallskip

Thus, a subvariety in~$\PP^{m-1}$ of dimension~$n$ with induced additive action is the projectivization of the closure of an orbit of a cyclic vector for a $\KK^\times \times \GG_a^n$-representation in~$\KK^m$. In order to show that the construction does not depend on the choice of a cyclic vector, we use item $(b)$. The representation $\GG_a^n \to \GL_m(\KK)$ corresponding to a pair $(A, U)$ is constructed in the following way: $\exp U \cong \GG_a^n$ acts on $A \cong \KK^m$ by multiplication in the algebra $A$. In these terms, the representation of $\KK^\times \times \GG_a^n$ we are interested in is the representation of the group $\KK^\times \exp U$ in $A$ by operators of multiplication, and the orbit of an element $a \in A$ is the set $\KK^\times\exp U\,a$. 

Recall that any element of the maximal ideal $\mm$ in the local algebra $A$ is nilpotent, and any element of $A \setminus \mm$ is invertible. If $a \in \mm$, then $a$ is nilpotent, and all the elements in $\KK^\times\exp U\,a$ are nilpotent as well, so $\KK^\times \exp U\,a \subseteq \mm$. This implies that $X \subseteq \PP(A)$ is contained in the hyperplane $\PP(\mm)$, so we do not consider this case. If $a \in A \setminus \mm$, then $a$ is invertible, so the orbit $\KK^\times \exp U \cdot 1$ is isomorphic to $\KK^\times\exp U\,a$ via the linear operator~$L_a$ of multiplication by~$a$. This isomorphism commutes with $\GG_a^n$-actions on these orbits by commutativity of multiplication in~$A$. Thus, for any pair $(A, U)$ there is a unique induced additive action corresponding to this pair up to equivalence of induced additive actions, which does not depend on the choice of a cyclic vector. 
\end{proof}

Let us sum up the discussion by the correspondence $(b) \to (e)$ between pairs $(A, U)$ and induced additive actions. 

\begin{construction}
\label{hyp_constr}
Suppose that $A$ is a local commutative associative unital algebra of dimension~$m$ with maximal ideal $\mm$, $U \subseteq \mm$ is a subspace of dimension~$n$ generating the algebra~$A$, and let $p\colon A \setminus \{0\} \to \PP(A) = \PP^{m-1}$ be the canonical projection. According to the proof of Theorem~\ref{hypHaTsch_theorem}, the corresponding projective subvariety is the projectivization of an orbit of a cyclic vector, i.e. 
\[X = p(\overline{\KK^\times \exp U}),\]
the additive action on $X$ is given by the operators of multiplication by elements from $\GG_a^n \cong \exp U \subseteq A$, and the set $p(\KK^\times \exp U) = p(\exp U)$ is an open orbit in~$X$. Denote by~$z_0$ the coordinate in $A = \KK \oplus \mm$ along~$\KK$ and consider the affine chart $\{z_0=1\} = 1 + \mm$ of projective space $\PP^{m-1} = \PP(A)$. Notice that $(\KK^\times\exp U) \cap (1+\mm) = \exp U$ is closed as an orbit of a unipotent group, which implies that $X \cap \{z_0\ne0\} = p(\exp U)$ is the open orbit. 
\end{construction}

\begin{example}
Let $n=1$, i.e. we are interested in induced $\GG_a$-actions on curves in~$\PP^{m-1}$. By Theorem~\ref{hypHaTsch_theorem}, classes of equivalence of such actions are in bijection with equivalence classes of pairs $(A,U)$, where $A$ is a local commutative algebra of dimension~$m$ with maximal ideal~$\mm$ and $U$ is a line in~$\mm$ generating the algebra~$A$.
Then we may assume that $A=\KK[S]/(S^m)$ and $U=\langle S\rangle$ up to automorphism of $A$. The pair $(A,U)$ corresponds to the additive action on the rational normal curve of degree $m-1$ in $\PP^{m-1}$. In particular, we have a unique class of equivalence of induced additive $\GG_a$-action on curves in~$\PP^{m-1}$. For explicit computations in the case $m=4$, see Example~\ref{me4} below. 
\end{example}

\subsection{The case of projective hypersurfaces: equations}
\label{subbus} 

In this subsection we obtain the equation of a projective hypersurface $X \subseteq \PP^{n+1}$ admitting an induced additive action in terms of the corresponding pair $(A, U)$, see Theorem~\ref{hypHaTsch_theorem} and Definition~\ref{hyp_Hpair_def}. First we consider a subvariety, not necessarily a hypersurface. The following proposition provides the condition that gives the open orbit $\exp U \subseteq 1 + \mm$, see Construction~\ref{hyp_constr}. 

\smallskip

By $\ln$ we mean the standard logarithm series $\ln(1+z) = \sum\limits_{k=1}^\infty \frac{(-1)^{k-1}}{k}z^k$, which is inverse to the exponential map $\exp$. Applying $\ln$ to $1+z$ with a nilpotent element $z$, we obtain a polynomial in~$z$. 

\begin{proposition}
\label{hypexpU_prop}
Let $A$ be a local commutative associative unital algebra of dimension~$m$ with maximal ideal $\mm$, and let $U \subseteq \mm$ be a subspace of dimension~$n$ generating the algebra~$A$. Denote by $\pi\colon \mm \to \mm / U$ the canonical projection of vector spaces. Then $\exp U$ in $1+\mm$ is given by the condition
\begin{equation}
\label{hypopen_eq}
\pi(\ln(1+z)) = 0
\end{equation}
for $1 + z \in A = \KK \oplus \mm$, $z \in \mm$. 
\end{proposition}

\begin{proof}
An element $1 + z \in 1 + \mm$ belongs to $\exp U$ if and only if $\ln(1+z) \in U$. 
\end{proof}

The above proposition helps us to find the subvariety $X \subseteq \PP^{m-1}$ corresponding to a given pair $(A, U)$. 

\begin{example} \label{me4}
Let $A = \KK[S]/(S^4)$ and $U = \langle S\rangle \subseteq \mm = \langle S, S^2, S^3\rangle$. Let $\pi\colon \mm \to \mm/U$ be a projection to $\langle S^2, S^3\rangle$ along $\langle S\rangle$. By Proposition~\ref{hypexpU_prop}, the set of points $z=z_1S + z_2S^2+z_3S^3$ that belong to $\exp U \subseteq 1+\mm$ is given by the condition
\begin{multline*}
\pi\left(z_1S + z_2S^2+z_3S^3 - \frac{(z_1S + z_2S^2+z_3S^3)^2}{2} + \frac{(z_1S + z_2S^2+z_3S^3)^3}{3} - \ldots\right) = \\
= \left(z_2-\frac{z_1^2}{2}\right)S^2 + \left(z_3 - z_1z_2 + \frac{z_1^3}{3}\right)S^3 = 0,
\end{multline*}
i.e. the open orbit of $X$ in the affine chart $\{z_0=1\}$ is given by the system 
$z_2-\frac{z_1^2}{2} = 0$, ${z_3 - z_1z_2 + \frac{z_1^3}{3} = 0}$, or, substituting the first equation to the second one, by the parametrization
\[[z_0:z_1:z_2:z_3] = \left[1 : z_1 : \frac{z_1^2}{2} : \frac{z_1^3}{6}\right] \subseteq \PP^3.\]
Taking the closure, we add one more (fixed) point $[0:0:0:1]$ and obtain a twisted cubic in $\PP^3$. 

Notice that the closure of the intersection of hypersurfaces may not be equal to the intersection of their closures, so $X$ may not be given by a system of homogenized equations. For example, in our case the system
$z_0z_2-\frac{z_1^2}{2} = 0$, $z_0^2z_3 - z_0z_1z_2 + \frac{z_1^3}{3} = 0$ gives outside the considered affine chart $\{z_0=1\}$ the projective line $z_0=z_1=0$, not a point. 
\end{example}

\begin{remark} \label{hyp_twisted_rem}
While the additive action from Example~\ref{hypnonind_ex} is not induced, it is the projection of the induced additive action from the above example along the coordinate~$z_2$. 
\end{remark}

Let us apply the above theory to the case of codimension one. Suppose that $X \subseteq \PP^{m-1}$ is a projective hypersurface that is not a hyperplane, 
$(A, U)$ is the corresponding pair, that is, $A$ is a local commutative associative unital algebra of dimension~$m$ with maximal ideal $\mm$, and $U \subseteq \mm$ is a subspace of dimension~$m-2$ generating the algebra~$A$. The next definition is taken from~\cite{ArSh2011}. 

\begin{definition} \label{hyp_Hpair_def}
The \emph{H-pair} corresponding to an induced additive action on a hypersurface~$X \subseteq \PP^{n+1}$ is the corresponding pair $(A, U)$, where $A$ is a local commutative associative unital algebra of dimension~$n+2$ with maximal ideal $\mm$ and $U \subseteq \mm$ is a subspace of dimension~$n$ generating the algebra~$A$. 
\end{definition}

Let us proceed with a proof of the following technical lemma. 
\begin{lemma} \label{hyp_zd_lem}
Suppose $\mm$ is the maximal ideal of a local commutative associative algebra~$A$ and $U$ is a subspace of~$\mm$. Then $\mm^d \subseteq U$ if and only if $z^d \in U$ for all $z \in \mm$. 
\end{lemma}

\begin{proof}
Let $f_d(t) = \sum_{k=0}^d z_kt^k$ be a polynomial with coefficients in~$A$ and $f_d(t) \in U$ for any $t \in \KK$. Let us show that $z_0, z_1, \ldots, z_d \in U$. First, $z_0 = f_d(0) \in U$. Then for a polynomial $f_{d-1}(t) = \sum_{k=1}^d z_kt^{k-1}$ we have $f_{d-1}(t) = \frac{f_d(t) - f_d(0)}{t} \in U$ for any $t \in \KK^\times$. Note that for $t = 0$ we also have $f_{d-1}(t) \in U$ since the set $\{t \in \KK \mid f_{d-1}(t) \in U\}$ is closed in~$\KK$. So $z_1 = f_{d-1}(0) \in U$. Arguing as above for polynomials~$f_{d-1}(t)$ of degree~$d - 1$, $f_{d-2}(t)$ of degree~$d-2$, \ldots, $f_0(t)$ of degree~$0$, we finally obtain that $z_0, z_1, \ldots, z_d \in U$. 

Now let $z^d \in U$ for any $z \in \mm$. Then $f(t_1, \ldots, t_d) = (t_1z_1 + \ldots + t_dz_d)^d \in U$ for any $t_1, \ldots, t_d \in \KK$ and $z_1, \ldots z_d \in \mm$. Fixing any $t_2, \ldots, t_d \in \KK$ and applying the above, we obtain that all coefficients of $f$ as a polynomial in $t_1$ belong to $U$ for any $t_2, \ldots, t_d \in \KK$. Consider these coefficients with fixed $t_3, \ldots, t_d$ as polynomials in~$t_2$ and obtain that they also belong to~$U$ for any $t_3, \ldots, t_d \in \KK$. Finally, we see that all coefficients of~$f$ belong to~$U$. In particular, the coefficient $d!z_1\ldots z_d$ at $t_1\ldots t_d$ is an element of~$U$, so $\mm^d \subseteq U$. The converse is immediate. 
\end{proof}

If $f(1+z)=0$, $z \in \mm$, is the equation of degree~$d$, defining the open orbit $\exp U$ in the affine chart $1+\mm$, then $X \subseteq \PP^{n+1}$ is given by the homogenization $h\!f$ of the polynomial~$f$: $h\!f(z_0+z) = z_0^d f(1+\frac{z}{z_0})$, $z_0 \in \KK$, $z \in \mm$. In particular, the degree of the projective hypersurface~$X$ equals the degree~$d$ of the affine hypersurface $\exp U \subseteq 1+\mm$. 

\smallskip 

The first statement in the theorem below is proved in~\cite[Theorem~5.1]{ArSh2011}. 

\begin{theorem}
\label{hypdeg_theorem}
Let $X \subseteq \PP^{n+1}$ be a projective hypersurface admitting an induced additive action, and $(A, U)$ be the corresponding H-pair. Denote by $\pi\colon \mm \to \mm / U \cong \KK$ the canonical projection. Then 

1) the degree of the hypersurface $X$ equals the maximal exponent $d$ with $\mm^d \nsubseteq U$; 

2) $X$ is given by the homogeneous equation of degree~$d$
\begin{equation}
\label{hyp_eq}
z_0^d \, \pi\left(\ln\left(1+\tfrac{z}{z_0}\right)\right) = 0
\end{equation}
for $z_0 + z \in A = \KK \oplus \mm$, $z_0 \in \KK$, $z \in \mm$. 
\end{theorem}

\begin{proof}
By Proposition~\ref{hypexpU_prop}, the open orbit in the affine chart $\exp U \subseteq 1+\mm$ is given by the polynomial $f(1+z) = \pi(\ln(1+z))$, $z \in \mm$. By definition of the number $d$, we have $\mm^k \subseteq U$ for all $k > d$. It follows that $f$ is of degree at most $d$ since $\pi$ takes all the summands in logarithm series with exponents greater than~$d$ to zero. On the other hand, $\mm^d \nsubseteq U$, so by Lemma~\ref{hyp_zd_lem} there exists $z \in \mm$ with $\pi(z^d) \ne 0$. Thus, the degree of the polynomial~$f$ equals~$d$.

Let us prove that $f$ is irreducible. Since $\mm^d \nsubseteq U$, we have $\mm^d \cap U \subsetneq \mm^d$. Since the codimension of~$U$ in~$\mm$ equals one, the codimension of $\mm^d \cap U$ in~$\mm^d$ is at most one, whence according to the above we can consider a decomposition $\mm^d = (\mm^d \cap U) \oplus \langle S \rangle$ for some vector $S \in \mm^d$. Since $S \notin U$ and $U$ in~$\mm$ is of codimension one, we also have $\mm = U \oplus \langle S \rangle$ in this case. Let $z=z_U+z_{n+1}S$, $z_U \in U$, $z_{n+1} \in \KK$, be the corresponding decomposition of $z \in \mm$. Then
\[\pi(\ln(1+z)) = \pi\left(\sum\limits_{k=1}^d\tfrac{(-1)^{k-1}}{k}(z_U+z_{n+1}S)^k\right),\]
and one can see that the map $\pi$ takes all~$z_{n+1}$ to zero except for $z_{n+1}$ in the summand with $k=1$ since $\mm S \subseteq \mm^{d+1} \subseteq U$. So the variable $z_{n+1}$ appears in the polynomial $f$ only in the linear term, whence $f$ is irreducible. 

Thus, $\exp U$ is given by the irreducible polynomial~$f$ of degree~$d$, whence the degree of a hypersurface~$X$ equals~$d$, and $X$ is given by the homogenization~$h\!f$ as in~\eqref{hyp_eq}. 
\end{proof}

\begin{example} \label{hyp_no30_ex}
Let $A = \KK[S_1,S_2,S_3]/(S_1^2,S_2^2,S_1S_3,S_2S_3,S_1S_2-S_3^3)$ be the 6-dimensional algebra no.~30 from Table~\ref{table_localg6}. Notice that $A = \langle 1,S_1,S_2,S_3,S_3^2,S_3^3=S_1S_2\rangle$, and consider $U = \langle S_1,S_2,S_3,S_3^2\rangle \subseteq \mm$. Since $\mm^3=\langle S_3^3\rangle \nsubseteq U$ and $\mm^4=0$, the H-pair $(A,U)$ corresponds to an induced additive action on a cubic hypersurface $X \subseteq \PP^5$. According to~\eqref{hyp_eq} for the projection $\pi\colon \mm \to \langle S_3^3\rangle$ along $U$, the left side of the equation of $X$ is
\[z_0^3\pi\left(\ln\left(1+\tfrac{z_1}{z_0}S_1+\tfrac{z_2}{z_0}S_2+\tfrac{z_3}{z_0}S_3+\tfrac{z_4}{z_0}S_3^2+\tfrac{z_5}{z_0}S_3^3\right)\right) = z_0^3\left(\tfrac{z_5}{z_0} - \tfrac12\cdot 2\tfrac{z_3}{z_0}\tfrac{z_4}{z_0}- \tfrac12\cdot 2\tfrac{z_1}{z_0}\tfrac{z_2}{z_0} + \tfrac13\tfrac{z_3^3}{z_0^3}\right),\]
which gives $X = \{z_0^2z_5 - z_0z_3z_4 - z_0z_1z_2 + \frac13z_3^3 = 0\}$. 
\end{example}

\begin{corollary} {\cite[Corollary~5.2]{ArSh2011}}
\label{hypdeg_cor}
If $X \subseteq \PP^{n+1}$ is a hypersurface of degree~$d$ admitting an induced additive action, then $d \le n+1$. 
\end{corollary}

\begin{proof}
Since $\mm \supsetneq \mm^2 \supsetneq \ldots$ and $\dim \mm = n + 1$, we have $\mm^{n+2} = 0 \subseteq U$.
\end{proof}

Let us illustrate the developed method by proving a variant of Proposition~\ref{newprop}. 

\begin{corollary} {\cite[Proposition~4]{AP2014}}
If $X$ is a smooth hypersurface admitting an induced additive action, then $X$ is a non-degenerate quadric or a hyperplane. 
\end{corollary}

\begin{proof}
As in the proof of Theorem~\ref{hypdeg_theorem}, let us choose a vector $S \in \mm^d \setminus U$ and obtain the decomposition of vector spaces $A = \KK \oplus \mm = \KK \oplus U \oplus \langle S\rangle$. For compatible coordinates $z_0, \ldots, z_{n+1}$, the variable $z_{n+1}$ appears in the equation $h\!f(z_0, \ldots, z_{n+1}) = 0$ of~$X$ only in the term $z_0^{d-1}z_{n+1}$ since $z_{n+1}$ appears in the polynomial $f$ only in a linear term. Thus, the point $[0:\ldots:0:1]$ lies on $X$ and is singular provided $d \ge 3$. It remains to note that the only smooth quadric is a non-degenerate one. 
\end{proof}

\begin{corollary}
If a hypersurface~$X$ of degree~$d$ admits an induced additive action and $(A, U)$ is the corresponding H-pair, then the complement in~$X \subseteq \PP(A)$ of the open orbit is \[p(\{z \in \mm \mid z^d \in U\}),\]
where $p\colon A\setminus\{0\} \to \PP(A)$ is the canonical projection.
\end{corollary}

\begin{proof}
According to Construction~\ref{hyp_constr}, the complement of the open orbit consists of points $z \in X$ with zero $z_0$-coordinate. The substituting $z_0 = 0$ to the equation~\eqref{hyp_eq} annihilates all summands of logarithm series except the last one of degree~$d$, so we obtain the equation $\pi\left(\frac{(-1)^d}{d!}(0 + z)^d\right) = 0$, or $z^d \in U$.
\end{proof}

\subsection{The case of projective hypersurfaces: invariant multilinear forms} \label{hyp_mult_subsec}
It is well known that quadratic forms $f(z)$ on a vector space $V$ are in one-to-one correspondence with bilinear maps $F\colon V \times V \to \KK$. If $X$ is a quadric given by the quadratic equation $f(z) = 0$, the corresponding bilinear form $F$ gives a lot of information on~$X$. Let us recall that in the same way any homogeneous polynomial $f(z)$ of degree $d$ corresponds to a $d$-linear symmetric form $F\colon \underbrace{V \times \ldots \times V}_{d} \to \KK$: any $d$-linear form $F$ gives a polynomial $f(z) = F(z, \ldots, z)$, and conversely $F(z^{(1)}, \ldots, z^{(d)})$ can be found from $f$ as a coefficient at $t_1\ldots t_d$ in the polynomial $f(t_1z^{(1)} + \ldots + t_dz^{(d)})$. This fact allows us to study hypersurfaces admitting induced additive actions in terms of multilinear forms. 

\smallskip

Suppose $(A, U)$ is an H-pair. Following~\cite[Section~4]{AP2014}, we call a $d$-linear form $F\colon \underbrace{A \times \ldots \times A}_{d} \to \KK$ \emph{invariant} if the following conditions hold:
\begin{itemize}
\item $F(1, \ldots, 1) = 0$;
\item for any $u \in U$, $z^{(1)}, \ldots, z^{(d)} \in A$, we have
\begin{equation} \label{hyp_invform_eq}
F(uz^{(1)}, z^{(2)}, \ldots, z^{(d)}) + F(z^{(1)}, uz^{(2)}, \ldots, z^{(d)}) + \ldots + F(z^{(1)}, z^{(2)}, \ldots, uz^{(d)}) = 0.
\end{equation}
\end{itemize}
Suppose that an H-pair $(A, U)$ corresponds to an induced additive action on a hypersurface $X \subseteq \PP^{n+1} = \PP(A)$ given by the polynomial $f$ of degree $d$ on~$A$. By the above, there is a $d$-linear form $F\colon \underbrace{A \times \ldots \times A}_{d} \to \KK$ corresponding to the polynomial $f$. It is an invariant $d$-linear form on $(A, U)$. 
Indeed, the first property follows from construction: we take $1$ as a cyclic vector in $A$, so $F(1,\ldots, 1) = f(1) = 0$. 
For equation~\eqref{hyp_invform_eq}, notice that $X = \{f(x) = 0\}$ is invariant with respect to the action of the group~$\GG_a^n \cong \exp U$, i.e. the polynomial~$f$ is semi-invariant. But the group $\GG_a^n$ has no non-trivial character, so $f$ is invariant with respect to~$\GG_a^n \cong \exp U$ and hence with respect to the Lie algebra~$U$, i.e. we have~\eqref{hyp_invform_eq}. The invariant $d$-linear form corresponding to a hypersurface $X \subseteq \PP(A)$ is defined up to scalar. Notice also that the number $d$ is determined by the pair $(A,U)$. 

\smallskip

Let $F$ be a $d$-linear form on a vector space $V$. Define

\smallskip

\begin{itemize}
\item $L^\perp = \{x \in V \mid F(x,z^{(2)}, \ldots, z^{(d)}) = 0 \;\; \forall z^{(2)},\ldots,z^{(d)} \in L\}$ for a subset~${L \subseteq V}$;
\item the kernel $\Ker F = V^\perp$.
\end{itemize}

\begin{lemma} \label{hyp_invform_lem}
Let $F\colon \underbrace{A \times \ldots \times A}_{d} \to \KK$ be an invariant $d$-linear form on an H-pair $(A, U)$. Then
\begin{enumerate}
  \item $U \subseteq 1^\perp$;
  \item $\Ker F$ is the maximal ideal of $A$ contained in~$U$. 
\end{enumerate}
\end{lemma}

\begin{proof}
(a) Follows from~\eqref{hyp_invform_eq} with $z^{(1)} = z^{(2)} = \ldots = z^{(d)} = 1$. 

\smallskip

(b) First let us prove that $\Ker F$ is an ideal of $A$. If $z \in \Ker F$ and $u \in U$, then $F(uz, z^{(2)}, \ldots, z^{(d)}) = -F(z, uz^{(2)}, \ldots, z^{(d)}) - \ldots - F(z, z^{(2)}, \ldots, uz^{(d)}) = 0$ for any $z^{(2)}, \ldots, z^{(d)} \in A$, so $uz \in \Ker F$ for any $u \in U$. Since $U$ generates $A$ as an algebra, it follows that $Az \subseteq \Ker F$. 

\smallskip

Now we are going to prove that $\Ker F \subseteq U$. Since $\Ker F$ is an ideal in~$A$ and $F$ is not equal to $0$, the kernel $\Ker F$ contains no invertible elements, i.e. $\Ker F \subseteq \mm$. Assume the converse, i.e. $\Ker F \nsubseteq U$. Since $\dim \mm = n+1$ and $\dim U = n$, it follows that $\mm = \Ker F + U$. 

Let us prove by induction on~$k$ that $F(u^{(1)}, \ldots, u^{(k)}, 1 \ldots, 1) = 0$ for any $u^{(1)}, \ldots, u^{(k)} \in U$. For $k=0$ we have $F(1,\ldots,1) = 0$. Let the assertion be proved for some~$k$. According to~\eqref{hyp_invform_eq}, we have
\[
\sum_{i=1}^k F(u^{(1)}, \ldots, u^{(k+1)}u^{(i)}, \ldots, u^{(k)}, 1, \ldots, 1) + \sum_{i=k+1}^d F(u^{(1)}, \ldots, u^{(k)}, 1, \ldots, \underset{i\;\;}{u^{(k+1)}}, \ldots, 1) = 0.
\]
All $d-k$ summands of the second sum equal $F(u^{(1)}, \ldots, u^{(k+1)}, 1, \ldots, 1)$ since $F$ is a symmetric form. For a summand of the first sum, one can decompose an element $u^{(k+1)}u^{(i)} \in \mm$ into $u^{(k+1)}u^{(i)} = z_i + u_i$, where $z_i \in \Ker F, \, u_i \in U$. Then the summand equals 
$F(u^{(1)}, \ldots, z_i, \ldots, u^{(k)}, 1, \ldots, 1) + F(u^{(1)}, \ldots, u_i, \ldots, u^{(k)}, 1, \ldots, 1)$, and the first one equals zero by the kernel condition, and the second one by induction hypothesis. Thus, $(d-k)F(u^{(1)}, \ldots, u^{(k+1)}, 1, \ldots, 1) = 0$, which completes the induction. 

Since $F$ is multilinear, it follows that $1 \in \langle1,U\rangle^\perp$. Moreover, $1 \in (\Ker F)^\perp$ and $\mm = \Ker F + U$, so $1 \in A^\perp = \Ker F \subseteq \mm$, a contradiction. 

\smallskip

It remains to prove the maximality. Let $J \subseteq U$ be an ideal of the algebra~$A$. Let us prove by induction on~$k$ that $F(z^{(1)}, \ldots, z^{(k)}, y, 1, \ldots, 1) = 0$
for any $y \in J$ and $z^{(1)}, \ldots, z^{(k)} \in A$. For $k = 0$ according to~\eqref{hyp_invform_eq} we have $\sum_{k=1}^d F(1, \ldots, \underset{i}{y}, \ldots, 1) = 0$ since $y \in J \subseteq U$, which gives $F(y, 1, \ldots, 1) = 0$ since $F$ is a symmetric form. Suppose that the assertion is proved for~$k-1$. Then
\[\sum_{i=1}^{k} F(z^{(1)}, \ldots, yz^{(i)}, \ldots, z^{(k)}, 1, \ldots, 1) + 
\sum_{i=k+1}^d F(z^{(1)}, \ldots, z^{(k)}, 1, \ldots, \underset{i}{y}, \ldots, 1) = 0\]
All $d-k$ summands of the second sum equal $F(z^{(1)}, \ldots, z^{(k)}, y, 1, \ldots, 1)$ since $F$ is symmetric. A summand of the first sum equals zero by induction hypothesis since $yz^{(i)} \in J$. Thus, $F(z^{(1)}, \ldots, z^{(k)}, y, 1, \ldots, 1) = 0$, which completes the induction. 

For $k=d-1$ we obtain $F(z^{(1)}, \ldots, z^{(d-1)}, y) = 0$ for any $z^{(1)}, \ldots, z^{(d-1)} \in A$ and $y \in J$. It follows that $y \in \Ker F$, i.e. for any ideal $J \subseteq U$ of $A$ we have $J \subseteq \Ker F$. 
\end{proof}

\smallskip

Now let us define the reduction of an induced additive action. First, return to the general case of a projective subvariety which is not necessarily a hypersurface. 

\begin{proposition}
\label{hypAUreduce_prop}
Let a pair $(A, U)$ correspond to an induced additive action on a projective subvariety $X \subseteq \PP(A)$. 
Suppose there is an ideal $J$ of $A$ such that $J \subseteq U$. Then the pair $(A/J, U/J)$ corresponds to an induced additive action on a projective subvariety $X_0 \subseteq \PP(A/J)$, and $X$ is the projective cone over $X_0$. In other words, if we choose coordinates in $\PP^{n+1} = \PP(A)$ compatible with inclusions $A = \KK \oplus \mm \supseteq \mm \supseteq J$, then $X$ does not depend on coordinates in $J$. 
Moreover, the additive action on $X$ is coherent with the additive action on $X_0$, i.e. the following diagram is commutative for the projection $\phi\colon A \to A/J$: 
\[
\begin{diagram}
\node{\exp U \times A} \arrow[3]{e,t}{\text{mult. in } A} \arrow{s,r}{\phi\times\phi} \node[3]{A} \arrow{s,r}{\phi} \\
\node{\exp(U/J) \times (A/J)} \arrow[3]{e,t}{\text{mult. in } A/J} \node[3]{A/J}
\end{diagram}
\]
\end{proposition}

\begin{proof}
If $J$ is an ideal in a local commutative unital algebra $A = \KK \oplus \mm$ with the maximal ideal~$\mm$, then $A/J = \KK \oplus (\mm/J)$ is a local commutative unital algebra with the maximal ideal~$\mm/J$. Since the subspace $U \subseteq \mm$ generates the algebra~$A$, it follows that the subspace~$U/J$ generates the algebra~$A/J$. 
Fix some decomposition $\mm = J \oplus \mm'$, and for $z \in \mm$ let $z = z_J + z'$. Let us find the equation of~$X$ according to Proposition~\ref{hypexpU_prop}. Since $J$ is an ideal of~$A$, we have
\[\ln(1+z) = \sum\limits_{k=1}^\infty\tfrac{(-1)^{k-1}}{k}(z_J+z')^k \in \ln(1+z') + J \subseteq \ln(1+z') + U,\]
i.e. $\pi(\ln(1+z))$ does not depend on coordinates in~$J$. 

For coherency, fix a decomposition $U = J \oplus U'$ and let the projection $\phi\colon A \to A/J$ be the projection on $U'$ along $J$. For $u = u_J + u' \in U = J \oplus U'$, notice that \[\exp(u) = \sum\limits_{k=1}^\infty \tfrac{1}{k!}(u_J+u')^k \in \exp(u')+J\] since $J$ is an ideal in~$A$, so the projection $\phi(\exp U) = \exp U'$. Consider any $a \in A$, then 
$\exp(u)a = \sum\limits_{k=1}^\infty \tfrac{1}{k!}(u_J+u')^ka \in \exp(u')a + J = (\exp(u')+J)(a+J)$,
which proves the required commutativity of the diagram.
\end{proof}

Proposition~\ref{hypAUreduce_prop} motivates the following definition; see~\cite[Section 4]{ArSh2011}. 

\begin{definition}
The induced additive action corresponding to a pair $(A, U)$ is \emph{reducible} to the induced additive action corresponding to a pair $(A', U')$ if there exists an algebra homomorphism $\phi\colon A \to A'$ with $\phi(U) = U'$ and $\codim_A U = \codim_{A'}U'$. 
\end{definition}

In such a case, $\phi$ is surjective since $U'$ generates~$A'$. So there exists an ideal~$J = \Ker \phi$ of the algebra~$A$ such that $J \subseteq U$ and the factorization $A \to A/J \cong A'$ maps $U$ to~$U'$, i.e. we are in the situation of Proposition~\ref{hypAUreduce_prop}.

\begin{definition} \label{hyp_nondeg_def}
Suppose a projective hypersurface $X \subseteq \PP(V)$ of degree $d$ is given by the equation $f(z_1, \ldots, z_n) = 0$ and $F$ is the corresponding $d$-linear form. A hypersurface $X$ is called \emph{non-degenerate} if one of the following equivalent conditions hold:
\begin{itemize}
  \item $\Ker F = 0$;
  \item $\frac{\pa f}{\pa z_1}, \ldots, \frac{\pa f}{\pa z_n}$ are linearly independent $(d-1)$-linear forms; 
  \item there is no linear transform of variables reducing the number of variables in $f$.
\end{itemize}
\end{definition}

\begin{corollary}
Any induced additive action on a hypersurface is reducible to an induced additive action on a non-degenerate hypersurface. More precisely, an induced additive action corresponding to the H-pair $(A,U)$ is reducible to the induced additive action corresponding to the H-pair $(A/\Ker F, \, U/\Ker F)$, where $F$ is the invariant multilinear form of~$X$. 
\end{corollary}

\begin{proof}
Follows from Proposition~\ref{hypAUreduce_prop} and Lemma~\ref{hyp_invform_lem}(b).
\end{proof}

\medskip

In~\cite[Lemma~1]{AP2014}, an explicit formula for the form $F$ corresponding to a pair $(A, U)$ is obtained. Since $A = \KK \oplus \mm$ and $F$ is multilinear, it is sufficient to define $F$ for arguments that belong to $\mm$ or equal $1$. Let $\pi\colon \mm \to \mm/U \cong \KK$ be the canonical projection. Then 
\begin{equation} \label{hyp_Fpi_eq}
  F(z^{(1)}, z^{(2)}, \ldots, z^{(d)}) = (-1)^k k! (d-k-1)! \pi(z^{(1)} \ldots z^{(d)}),
\end{equation}
where $k$ is the amount of~$1$ along $z^{(1)}, z^{(2)}, \ldots, z^{(d)}$, other arguments belong to~$\mm$, and for $k=d$ we let $F(1, \ldots, 1) = 0$. 
One can check that this agrees with the equation $f(z_0+z) = z_0^d\pi\left(\ln\left(1+\frac{z}{z_0}\right)\right) = 0$ obtained in Theorem~\ref{hypdeg_theorem}. Indeed, for a multilinear form $F$ defined by~\eqref{hyp_Fpi_eq} and any $z_0+z \in A = \KK\oplus\mm$ we have
\begin{multline*}
F(z_0+z, \ldots, z_0+z) = \sum_{k=0}^d \binom{d}{k} F(\underbrace{z_0, \ldots, z_0}_{k}, \underbrace{z, \ldots, z}_{d-k}) = \\
= \sum_{k=0}^d z_0^k \frac{d!}{k!(d-k)!} (-1)^k k! (d-k-1)! \pi(z^{d-k}) = d!(-1)^d \ z_0^d \pi\left(\ln\left(1+\tfrac{z}{z_0}\right)\right). 
\end{multline*}

\subsection{The case of quadrics and cubics}
\label{subqgen} 

Let us start with the following well known fact. 

\begin{lemma}
\label{hypPSO_lem}
The automorphism group of the non-degenerate quadric $Q_n \subseteq \PP^{n+1}$ is $\PSO_{n+2}(\KK)$. 
\end{lemma}

\begin{proof}
Let us notice that the Picard group of the quadric $\Pic Q_n \cong \ZZ$ is generated by the line bundle $\Of(1)$ for $n \ge 3$. Any automorphism of $Q_n$ induces the automorphism of the Picard group, which can bring the generator $\Of(1)$ either to $\Of(1)$ or to $\Of(-1)$. The last case is impossible since $\Of(-1)$ has no global section, so any hyperplane section of $Q_n$ is mapped to a hyperplane section. The last assertion holds for $n=1,2$ as well. Thus, any automorphism of $Q_n$ corresponds to the transformation of $\PP(H^0(\Of(1))) = (\PP^{n+1})^*$. The dual of this transformation is the extension of the initial automorphism of $Q_n$ to the automorphism of $\PP^{n+1}$. 
\end{proof}

The following theorem answers Question~3.1 (4) in~\cite{HaTs1999}. 

\begin{theorem} {\cite[Theorem~4]{Sh2009}} \label{hyp_quadr_nondeg_theor}
Let $Q_n$ be a non-degenerate quadric in $\PP^{n+1}$. Then there is a unique additive action on~$Q_n$ up to equivalence. It corresponds to the H-pair $(A_n, U_n)$, where 
\begin{gather*}
A_n = \KK[S_1, \ldots, S_n] / (S_i^2 - S_j^2, S_iS_j, i \ne j) \text{ and } U_n = \langle S_1, \ldots, S_n\rangle \text{ if } n \ge 2;\\
A_1 = \KK[S_1]/(S_1^3) \text{ and } U_1 = \langle S_1\rangle.
\end{gather*} 
\end{theorem}

\begin{proof}
By Lemma~\ref{hypPSO_lem}, any additive action on~$X$ is induced. Let $(A, U)$ be the corresponding H-pair, i.e. $A$ is a local algebra of dimension~$n+2$ and $U \subseteq \mm$ is a subspace of dimension~$n$ generating the algebra~$A$. 
By the first statement of Theorem~\ref{hypdeg_theorem}, $\mm^2 \nsubseteq U$ and $\mm^3 \subseteq U$. Since $Q_n$ is non-degenerate, the corresponding multilinear form has the trivial kernel, see Definition~\ref{hyp_nondeg_def}. By Lemma~\ref{hyp_invform_lem}(b), it follows that there is no nonzero ideal of $A$ in $U$. Then $\mm^3 \subseteq U$ implies $\mm^3=0$. It follows that $\mm^2 \cap U$ is an ideal in $U$, so $\mm^2 \cap U = 0$ as well. Since $\mm^2 \nsubseteq U$ implies $\mm^2 \ne 0$, we have $\mm = U \oplus \mm^2$ and $\dim \mm^2 = 1$ by dimension reasons. 

It remains to prove that there is a unique pair $(A, U)$ satisfying the above conditions. Notice that the multiplication in the algebra $A = \KK \oplus U \oplus \mm^2$ is defined by the restriction $B\colon U \times U \to \mm^2$. Indeed, $U \cdot U \subseteq \mm^2$ since $U \subseteq \mm$, $U \cdot \mm^2 = 0$ and $\mm^2\cdot\mm^2 = 0$ since $\mm^3 = 0$, and $1 \cdot x = x$ for any $x \in A$. Since $\dim \mm^2 = 1$, it follows that $B$ is a bilinear form on $U$, and now we are going to prove that this form is non-degenerate. A non-degenerate bilinear form on a vector space is unique up to a linear change of variables, so the non-degeracy of the form $B$ will show the uniqueness of the pair $(A,U)$ and the corresponding additive action. 

In our situation, the left side of equation~\eqref{hyp_eq} of the quadric~$Q_n$ turns into 
${z_0^2\,\pi\!\left(\ln\left(1 + \frac{z}{z_0}\right)\right) = z_0^2\,\pi\!\left(\frac{z}{z_0} - \frac12 \frac{z^2}{z_0^2}\right) = z_0 \pi(z) - \frac12 \pi(z^2)}$, where $\pi\colon \mm \to \mm/U \cong \KK$ is a projection. Recall that $\mm = U \oplus \mm^2$, so $\pi$ can be chosen as the projection $\pi\colon \mm \to \mm^2$ along~$U$. For $z \in \mm$, denote $z = z_U + z_{\mm^2}$, where $z_U \in U$, $z_{\mm^2} \in \mm^2$. Then $z^2 = z_U^2$ since $\mm^3 = 0$, so we obtain that the equation $z_0 \pi(z) - \frac12 \pi(z^2) = 0$ of~$Q_n$ turns into
\[z_0 z_{\mm^2} - \frac12 B(z_U, z_U) = 0.\] 
It defines a non-degenerate quadric if and only if the form $B$ is non-degenerate, so we come to the desired uniqueness. 

Now it is easy to calculate the pair $(A, U)$. Denote by $S$ a basis of $\mm^2$, and let $S_1, \ldots, S_n$ be a basis of $U$ such that $B(z_U, z_U) = (z_1^2 + \ldots + z_n^2)S$ for $z_U = z_1S_1 + \ldots z_nS_n$. By definition of $B$, it follows that $S_i^2 = S_j^2 = S$ and $S_iS_j = 0$ for $i \ne j$, and $\mm S = 0$ since $S \in \mm^2$. Thus, the algebra $A$ is isomorphic to the required $\KK[S_1, \ldots, S_n] / (S_i^2 - S_j^2, S_iS_j, i \ne j)$ if $n\ge 2$ and $\KK[S_1]/(S_1^3)$ if $n=1$. 
\end{proof}

The next result was observed in~\cite[Proposition~4.2]{ArSh2011}.

\begin{corollary}
An H-pair $(A, U)$ corresponds to an induced additive action on a quadric if and only if there exists a homomorphism of H-pairs $(A, U) \to (A_n, U_n)$. 
\end{corollary}

Denote projective quadrics in $\PP^{n+1}$ by
\[Q(n,k)=\{[z_0:\ldots:z_{n+1}] \mid q(z_0, \ldots, z_{n+1}) = 0\},\]
where $q$ is a quadratic form of rank $k+2$ with $1 \le k \le n$. In this notation, the non-degenerate quadric $Q_n \subseteq \PP^{n+1}$ is $Q(n, n)$. 

In~\cite{BGT2020}, the authors obtain a generalization of Hassett-Tschinkel correspondence for induced actions of commutative linear algebraic groups on non-degenerate quadrics with an open orbit. In~\cite[Theorem~3]{BGT2020} it is proved that besides the unique additive action from Theorem~\ref{hyp_quadr_nondeg_theor} there are only three cases: $\GG_m$-action on~$Q_1$, $\GG_a \times \GG_m$-action on~$Q_2$, and $\GG_m^2$-action on~$Q_2$. 

\medskip

At the same time, additive actions on degenerate quadrics are not unique. In particular, there is an infinite family of pairwise non-equivalent induced additive actions on quadrics~$Q(n, n-1)$ for $n \ge 4$; see~\cite[Section~4]{ArSh2011}. 

\def\rddots{\text{\reflectbox{$\ddots$}}}

\begin{proposition} {\cite[Proposition~7]{AP2014}} \label{hyp_quadr_cork1_prop}
The H-pairs corresponding to induced additive actions on quadrics $Q(n,n-1) \subseteq \PP^{n+1}$ are:
\begin{enumerate}
\smallskip
\item $A = \raisebox{0.5\baselineskip}{$\KK[S_1, \ldots, S_n]$}\scalebox{2}{/}\begin{pmatrix}S_iS_j - \lambda_{ij}S_n, \, S_i^2 - S_j^2 - (\lambda_{ii} - \lambda_{jj})S_n, \ 1 \le i < j \le n-1 \\ S_lS_n, \ 1 \le l \le n\end{pmatrix}$, 
\medskip
where $n \ge 3$ and $\lambda_{ij}$ are the elements of a symmetric $(n-1)\times(n-1)$ block-diagonal matrix~$\Lambda$ such that each block $\Lambda_l$ is
\[\begin{pmatrix}
\lambda_l & 1 & 0 & \ldots & 0 \\ 1 & \lambda_l & 1 & \ddots & \vdots \\ 0 & 1 & \lambda_l & \ddots & 0 \\
\vdots & \ddots & \ddots & \ddots & 1 \\ 0 & \ldots & 0 & 1 & \lambda_l
\end{pmatrix} + 
\begin{pmatrix}
0 & \ldots & 0 & \imath/2 & 0 \\
\vdots & \rddots & \imath/2 & 0 & -\imath/2 \\
0 & \rddots & 0 & -\imath/2 & 0 \\
\imath/2 & \rddots & \rddots & \rddots & \vdots \\
0 & -\imath/2 & 0 & \ldots & 0
\end{pmatrix}, \; \imath^2=-1,\]
and $U = \langle S_1, \ldots, S_n\rangle$;
\smallskip
\item $\KK[S_1, S_2]/(S_1^3, S_1S_2, S_2^2)$, $U = \langle S_1, S_2\rangle$;
\smallskip
\item $\KK[S_1]/(S_1^4)$, $U = \langle S_1, S_1^3\rangle$.
\end{enumerate}

The matrix $\Lambda$ is defined up to permutation of blocks, scalar multiplication, and adding a scalar matrix. 
\end{proposition}

In~\cite[Section~4]{ArSh2011}, an explicit description of the actions for $n=4$ is also given.

Using the same techniques, a classification of additive actions on quadrics of corank two having at least one singular point that is not fixed by the $\GG_a^n$-action is given in~\cite{Liu2022}. 

\begin{remark}\label{rsha}
In~\cite[Section~5]{Shaf2021}, a classification of additive actions on the quadrics of small dimensions is given. It is proved that the surface $Q(2,1)$ admits two additive actions (cf. Proposition~\ref{hyp_quadr_cork1_prop}, (b)-(c)), the 3-folds $Q(3,1)$ and $Q(3,2)$ admit seven and three additive actions, respectively. The number of additive actions on $Q(4,1)$ is finite, but this number is at least $25$. Finally, there are infinitely many additive actions on $Q(4,2)$. The classification is given in terms of H-pairs. 
\end{remark}

In~\cite{Ba2013}, the case of cubic hypersurfaces is studied and the following theorem is proved. 

\begin{theorem} \label{tbazhov} 
A cubic hypersurface $\{f = 0\}$ in $\PP^{n+1}$ admits an induced additive action if and only if for some $1 \le k \le s \le n - k$ one can choose homogeneous coordinates
$z_0, z_1, \ldots , z_s, w_0, w_1, \ldots, w_{n-s}$ in $\PP^{n+1}$ such that the polynomial $f$ has the form
\[f = z_0^2w_0 + z_0(z_1w_1 + \ldots + z_kw_k) + z_0(z_{k+1}^2 + \ldots + z_s^2) + g(z_1, \ldots , z_k),\]
where $g$ in a non-degenerate cubic form in $k$ variables. Moreover, an induced additive action is unique if and only if the hypersurface is non-degenerate, i.e. $k + s = n$.
\end{theorem}

\subsection{Non-degenerate hypersurfaces and Gorenstein algebras} Let $A$ be a Gorenstein local algebra. The socle of $A$ is equal to the one-dimensional ideal $\mm^d$. A hyperplane $U$ in $\mm$ is called \emph{complementary} if $\mm=U\oplus\mm^d$. 

\begin{theorem}
\label{hypGor_prop}
Induced additive actions on non-degenerate hypersurfaces~$X$ of degree~$d$ in $\PP^{n+1}$ are in bijection with H-pairs $(A, U)$, where $A$ is a Gorenstein algebra of dimension $n+2$ with the socle $\mm^d$ and $U$ is a complementary hyperplane. 
\end{theorem}

\begin{proof}
Let us observe first that a subspace $U$ in the maximal ideal $\mm$ of a local algebra $A$ contains no nonzero ideal of $A$ if and only if $(\Soc A)\cap U=0$. The `only if' part is clear since $(\Soc A)\cap U$ is an ideal of~$A$ in~$U$. 
Conversely, if $J$ is a nonzero ideal of $A$ in $U$, then for a maximal $s$ such that $J\mm^s$ is nonzero we have $J\mm^s\subseteq(\Soc A)\cap J \subseteq (\Soc A)\cap U$. 

This shows that if $(A,U)$ is the H-pair corresponding to an induced additive action on a non-degenerate hypersurface in $\PP^{n+1}$, then $\mm=U \oplus (\Soc A)$ and $A$ is Gorenstein by dimension reasons. 

Conversely, let $A$ be a Gorenstein local algebra of dimension $n+2$ and $U$ be a complementary hyperplane in $\mm$. Since $A=\KK\oplus U\oplus(\Soc A)$, we conclude that $U$ generates~$A$, so $(A,U)$ is an H-pair corresponding to an induced additive action on a hypersurface~$X$ in~$\PP^{n+1}$. Since $(\Soc A) \cap U = 0$, it follows that $U$ contains no nonzero ideal, so the hypersurface $X$ is non-degenerate. 

Finally, we notice that by Theorem~\ref{hypdeg_theorem} the degree of $X$ is $d$, where $\Soc A = \mm^d$. 
\end{proof}

\begin{remark}
We do not have an example where two different complementary hyperplanes $U$ and $U'$ in the maximal ideal~$\mm$ of the same Gorenstein local algebra~$A$ give rise to additive actions on non-isomorphic hypersurfaces. 
\end{remark}

As an application of Theorem~\ref{hypGor_prop}, we see from Table~1 that for $n\le 5$ there are induced additive actions on non-degenerate hypersurfaces in $\PP^n$ of all degrees from $2$ to $n$. Moreover, in $\PP^5$ there are three types of non-degenerate cubic hypersurfaces that come from different Gorenstein algebras. 

\smallskip

Now let us prove a generalization of the uniqueness of an additive action on non-degenerate quadrics and cubics.

\begin{theorem} \label{tnew}
Let $X \subseteq \PP^{n+1}$ be a non-degenerate hypersurface. Then there is at most one induced additive action on~$X$ up to equivalence.
\end{theorem}

\begin{proof}
Let $(A, U)$ be the H-pair corresponding to an induced additive action on $X$. The hypersurface $X$ defines the corresponding invariant multilinear form~$F$, see subsection~\ref{hyp_mult_subsec}. We have to prove that $F$ uniquely defines the pair $(A, U)$ up to equivalence, see Theorem~\ref{hypHaTsch_theorem}. 

Denote by $d$ the degree of the hypersurface~$X$. Since $X$ is non-degenerate, by Theorem~\ref{hypGor_prop} the algebra~$A$ is Gorenstein, $\mm^{d+1}=0$, $\mm = U \oplus \mm^d$, and $\dim \mm^d = 1$. Since $1 \cdot a = a$ for any $a \in A$ and $\mm^{d}\cdot\mm = 0$, it remains to define the multiplication of elements in~$U$. Let $\pi\colon \mm \to \mm^d \cong \KK$ be the canonical projection along~$U$, and $B\colon U \times U \to \mm^d \cong \KK$ be a bilinear map defined by the formula $B(u_1, u_2) = \pi(u_1u_2)$. 

Let us prove that $\Ker B$ is an ideal in the algebra $A$. If $u \in \Ker B$, then $B(u, u_2) = \pi(uu_2) = 0$ for any $u_2 \in U$, i.e. $uu_2 \in U$ for any $u_2 \in U$. Moreover, $uU \subseteq U$ implies $uA \subseteq U$ since $u \cdot 1 \in U$ and $u \cdot \mm^d = 0$. Then for any elements $a \in A$, $u_2 \in U$ we have $uau_2 \subseteq uA \subseteq U$, so $B(ua, u_2) = \pi(uau_2) = 0$. Thus, $ua \in \Ker B$ for any $a \in A$, so $\Ker B$ is an ideal of $A$. By Lemma~\ref{hyp_invform_lem}(b), $\Ker F$ is the maximal ideal of $A$ contained in $U$, so $\Ker B \subseteq \Ker F$. But $F$ is non-degenerate, so $B$ is non-degenerate as well. 

Denote by $S$ a nonzero vector in $\mm^d$, and let $S_1, \ldots, S_n$ be a basis of $U$ such that $B(z_U, z_U) = (z_1^2 + \ldots + z_n^2)S$ for $z_U = z_1S_1 + \ldots z_nS_n$. By definition of $B$, it follows that $S_i^2 - S \in U$ and $S_iS_j \in U$ for $i \ne j$. 
Denote $S_i^2 - S = \sum_l \alpha_{il}S_l$ and $S_iS_j = \sum_l \beta_{ijl}S_l$. According to~\eqref{hyp_Fpi_eq}, $F(z^{(1)}, \ldots, z^{(d)}) = (-1)^k k! (d-k-1)! \pi(z^{(1)}\ldots z^{(d)})$, where all arguments are either $1$ or nilpotent and $k$ is the number of~$1$. In particular, $F(z^{(1)}, z^{(2)}, 1, \ldots, 1) = (-1)^{d-2} (d-2)! \pi(z^{(1)}z^{(2)})$ for all $z^{(1)}, z^{(2)} \in \mm$, so the bilinear form $B$ is uniquely defined by the multilinear form~$F$. Notice that $B(S_i^2-S, S_l) = \alpha_{il}S$ and $B(S_iS_j, S_l) = \beta_{ijl}S$, whence the coefficients $\alpha_{il}$ and $\beta_{ijl}$, and consequently the products $S_i^2$ and $S_iS_j$, $i \ne j$, are uniquely defined by the form~$F$. Thus, $F$ defines the multiplication on~$U$. 
\end{proof}

\begin{example} \label{hyp_no30_2_ex}
Let $A = \KK[S_1,S_2,S_3]/(S_1^2,S_2^2,S_1S_3,S_2S_3,S_1S_2-S_3^3)$ be the 6-dimensional Gorenstein algebra no.~30 from Table~\ref{table_localg6} and $U = \langle S_1,S_2,S_3,S_3^2\rangle \subseteq \mm$, see Example~\ref{hyp_no30_ex}. Recall that $A = \langle 1,S_1,S_2,S_3,S_3^2,S_3^3=S_1S_2\rangle$ and the H-pair $(A,U)$ corresponds to an induced additive action on a cubic hypersurface $X \subseteq \PP^5$. Since $\Soc A=\mm^3$ and $\mm=U\oplus\mm^3$, we conclude that $X$ is non-degenerate. 
By Theorem~\ref{tnew}, there is a unique induced additive action on~$X$. One can write it down explicitly by Theorem~\ref{hypHaTsch_theorem}: identifying ${[z_0:z_1:z_2:z_3:z_4:z_5] \in \PP^5}$ with $z_0+z_1S_1+z_2S_2+z_3S_3+z_4S_3^2+z_5S_3^3 \in A$ and multiplying by 
$$
{\exp(\alpha_1S_1+\alpha_2S_2+\alpha_3S_3+\alpha_4S_3^2)} = 1 + \alpha_1S_1+\alpha_2S_2+\alpha_3S_3+\alpha_4S_3^2 + (\alpha_1\alpha_2 + \alpha_3\alpha_4 + \tfrac{\alpha_3^3}{6})S_3^3,
$$ 
we obtain that the action of $(\alpha_1,\alpha_2,\alpha_3,\alpha_4)\in\GG_a^4$ is given by
\begin{multline*}
[z_0:z_1+\alpha_1z_0 : z_2+\alpha_2z_0 : z_3+\alpha_3z_0: z_4+\alpha_3z_3+\alpha_4z_0 : \\ : z_5+\alpha_3z_4+\alpha_4z_3+\alpha_1z_2+\alpha_2z_1 + (\alpha_1\alpha_2 + \alpha_3\alpha_4 + \tfrac{\alpha_3^3}{6})z_0].
\end{multline*}
Finally, the linear change of coordinates 
$$
Z_0=-z_0, \ W_0=z_5, \ Z_1=z_3, \ W_1=z_4, \ Z_2=\frac{\imath(z_2-z_1)}{2}, \ Z_3=\frac{z_2+z_1}{2} 
$$
brings the equation of $X$ into the form of Theorem~\ref{tbazhov} with $k=1$, $s=3$ and $n=4$.  
\end{example}

%
\section{Additive actions on flag varieties}
\label{secaafv}
The main aim of this section is to describe additive actions on flag varieties $G/P$ of a semisimple linear algebraic group $G$. 

We begin with a brief overview of geometric properties of varieties $X$ admitting an additive action: such a variety is rational and it has a free finitely generated divisor class group that is generated by classes of boundary divisors. If the variety $X$ is complete then classes of boundary divisors generate freely the monoid of classes of effective divisors. Moreover, for $X$ smooth the anti-canonical class $-K_X$ is an integer linear combination of classes of boundary divisors with coefficients at least 2.  We also observe that if a commutative subgroup in the automorphism group of a variety $X$ acts on $X$ with an open orbit, then it is a maximal commutative subgroup in $\Aut(X)$. 

In subsection~\ref{sub32} we classify flag varieties $G/P$ which admit an additive action. It turns out that in this case the parabolic subgroup $P$ is maximal and the existence of an additive action on $G/P$ is almost equivalent to commutativity of the unipotent radical $P_u$ with few explicit exceptions. 

In subsection~\ref{sub33} we discuss the uniqueness result: if a flag variety $G/P$ is not isomorphic to a projective space, then $G/P$ admits at most one additive action. 

Finally, in subsection~\ref{sub34} we present a construction due to Feigin that allows to degenerate an arbitrary flag variety $G/P$ to a projective variety with an additive action. 

\subsection{Generalities on additive actions on complete varieties} 
In this subsection we briefly recall basic geometric properties of varieties admitting an additive action. 

Clearly, any variety $X$ with an additive action $\GG_a^n\times X\to X$ contains an open $\GG_a^n$-orbit that is isomorphic to an affine space. This implies that $X$ is a rational variety. 

It is well-known that the complement $X\setminus\Uf$ of an affine open subset $\Uf$ on an irreducible variety $X$ is a union $D_1\cup\ldots\cup D_k$ of prime divisors. If $X$ is a normal variety with an additive action and $\Uf$ is the open orbit, we call $[D_1],\ldots,[D_k]$ the \emph{boundary classes} in $\Cl(X)$. 

\begin{proposition} \label{adcl}
Let $X$ be a normal variety and $\Uf\subseteq X$ be an open subset isomorphic to an affine space. Then any invertible regular function on $X$ is constant and the divisor class group $\Cl(X)$ is a free finitely generated abelian group. Moreover, boundary classes form a basis in $\Cl(X)$. 
\end{proposition}

\begin{proof}
If $f$ is an invertible regular function on $X$, its restriction to $\Uf$ is invertible and regular as well. It follows that $f$ is constant on $\Uf$ and so on $X$. 

Since all divisors on $\Uf$ are principal, any divisor on $X$ is linearly equivalent to an integral linear combination of prime divisors in the complement $X\setminus\Uf$. If such a linear combination gives zero in $\Cl(X)$, then it is a principal divisor corresponding to some rational function $f$ on $X$. The function $f$ has neither zero nor pole on $\Uf$, so it is an invertible regular function on the affine space. Such a function is constant, so the combination is trivial. 
\end{proof}

\smallskip

As the next step, let us formulate a result proved in \cite[Theorems~2.5, 2.7]{HaTs1999}. 

\begin{theorem} \label{geHT}
Let $X$ be a complete normal variety with an additive action. Then the monoid of classes of effective divisors is generated freely by the boundary classes. Moreover, if $X$ is smooth then the anti-canonical class $-K_X$ is an integer linear combination of boundary classes, where all coefficients are $\ge 2$. 
\end{theorem}

To obtain the first statement of Theorem~\ref{geHT} one should use a linearization of an arbitrary divisor with respect to an action of a unipotent group. Namely, consider an effective divisor~$D$ on~$X$. The representation of $\GG_a^n$ on the projectivization of the space $H^0(X,\Of(D))$ has a fixed point that corresponds to an effective divisor supported at the boundary which is linearly equivalent to~$D$. The proof of the second statement is more delicate, it requires computations with vector fields and related exact sequences. 

\begin{remark}
If a complete normal variety $X$ admits an additive action, such a variety need not be projective. Examples of additive actions on smooth non-projective complete toric varieties~$X$
in any dimension starting from~$3$ are constructed in~\cite{Sha2020}.
\end{remark}

Let us proceed with one more observation. We say that a subgroup $H$ in the automorphism group $\Aut(X)$ of an algebraic variety $X$ is \emph{algebraic} if there is a structure of an algebraic group on $H$ such that the action $H\times X\to X$ is regular. Note that if the group $\Aut(X)$ itself has a structure of an algebraic group such that the action $\Aut(X)\times X\to X$ is regular, the notion of an algebraic subgroup coincides with the notion of an algebraic subgroup in an algebraic group.

\begin{proposition} \label{popor}
Assume that a commutative algebraic group $H$ acts effectively on an irreducible variety with an open orbit. Then the subgroup $H$ is a maximal (with respect to inclusion) commutative algebraic subgroup in the group $\Aut(X)$. 
\end{proposition}

\begin{proof}
Assume that $H$ is contained in a bigger commutative algebraic subgroup $F$ and let $g\in F\setminus H$. Then the element $g$ permutes $H$-orbits on $X$ and, in particular, it preserves the open $H$-orbit $\Uf$ on $X$. Take a point $x\in \Uf$. Then there is an element $h\in H$ such that $gx=hx$. This shows that $h^{-1}g$ fixes $x$. So $h^{-1}g$ acts identically on $\Uf$ and on $X$, a contradiction. 
\end{proof} 

This result shows that each additive action $\GG_a^n\times X\to X$ provides a maximal commutative unipotent subgroup in the automorphism group $\Aut(X)$. In particular, the Hassett-Tschinkel correspondence allows to construct many non-conjugate maximal commutative unipotent subgroups of dimension $n$ in the group $\GL_{n+1}(\KK)$. At the same time, there are maximal commutative unipotent subgroups in $\GL_{n+1}(\KK)$ of other dimensions. 

In the next subsections we study additive actions on (generalized) flag varieties, i.e. on homogeneous spaces $G/P$ of a connected semisimple group $G$ modulo a parabolic subgroup~$P$. The following proposition provides an additional motivation to concentrate on varieties of this type.

It is well known that the connected component $\Aut(X)^0$ of the automorphism group of a complete variety is a connected linear algebraic group. In view of importance of the problem of the existence of a K\"ahler-Einstein metric, the cases when the group $\Aut(X)^0$ is reductive are of particular interest. 

\begin{proposition} {\cite[Proposition~1]{AP2014}}
Let $X$ be a complete variety admitting an additive action. Assume that the group $\Aut(X)^0$ is a reductive linear algebraic group. Then $X$ is a flag variety $G/P$ for some semisimple group $G$ and some parabolic subgroup $P$. 
\end{proposition}

\begin{proof}
Let $X'$ be the normalization of $X$. The action of $\Aut(X)^0$ on $X$ can be lifted to $X'$. This implies that some commutative unipotent group acts on $X'$ with an open orbit. In particular, a maximal unipotent subgroup of the reductive group $\Aut(X)^0$ acts on $X'$ with an open orbit. It means that $X'$ is a spherical variety of rank zero, see~\cite[Section~1.5.1]{Ti2011} for details. It yields that $X'$ is a flag variety $G/P$, see~\cite[Proposition~10.1]{Ti2011}, and $\Aut(X)^0$ acts on $X'$ transitively. This implies $X=X'$. 
\end{proof}

\subsection{Existence of an additive action on a flag variety}
\label{sub32}
In this subsection we follow~\cite{Ar2011} and classify flag varieties that admit an additive action. 

Let $G$ be a connected semisimple linear algebraic group of adjoint type over an algebraically closed field of characteristic zero, and $P$ be a parabolic subgroup of $G$. The homogeneous space $G/P$ is called a \emph{(generalized) flag variety}. Recall that the variety $G/P$ is projective and the action of the unipotent radical $P_u^{-}$ of the opposite parabolic subgroup $P^-$ on $G/P$ by left multiplication has an open orbit. This open orbit $\Uf$ is called the big Schubert cell on $G/P$. Since $\Uf$ is isomorphic to the affine space $\AA^n$, where $n=\dim G/P$, every flag variety may be regarded as a completion of an affine space.

Our goal is to find all flag variety $G/P$ that are equivariant completions of $\GG_a^n$. Clearly, this is the case when the group $P_u^{-}$, or, equivalently, the group $P_u$ is commutative. 

It is a classical result that the connected component $\widetilde{G}$ of the automorphism group of the variety $G/P$ is a semisimple group of adjoint type, and $G/P=\widetilde{G}/Q$ for some parabolic subgroup $Q\subseteq \widetilde{G}$. In most cases the group $\widetilde{G}$ coincides with $G$, and all exceptions are well known, see, e.g., \cite[Theorem~7.1]{On1962} or \cite[page~118]{Ti1962}. If $\widetilde{G} \ne G$, we say that $(\widetilde{G}, Q)$ is the \emph{covering} pair of the \emph{exceptional} pair $(G,P)$. For a simple group $G$, the exceptional pairs are $(\PSp(2r), P_1)$, $(\SO(2r+1), P_r)$ and $(G_2, P_1)$ with the covering pairs $(\PSL(2r), P_1)$, $(\PSO(2r+2), P_{r+1})$ and $(\SO(7), P_1)$ respectively, where $PH$ denotes the quotient of the group $H$ by its center, and $P_i$ is the maximal parabolic subgroup associated with the $i$th simple root. It turns out that for a simple group $G$ the condition $\widetilde{G}\ne G$ implies that the unipotent radical $Q_u$ is commutative and $P_u$ is not. In particular, in this case $G/P$ is an equivariant completion of $\GG_a^n$. Our main result states that these are the only possible cases.

\begin{theorem} {\cite[Theorem~1]{Ar2011}} \label{gpex} 
Let $G$ be a connected semisimple group of adjoint type and $P$ be a parabolic subgroup of $G$. Then the flag variety $G/P$ is an equivariant completion of $\GG_a^n$ if and only if for every pair $(G^{(i)}, P^{(i)})$, where $G^{(i)}$ is a simple component of $G$ and $P^{(i)}= G^{(i)} \cap P$, one of the following conditions holds:
\begin{enumerate}
\item
the unipotent radical $P^{(i)}_u$ is commutative;
\item
the pair $(G^{(i)},P^{(i)})$ is exceptional.
\end{enumerate}
\end{theorem}

For the convenience of the reader, we list all pairs $(G,P)$, where $G$ is a simple group (up to local isomorphism) and $P$ is a parabolic subgroup with a commutative unipotent radical:
$$
(\SL(r+1),\, P_i), \ i=1,\ldots,r; \quad (\SO(2r+1),\, P_1); \quad (\Sp(2r),\, P_r);
$$
$$
(\SO(2r),\, P_i), \ i=1,r-1,r; \quad (E_6,\, P_i), \ i=1,6; \quad (E_7,\, P_7),
$$
see~\cite[Section~2]{RRS1992}. Note that the unipotent radical of $P_i$ is commutative if and only if the simple root $\alpha_i$ occurs in the highest root $\rho$ with
coefficient 1, see \cite[Lemma~2.2]{RRS1992}. Another equivalent condition is that the fundamental weight $\omega_i$ of the dual group $G^\vee$ is minuscule, i.e., the weight system of the simple $G^\vee$-module $V(\omega_i)$ with the highest weight $\omega_i$ coincides with the orbit $W\omega_i$ of the Weyl group $W$.

\begin{proof}[Proof of Theorem~\ref{gpex}]
If the unipotent radical $P_u^-$ is commutative, then the action of $P_u^-$ on $G/P$ by left multiplication is the desired additive action. The same arguments work when for the connected component $\widetilde{G}$ of the automorphism group $\Aut(G/P)$ one has $G/P=\widetilde{G}/Q$ and the unipotent radical $Q_u^-$ is commutative. Since
$$
G/P \, \cong \, G^{(1)}/P^{(1)} \times \ldots \times G^{(k)}/P^{(k)},
$$
where $G^{(1)},\ldots,G^{(k)}$ are the simple components of the group $G$, the group $\widetilde{G}$ is isomorphic to the direct product $\widetilde{G^{(1)}}\times \ldots \times \widetilde{G^{(k)}}$.
Moreover, $Q_u \cong Q_u^{(1)}\times \ldots \times Q_u^{(k)}$ with $Q^{(i)}=\widetilde{G^{(i)}}\cap Q$. Thus the group $Q_u^-$ is commutative if and only if for every pair $(G^{(i)}, P^{(i)})$ either $P_u^{(i)}$ is commutative or the pair $(G^{(i)}, P^{(i)})$ is exceptional.

Conversely, assume that $G/P$ admits an additive action. One may identify $\GG_a^n$ with a commutative unipotent subgroup $H$ of $\widetilde{G}$, and the flag variety $G/P$ with $\widetilde{G}/Q$, where $Q$ is a parabolic subgroup
of $\widetilde{G}$.

Let $T\subseteq B$ be a maximal torus and a Borel subgroup of the group $\widetilde{G}$ such that $B\subseteq Q$. Consider the root system $\Phi$ of the tangent algebra $\fg=\Lie(\widetilde{G})$ defined by the torus $T$, its decomposition $\Phi=\Phi^+ \cup \Phi^-$ into positive and negative roots associated with $B$, the set of simple roots $\Delta \subseteq \Phi^+$, $\Delta=\{\alpha_1,\ldots,\alpha_r\}$, and the root decomposition
$$
\fg \ = \ \bigoplus_{\beta\in\Phi^-} \fg_{\beta} \ \oplus \ft \ \oplus
\bigoplus_{\beta\in\Phi^+} \fg_{\beta},
$$
where $\ft=\Lie(T)$ is a Cartan subalgebra in $\fg$ and $$\fg_{\beta}=\{x\in\fg : [y,x]=\beta(y)x\ \text{for all} \ y\in\ft\}$$ is the root subspace. Set $\fq=\Lie(Q)$ and $\Delta_Q=\{\alpha\in\Delta : \fg_{-\alpha} \nsubseteq \fq \}$.
For every root $\beta = a_1\alpha_1 + \ldots + a_r\alpha_r$ we define $\deg(\beta)=\sum_{\alpha_i\in \Delta_Q} a_i$. This gives a $\ZZ$-grading on the Lie algebra~$\fg:$
$$
\fg = \bigoplus_{k\in\ZZ} \fg_k, \quad \text{where} \quad \ft\subseteq\fg_0 \quad
\text{and} \quad \fg_{\beta} \subseteq \fg_k \quad \text{with} \quad \ k=\deg(\beta).
$$
In particular,
$$
\fq \ = \ \bigoplus_{k\ge 0} \fg_k \quad \text{and} \quad
\fq_u^- \ = \ \bigoplus_{k<0} \fg_k.
$$
Assume that the unipotent radical $Q_u^-$ is not commutative, and consider $\fg_{\beta} \subseteq [\fq_u^-,\fq_u^-]$. For every $x\in \fg_{\beta} \setminus \{0\}$ there exist $z'\in\fg_{\beta'}\subseteq\fq_u^-$ and $z''\in\fg_{\beta''}\subseteq\fq_u^-$ such that $x=[z',z'']$. In this case $\deg(z')>\deg(x)$ and $\deg(z'')>\deg(x)$.

Since the subgroup $H$ acts on $\widetilde{G}/Q$ with an open orbit, one may conjugate $H$ and assume that the $H$-orbit of the point $eQ$ is open in $\widetilde{G}/Q$. This implies $\fg = \fq \oplus \fh$, where $\fh=\Lie(H)$. On the other hand, $\fg = \fq \oplus \fq_u^-$. So every element $y\in\fh$ may be uniquely written as $y=y_1+y_2$, where $y_1 \in \fq$, $y_2 \in \fq_u^-$, and the linear map $\fh \to \fq_u^-$, $y \mapsto y_2$, is bijective.
Take the elements $y,y',y'' \in \fh$ with $y_2=x,\, y'_2=z',\, y''_2=z''$. Since the subgroup $H$ is commutative, one has $[y',y'']=0$. Thus
$$
[y'_1 + y'_2, y''_1+y''_2]\ =\ [y'_1,y''_1] + [y'_2,y''_1] + [y'_1,y''_2] + [y'_2,y''_2]\ =\ 0.
$$
But
$$
[y'_2,y''_2]=x \quad \text{and} \quad [y'_1,y''_1] + [y'_2,y''_1] + [y'_1,y''_2] \in
\bigoplus_{k>\deg(x)} \fg_k.
$$
This contradiction shows that the group $Q_u^-$ is commutative. As we have seen, the latter condition means that for every pair $(G^{(i)}, P^{(i)})$ either the unipotent radical $P^{(i)}_u$ is commutative or the pair $(G^{(i)},P^{(i)})$ is exceptional.
\end{proof}

It is well known that if the ground field $\KK$ is the field of complex numbers then Hermitian symmetric spaces of compact type are precisely the homogeneous spaces $G/P$, where the parabolic subgroup $P$ has an commutative unipotent radical, see e.g.~\cite{RRS1992}. So we conclude from Theorem~\ref{gpex} the following observation. 

\begin{corollary}
A complete complex homogeneous variety $X$ admits an additive action if and only if $X$ is a Hermitian symmetric space of compact type. 
\end{corollary}

\subsection{Uniqueness results} 
\label{sub33}
The following result is a generalization of Theorem~\ref{hyp_quadr_nondeg_theor}; in the case of Grassmannians it was conjectured in~\cite[Section~6]{ArSh2011}. 

\begin{theorem} \label{unFHD}
Let $G$ be a connected simple linear algebraic group and $P$ a parabolic subgroup of $G$. Assume that the flag variety $X=G/P$ is not isomorphic to the projective space $\PP^n$. Then $X$ admits at most one additive action up to equivalence. 
\end{theorem} 

Two different ways to prove this result were obtained independently by Fu and Hwang~\cite{FH2014} and by Devyatov~\cite{De2015}. Let us discuss briefly each of these approaches. 

\smallskip

Fu and Hwang's proof is based on the study of varieties of minimal rational tangents (VMRT). In~\cite{FH2014}, the authors prove the following theorem.

\begin{theorem} \label{tFH}
Let $X$ be a smooth Fano variety of dimension $n$ with Picard number one that is not isomorphic to $\PP^n$. Assume that $X$ has a family of minimal rational curves whose variety of minimal rational tangents $C_x\subseteq \PP T_x(X)$ at a general point $x\in X$ is smooth. Then any two additive actions on $X$ are equivalent.
\end{theorem} 

\begin{corollary} \label{cFH} 
Let $X\subseteq \PP^N$ be a smooth projective subvariety of Picard number one such that for a general point $x\in X$, there exists a line of $\PP^N$ passing through $x$ and lying on~$X$. If $X$ is different from the projective space, then any two additive actions on $X$ are equivalent.
\end{corollary}

One can show that when $X$ has a projective embedding satisfying the assumption of Corollary~\ref{cFH}, some family of lines lying on $X$ gives a family of minimal rational curves, for which the VMRT $C_x$ at a general point $x\in X$ is smooth; see e.g.~\cite[Proposition~1.5]{Hw2001}. Thus Corollary~\ref{cFH} follows from Theorem~\ref{tFH}. This corollary can be applied to smooth quadratic hypersurfaces and Grassmanians because a smooth hyperquadric can be embedded into projective space with the required property.

\smallskip

The approach of Devyatov is completely different. It is based on representation theory of Lie algebras and the classification of certain multiplications on finite-dimensional spaces and may be considered as a generalization of the Hassett-Tschinkel correspondence. 

Let $L$ be a connected reductive algebraic group, $V$ be a finite-dimensional $L$-module, and $\fl$ be the Lie algebra of the group $L$. 
By an \emph{$\fl$-compatible multiplication} on $V$ we mean an associative commutative bilinear map $\mu\colon V\times V\to V$ such that for each $v\in V$
the operator $\mu_v\colon V\to V$, $\mu_v(w)=\mu(v,w)$ is nilpotent and for each $v\in V$ there exists $x\in\fl$ such that the operator
$\mu_v$ coincides with the action of $x$ on $V$. 

The classification of $\fl$-compatible multiplications on an arbitrary module $V$ of a reductive group $L$ is given in~\cite[Sections~5-6]{De2015}. It may be of independent interest. 
After a series of reductions to the case of a simple group $G$ and a simple module $V$, it is proved in \cite[Theorem~21]{De2015} that there exists a nonzero $\fl$-compatible multiplication on $V$ if and only if either $L$ is a group of type $A$ and $V$ is the tautological $L$-module or it dual, or $L$ is of type $C$ and $V$ is the tautological $L$-module.

Let $X=G/P$ be a flag variety, where $G$ is a connected simple linear algebraic group and $P$ is a parabolic subgroup of $G$. Without loss of generality, we may assume that the pair $(G,P)$ is not exceptional. Then the connected component of the group $\Aut(X)$ coincides with $G$. By Theorem~\ref{gpex}, we may assume that the group $P_u$ is commutative and, in particular, the parabolic subgroup $P$ is maximal. 

As we know from the previous subsection, additive actions on $X$ correspond to commutative Lie subalgebras $\fh$ that are complementary to $\fp=\Lie(P)$ in $\fg=\Lie(G)$, that is $\fg=\fp\oplus\fh$. One such subalgebra is $\fp_u^-$, and the uniqueness result means that all other such subalgebras $\fh$ are conjugated to $\fp_u^-$. In~\cite[Theorem~15]{De2015}, Devyatov establishes a correspondence between complementary to $\fp_u$ commutative subalgebras $\fh$ in $\fg$ and $\fl$-compatible multiplications $\fp_u^-\times\fp_u^-\to\fp_u^-$, where $\fl$ is a Levi subalgebra of the algebra $\fp$. Under this correspondence, the subalgebra $\fh=\fp_u^-$ corresponds to the zero multiplication. So, the classification of compatible multiplications mentioned above shows that non-conjugated to $\fp_u^-$ complementary to $\fp_u$ commutative subalgebras $\fh$ appear only in the cases when the flag variety $G/P$ is isomorphic to the projective space $\PP^n$. This proves Theorem~\ref{unFHD}. 

A more detailed analysis shows that in terms of the results of Hassett and Tschinkel, the compatible multiplication in the case of the tautological module $V$ of the group of type $A$ is precisely the multiplication $\mm\times\mm\to\mm$ on the maximal ideal of the local algebra $A=\KK\oplus\mm$ corresponding to a given additive action on $\PP^n$. 

\smallskip

Finishing this subsection, let us mention one more related research. In~\cite{Ch2017}, a uniqueness result for equivariant completions of non-commutative unipotent groups by flag varieties is proved. More precisely, let $G$ be a simple linear algebraic group over the field of complex numbers and $P$ a parabolic subgroup in $G$. By the Bruhat decomposition, the unipotent radical $P_u^-$ acts on $G/P$ with an open orbit isomorphic to $P_u^-$. Cheong proves that the structure of an equivariant completion of $P_u^-$ on $G/P$ is unique up to isomorphism if $G/P$ is not isomorphic to the projective space and the pair $(G,P)$ is not exceptional in the sense of the proceeding subsection. The proof exploits the notion of smooth VMRT and mostly follows a proving scheme of Fu and Hwang~\cite{FH2014}. The main difference lies in which tool to use in order to obtain an extension of a locally defined map to the entire space, which is the most essential in proving the uniqueness: Fu and Hwang use the Cartan-Fubini Extension Theorem, which is applicable only to smooth Fano varieties of Picard number one, while Cheong uses the Yamaguchi's result on the prolongation of simple graded Lie algebras. For exceptional pairs $(G,P)$ the question whether $G/P$ admits a unique equivariant completion structure for $P_u^-$ remains open. 


\subsection{Degeneration of flag varieties to equivariant completions}
\label{sub34} 
By Theorem~\ref{gpex}, not so many flag varieties admit an additive action. At the same time, in~\cite{Fe2012} Feigin proposed a construction of a canonical flat degeneration of an arbitrary flag variety to a projective variety with an additive action. This result may be considered as an additive analogue of intensively studied flat degenerations of flag varieties to toric varieties, see \cite{Ca2002,GL1996,Lak1995}. 

Let $G$ be a simple linear algebraic group with the Lie algebra $\fg$. Recall that each flag variety $G/P$ can be realized as a $G$-orbit $G[v_{\lambda}]\subseteq\PP(V(\lambda))$ in the projectivization of a simple $G$-module $V(\lambda)$ with highest weight vector $v_{\lambda}$. Here we denote $[v_{\lambda}]=\KK v_{\lambda}$ and let $\Ff_{\lambda}:=G[v_{\lambda}]$. Feigin introduces a new family of varieties $\Ff_{\lambda}^a$, which are flat degenerations of $\Ff_{\lambda}$; the superscript $a$ is for abelian.

The variety $\Ff_{\lambda}^a$ is defined as follows. Let $\{F_s, \, s\ge 0\}$ be the PBW filtration on $V(\lambda)$: 
$$
F_s=\Span\{x_1\ldots x_lv_{\lambda} \ : \ x_i\in\fg, l\le s\}.
$$
We define $V(\lambda)^a=F_0\oplus \bigoplus\limits_{s\ge 0} F_{s+1}/F_s$. Let $\fg = \fn\oplus\ft\oplus\fn^-$ be the Cartan decomposition. The space $V(\lambda)^a$ has a natural structure of a module over the degenerate algebra $\fg^a$, where the algebra $\fg^a$ is isomorphic to $\fg$ as a vector space and is a semidirect sum of two subalgebras, the first one is the Borel subalgebra $
\fb=\fn\oplus\ft$ and the second is an abelian ideal $(\fn^-)^a$, which is isomorphic to $\fn$ as a vector space. Here the structure of the $\fb$-module on $(\fn^-)^a$
is given via identification of the space $(\fn^-)^a$ with the factor module $\fg/\fb$. The corresponding algebraic group $G^a$ is the semidirect product $B\rightthreetimes \GG_a^n$, where $n$ is the dimension of $\fn$. We define the variety $\Ff_{\lambda}^a$ as the closure of the $\GG_a^n$-orbit of the highest weight vector: 
$$
\Ff_{\lambda}^a:=\overline{\GG_a^n[v_{\lambda}]}\subseteq\PP(V(\lambda)^a).
$$
By definition, this variety carries an additive action. Since the highest weight vector $v_{\lambda}$ is $B$-semi-invariant, the variety $\Ff_{\lambda}^a$ is invariant under the action of the group $G^a$ as well. Despite the case of usual flag varieties, here the action of $G^a$ on $\Ff_{\lambda}^a$ need not be transitive. 

Let $\fp$ be the parabolic subalgebra annihilating the vector $v_{\lambda}$. Assume for a moment that the nilpotent radical $\fp_u^-$ is commutative. Then all root operators from $\fn^-\setminus\fp_u^-$ annihilate the vector $v_{\lambda}$, while the operators from $\fp_u^-$ act as pairwise commuting operators in $V(\lambda)$ even before passing to $V(\lambda)^a$. This shows that in this case there is no difference between the original variety $\Ff_{\lambda}$ and the degenerate variety $\Ff_{\lambda}^a$. By Theorem~\ref{gpex}, this is precisely the case when the variety $\Ff_{\lambda}$ itself admits an additive action.

In the case of a group of type $A$, the varieties $\Ff_{\lambda}$ are isomorphic to partial flag varieties. In particular, the ones corresponding to fundamental weights $\lambda$ are Grassmannians 
$\Gr(d, n)$. There exist embeddings of partial flags into the product of projective spaces, and the image is given by Pl\"ucker relations. These relations describe the coordinate rings on the affine cones over flag varieties. 

It is shown in~\cite{Fe2012} that each degenerate flag variety can be embedded into the product of Grassmannians and thus into the product of projective spaces. Feigin shows that this embedding can be described in terms of an explicit set of multi-homogeneous algebraic equations that are obtained from the Pl\"ucker relations by a certain degeneration. He proves that the degeneration $\Ff_{\lambda}\to\Ff_{\lambda}^a$ is flat. 
 
For further results on varieties $\Ff_{\lambda}^a$, see e.g. \cite{Fe2012,FF2013,IFR2012} and references therein. Paper~\cite{FF2013} is devoted to the study of varieties $\Ff_{\lambda}^a$ 
for groups of type $A$. The authors prove that in this case the variety $\Ff_{\lambda}^a$ has rational singularities, is a normal and locally complete intersection, and construct explicitly a desingularization $R_{\lambda}$ of $\Ff_{\lambda}^a$. The variety $R_{\lambda}$ can be viewed as a tower of successive $\PP^1$-fibrations, thus providing an analogue of the classical Bott-Samelson-Demazure-Hansen desingularization. It is proved that the variety $R_{\lambda}$ is Frobenius split. This gives Frobenius splitting for the degenerate flag varieties and allows to prove the Borel-Weil type theorem for $\Ff_{\lambda}^a$.

The aim of paper~\cite{IFR2012} is to connect degenerate flag varieties with quiver Grassmannians. By definition, quiver Grassmannians are varieties parametrizing subrepresentations of a quiver representation. It turns out that certain quiver Grassmannians for type $A$ quivers are isomorphic to degenerate flag varieties $\Ff_{\lambda}^a$. This leads to the consideration of a class of Grassmannians of subrepresentations of the direct sum of a projective and an injective representation of a Dynkin quiver. It is proved that these are (typically singular) irreducible normal local complete intersections, which admit a group action with finitely many orbits and a cellular decomposition. 
%
%
\section{Additive actions on toric varieties}
\label{aatv}
In this section we study additive actions on toric varieties. To do this, we need to develop new techniques, namely, we consider graded algebras and homogeneous locally nilpotent derivations of such algebras. In order to apply these techniques, we would like to have global coordinates on any toric variety. Such coordinates are provided by Cox rings. In the case of a toric variety $X$, the Cox ring $R(X)$ is a polynomial ring in $m$ variables, where $m$ is the number of prime torus invariant divisors on $X$. Moreover, the ring $R(X)$ is graded by the divisor class group $\Cl(X)$, and locally nilpotent derivations, which correspond to $\GG_a$-actions on $X$ normalized by the acting torus, are represented by so-called Demazure roots of the corresponding fan. 

This allows us to characterize additive actions on $X$ normalized by the acting torus in terms of certain collections of Demazure roots. In particular, we obtain a combinatorial description of fans of toric varieties that admit a normalized additive action. We show that there is at most one normalized additive action on any toric variety up to equivalence. Further, we prove that if a complete toric variety $X$ admits an additive action, then it admits a normalized additive action. Moreover, this is the case if and only if a maximal unipotent subgroup $U$ in $\Aut(X)$ acts on $X$ with an open orbit. A characterization of polytopes corresponding to projective toric varieties admitting an additive action is obtained. 

In subsection~\ref{subcts} we describe additive actions on complete toric surfaces $X$. It turns out that there are at most two additive actions on $X$ in this case. The last subsection provides a uniqueness criterion for additive actions on a complete toric variety $X$ of arbitrary dimension: if $X$ admits a normalized additive action, then any other additive action on $X$ is equivalent to a normalized one if and only if a maximal unipotent subgroup $U$ in $\Aut(X)$ is commutative. The latter condition can be easily checked in terms of the fan. 

\subsection{Graded algebras and locally nilpotent derivations}
In this subsection we follow the presentation in~\cite{AR2017}. Consider an irreducible affine variety $X$ with an effective action of an
algebraic torus~$T$. Let $M$ be the character lattice of $T$ and $N=\Hom(M,\ZZ)$ be the dual lattice of one-parameter subgroups of $T$. Let $B=\KK[X]$ be the algebra of regular functions on $X$. It is well known that there is a bijection between faithful $T$-actions on $X$ and effective $M$-gradings on~$B$. In fact, the algebra $B$ is graded by a semigroup of lattice points in a convex polyhedral cone $\omega\subseteq M_{\QQ}=M\otimes_{\ZZ}\QQ$. We have
$$
B=\bigoplus_{m\in \omega_{ M}} B_m\chi^m,
$$
where $\omega_{M}=\omega\cap M$ and $\chi^m$ is the character of the torus $T$ corresponding to a point $m\in M$.

A derivation $\partial$ of an algebra $B$ is said to be {\itshape locally nilpotent} (LND) if for every $f\in B$ there exists $k\in\NN$ such that $\partial^k(f)=0$. For any LND $\partial$ on $B$ the map ${\varphi_{\partial}:\GG_a\times B\rightarrow B}$, ${\varphi_{\partial}(s,f)=\exp(s\partial)(f)}$, defines a
structure of a rational $\GG_a$-algebra on $B$. This induces a regular action $\GG_a\times X\to X$, where $X=\Spec B$. In fact, any regular $\GG_a$-action on $X$ arises this way, see \cite[Section~1.5]{Fr2006}. A derivation $\partial$ on $B$ is said to be {\itshape homogeneous} if it respects the $M$-grading, i.e. $\partial$ sends homogeneous elements to homogeneous ones. If ${f,h\in B\backslash \ker\partial}$ are homogeneous then ${\partial(fh)=f\partial(h)+\partial(f)h}$ is homogeneous as well and ${\dg\partial(f)-\dg f=\dg\partial(h)-\dg h}$. So any homogeneous derivation $\partial$ has a well defined {\itshape degree} given as $\dg\partial=\dg\partial(f)-\dg f$ for any homogeneous $f\in B\backslash \ker\partial$. It is easy to see that an LND on $B$ is homogeneous if and only if the corresponding $\GG_a$-action is normalized by the torus~$T$ in the automorphism group~$\Aut (X)$, cf. \cite[Section~3.7]{Fr2006}.

\smallskip

Let $X$ be an affine toric variety, i.~e. a normal affine variety with an effective action of a torus $T$ with an open orbit. In this case
$$
B=\bigoplus_{m\in \omega_{M}}\KK\chi^m=\KK[\omega_{M}]
$$
is the semigroup algebra. Recall that for a given cone $\omega\subseteq M_{\QQ}$, its {\itshape dual cone} $\sigma\subseteq N_{\QQ}$ is defined as
$$
\sigma=\{p\in N_{\mathbb{Q}} \mid \langle p,v\rangle\geqslant0\,\,\,\forall v\in\omega\},
$$
where $\langle\,\cdot\,,\,\cdot\,\rangle$ is the pairing $N_{\mathbb{Q}}\times M_{\mathbb{Q}}\to\QQ$ between dual spaces $N_{\mathbb{Q}}$ and $M_{\mathbb{Q}}$. Let $\sigma(1)$ be the set of rays of a cone $\sigma$ and $p_{\rho}$ be the primitive lattice vector on a ray $\rho$. For $\rho\in\sigma(1)$ we set
$$
\mathfrak{R}_{\rho}:=\{e\in M\,|\, \langle p_{\rho},e\rangle=-1
\,\,\mbox{and}\,\, \langle p_{\rho'},e\rangle\geqslant0
\,\,\,\,\forall\,\rho'\in \sigma(1), \,\rho'\ne\rho\}.
$$
One easily checks that the set $\mathfrak{R}_{\rho}$ is infinite for each $\rho\in\sigma(1)$ provided the cone $\sigma$ has dimension at least two. The elements of the set $\mathfrak{R}:=\bigsqcup\limits_{\rho\in\sigma(1)}\mathfrak{R}_{\rho}$ are called the {\itshape Demazure roots} of the cone $\sigma$. Let $e\in\mathfrak{R}_{\rho}$. Then $\rho$ is  the {\itshape distinguished ray} of the root $e$. One can define the homogeneous LND on the algebra $B$ by the rule
$$
\partial_e(\chi^m)=\langle p_{\rho},m\rangle\chi^{m+e}.
$$
In fact, every homogeneous LND on~$B$ has a form $\alpha\partial_e$ for some $\alpha\in
\KK,\, e\in \mathfrak{R}$, see \cite[Theorem~2.7]{Li2010}. In other words, $\GG_a$-actions on $X$
normalized by the acting torus are in bijection with Demazure roots of the cone $\sigma$.

\begin{example} \label{ex1}
Consider $X=\KK^n$ with the standard action of the torus $(\KK^{\times})^n$.
It is a toric variety with the cone $\sigma=\QQ^n_{\geqslant0}$ having rays
$\rho_i=\langle p_i\rangle_{\QQ_{\geqslant0}}$ with
$p_1=(1,0,\ldots,0),\ldots,p_n=(0,\ldots,0,1)$.
The dual cone $\omega$ is $\QQ^n_{\geqslant0}$ as well. In this case we have
$$
\mathfrak{R}_{\rho_i}=\{(c_1,\ldots,c_{i-1},-1,c_{i+1},\ldots,c_n)\,|\,c_j\in\ZZ_{\geqslant0}\},
$$
\vspace{0.05cm}
\begin{center}
\begin{picture}(100,75)
\multiput(50,15)(15,0){5}{\circle*{3}}
\multiput(35,30)(0,15){4}{\circle*{3}}
\put(20,30){\vector(1,0){100}} \put(50,5){\vector(0,1){80}}
\put(17,70){$\mathfrak{R}_{\rho_1}$} \put(115,7){$\mathfrak{R}_{\rho_2}$}
\put(100,70){$M_{\mathbb{Q}}=\mathbb{Q}^2$} \linethickness{0.5mm}
\put(50,30){\line(1,0){65}} \put(50,30){\line(0,1){50}}
\end{picture}
\end{center}
where $c_j=\langle p_j,e\rangle$. Denote $x_1=\chi^{(1,0,\ldots,0)},\ldots,x_n=\chi^{(0,\ldots,0,1)}$. Then $\KK[X]=\KK[x_1,\ldots,x_n]$. 
Consider the monomial $x^e:=x_1^{c_1}\ldots x_{i-1}^{c_{i-1}} x_{i+1}^{c_{i+1}}\ldots x_n^{c_n}$. 
It is easy to see that the homogeneous LND corresponding to the root
$e=(c_1,\ldots,c_n)\in \mathfrak{R}_{\rho_i}$ is
$$
\partial_e=x^e\frac{\partial}{\partial x_i}. 
$$

This LND gives rise to the $\GG_a$-action
$$
x_i\mapsto x_i+sx^e, \quad x_j\mapsto x_j, \quad j\ne i, \quad s\in\GG_a.
$$
\end{example}


\subsection{Cox rings and Demazure roots}

We keep notation of the previous subsection and continue to follow~\cite{AR2017}. Let $X$ be a toric variety of dimension $n$ with an
acting torus $T$. This time we do not assume that $X$ is affine, and so $X$ is represented by a fan $\Sigma$ of convex polyhedral cones in $N_{\QQ}$; see~\cite{Fu1993} or \cite{CLS2011} for details.

Let $\Sigma(1)$ be the set of rays of the fan $\Sigma$ and $p_{\rho}$ be the primitive lattice vector on a ray $\rho$. For $\rho\in\Sigma(1)$ we consider the set $\mathfrak{R}_{\rho}$ of all vectors $e\in M$ such that
\begin{enumerate}
\item[(1)]
$\langle p_{\rho},e\rangle=-1\,\,\mbox{and}\,\, \langle p_{\rho'},e\rangle\geqslant0
\,\,\,\,\forall\,\rho'\in \Sigma(1), \,\rho'\ne\rho$;
\smallskip
\item[(2)]
if $\sigma$ is a cone of $\Sigma$ and $\langle v,e\rangle=0$ for all $v\in\sigma$, then the cone generated by $\sigma$ and $\rho$ is in $\Sigma$ as well.
\end{enumerate}

Note that~(1) implies~(2) if $\Sigma$ is a fan with convex support. This is the case if $X$ is affine or complete.

The elements of the set $\mathfrak{R}:=\bigsqcup\limits_{\rho\in\Sigma(1)}\mathfrak{R}_{\rho}$ are called the {\itshape Demazure roots} of the fan $\Sigma$, cf.~\cite[D\'efinition~4]{De1970} and \cite[Section~3.4]{Od1988}. Again elements $e\in\mathfrak{R}$ are in bijection with $\GG_a$-actions on $X$ normalized by the acting torus, see~\cite[Th\'eoreme~3]{De1970} and \cite[Proposition~3.14]{Od1988}. If $X$ is affine, the $\GG_a$-action given by a Demazure root~$e$ coincides with the action corresponding to the locally nilpotent derivation $\partial_e$ of the algebra $\KK[X]$ as defined in the previous subsection. Let us denote by $H_e$ the image in $\Aut(X)$ of the group $\GG_a$ under this action. Thus $H_e$ is a one-parameter unipotent subgroup normalized by $T$ in $\Aut(X)$.

We recall basic facts from toric geometry. There is a bijection between cones $\sigma\in\Sigma$ and $T$-orbits $\Of_{\sigma}$ on $X$ such that $\sigma_1\subseteq\sigma_2$ if and only if $\Of_{\sigma_2}\subseteq\overline{\Of_{\sigma_1}}$. Here $\dim\Of_{\sigma}=n-\dim\langle\sigma\rangle$. Moreover, each cone $\sigma\in\Sigma$ defines an open affine $T$-invariant subset $U_{\sigma}$ on $X$ such that $\Of_{\sigma}$ is the unique closed $T$-orbit on $U_{\sigma}$ and $\sigma_1\subseteq\sigma_2$ if and only if $U_{\sigma_1}\subseteq U_{\sigma_2}$.

Let $\rho_e$ be the distinguished ray corresponding to a root $e$, $p_e$ be the primitive lattice vector on $\rho_e$, and $R_e$ be the one-parameter subgroup of $T$ corresponding to $p_e$. Denote by $X^{H_e}$ the set of $H_e$-fixed points on $X$.

\smallskip

We aim to describe the action of $H_e$ on $X$.

\begin{proposition} {\cite[Proposition~1]{AR2017}} \label{phe}
For every point $x\in X\setminus X^{H_e}$ the orbit $H_ex$ meets exactly two $T$-orbits $\Of_1$ and $\Of_2$ on $X$, where $\dim\Of_1=\dim\Of_2+1$. The intersection $\Of_2\cap H_ex$ consists of a single point, while
$$
\Of_1\cap H_ex=R_ey \quad \text{for any} \quad y\in\Of_1\cap H_ex.
$$
\end{proposition}

\begin{proof}
It follows from the proof of \cite[Proposition~3.14]{Od1988} that the affine charts $U_{\sigma}$, where $\sigma\in\Sigma$ is a cone containing $\rho_e$, are $H_e$-invariant, and the complement of their union is contained in $X^{H_e}$. This reduces the proof to the case $X$ is affine. Then the assertion is proved in \cite[Proposition~2.1]{AKZ2012}.
\end{proof}

A pair of $T$-orbits $(\Of_1,\Of_2)$ on $X$ is said to be {\itshape $H_e$-connected} if $H_ex\subseteq \Of_1\cup\Of_2$ for some $x\in X\setminus X^{H_e}$. By Proposition~\ref{phe}, $\Of_2\subseteq\overline{\Of_1}$ for such a pair (up to permutation) and $\dim\Of_1=\dim\Of_2+1$. Since the torus normalizes the subgroup $H_e$, any point of $\Of_1\cup\Of_2$ can actually serve as the point $x$.

\begin{lemma} {\cite[Lemma~1]{AR2017}} \label{lpe}
A pair of $T$-orbits $(\Of_{\sigma_1},\Of_{\sigma_2})$ is $H_e$-connected if and only if $e|_{\sigma_2}\le 0$ and $\sigma_1$ is a facet of $\sigma_2$ given by the equation
$\langle v,e\rangle=0$.
\end{lemma}

The proof again reduces to the affine case, where the assertion is \cite[Lemma~2.2]{AKZ2012}.

\smallskip

Now we recall the basic ingredients of the Cox construction; see~\cite[Chapter~1]{ADHL2015} for more details. Let $X$ be a normal variety with free finitely generated divisor class group $\Cl(X)$ and only constant invertible regular functions. Denote by $\WDiv(X)$ the group of Weil divisors on $X$ and fix a subgroup $K \subseteq \WDiv(X)$ which maps onto $\Cl(X)$ isomorphically. The {\itshape Cox ring} of the variety $X$ is defined as
$$
R(X)=\bigoplus_{D\in K} H^0(X,D),
$$
where $H^0(X,D)=\{f\in\KK(X)^{\times}\mid \ddiv(f)+D\geqslant0\}\cup\{0\}$ and multiplication on homogeneous components coincides with multiplication in the field of rational functions $\KK(X)$ and extends to $R(X)$ by linearity. It is easy to see that up to isomorphism the graded ring $R(X)$ does not depend on the choice of the subgroup $K$.

\smallskip

It is proved in \cite{Cox1995} that if $X$ is toric, then $R(X)$ is a polynomial algebra $\KK[x_1,\ldots,x_m]$, where the variables $x_i$ correspond to $T$-invariant prime divisors $D_i$ on $X$ or, equivalently, to the rays $\rho_i$ of the fan $\Sigma_X$. The $\Cl(X)$-grading on $R(X)$ is given by $\dg(x_i)=[D_i]$.

\smallskip

Suppose that the Cox ring $R(X)$ is finitely generated. Then ${\overline{X}:=\Spec R(X)}$ is called the \emph{total coordinate space} of the variety $X$. It is an affine factorial variety with an action of the torus $H_X := \Spec\KK[\Cl(X)]$. There is an open $H_X$-invariant subset $\widehat{X}\subseteq \overline{X}$ such that the complement $\overline{X}\backslash\widehat{X}$ is of codimension at least two in $\overline{X}$, there exists a good quotient $p_X\colon\widehat{X}\rightarrow\widehat{X}/\!/H_{X}$, and the quotient space $\widehat{X}/\!/H_{X}$ is isomorphic to $X$, see \cite[Construction~1.6.3.1]{ADHL2015}. If $X$ is smooth, the quotient map $p_X\colon\widehat{X}\rightarrow\widehat{X}/\!/H_{X}$ is called the \emph{universal torsor} over $X$. So we have the diagram
$$
\begin{CD}
\widehat{X} @>{i}>> \overline{X}=\Spec R(X)\\
@VV{/\!/H_{X}}V  \\
X
\end{CD}
$$
If $X$ is toric, then $\overline{X}$ is isomorphic to $\KK^m$ and $\overline{X}\setminus\widehat{X}$ is a union of some coordinate planes in $\KK^m$ of codimension at least two~\cite{Cox1995}.

\smallskip

By~\cite[Theorem~4.2.3.2]{ADHL2015}, any $\GG_a$-action on a variety $X$ can be lifted to a $\GG_a$-action on the variety $\overline{X}$ commuting with the action of the torus $H_X$.

If $X$ is toric and a $\GG_a$-action is normalized by the acting torus $T$, then the lifted $\GG_a$-action on $\KK^m$ is normalized by the diagonal torus $(\KK^{\times})^m$. Conversely,
any $\GG_a$-action on $\KK^m$ normalized by the torus $(\KK^{\times})^m$ and commuting with the subtorus $H_X$ induces a $\GG_a$-action on $X$. This shows that $\GG_a$-actions on $X$ normalized by $T$ are in bijection with locally nilpotent derivations of the Cox ring $\KK[x_1,\ldots,x_m]$ that are homogeneous with respect to the standard grading by the lattice
$\ZZ^m$ and have degree zero with respect to the $\Cl(X)$-grading.
%
\subsection{Normalized additive actions}

Let $X$ be a normal variety admitting an additive action with the open orbit $\Uf$. By Proposition~\ref{adcl}, any invertible regular function on $X$ is constant and the divisor class group $\Cl(X)$ is freely generated by classes $[D_1],\ldots,[D_l]$ of prime divisors such that $X\setminus\Uf=D_1\cup\ldots\cup D_l$. In particular, the Cox ring $R(X)$ is well defined for such a variety $X$.

Now we assume that $X$ is toric and an additive action $\GG_a^n\times X\to X$ is normalized by the acting torus $T$. Then the group $\GG_a^n$ is a direct product of $n$ subgroups $\GG_a$ each normalized by~$T$. They correspond to pairwise commuting homogeneous locally nilpotent derivations on the Cox ring $\KK[x_1,\ldots,x_m]$ having degree zero with respect to the $\Cl(X)$-grading. In turn, such derivations up to scalar are in bijection with Demazure roots of the fan $\Sigma_X$.

Consider a set of homogeneous derivations $\partial_{e}$ of the polynomial algebra $\KK[x_1,\ldots,x_m]$ corresponding to some Demazure roots $e$ of the fan $\Sigma_X$.

\begin{lemma} {\cite[Lemma~2]{AR2017}} \label{com}
Derivations $\partial_e$ and $\partial_{e'}$ commute if and only if either $\rho_e=\rho_{e'}$ or $\langle p_e,e'\rangle=\langle p_{e'},e\rangle = 0$.
\end{lemma}

Now we come to the key definition. 

\begin{definition} \label{deff}
A set $e_1,\ldots,e_n$ of Demazure roots of a fan $\Sigma$ of dimension $n$ is called a {\it complete collection} if $\langle p_i,e_j\rangle=-\delta_{ij}$ for all $1\le i,j\le n$, where $\delta_{ij}$ is the Kronecker delta. 
\end{definition}

In this case, the vectors $p_1,\ldots,p_n$ form a basis of the lattice $N$, and $-e_1,\ldots,-e_n$ is the dual basis of the dual lattice $M$.

\smallskip

The following result may be considered as a combinatorial description of normalized additive actions on toric varieties.

\begin{theorem} {\cite[Theorem~1]{AR2017}} \label{cc}
Let $X$ be a toric variety. Then additive actions on $X$ normalized by the acting torus $T$ are in bijection with complete collections of Demazure roots of the fan $\Sigma_X$.
\end{theorem}

As we have seen, a normalized additive action on $X$ determines $n$ pairwise commuting homogeneous locally nilpotent derivations of the Cox ring $\KK[x_1,\ldots,x_m]$. They have the form
$\partial_e$ for some Demazure roots $e$. So Theorem~\ref{cc} follows directly from the next lemma. 

\begin{lemma} {\cite[Lemma~3]{AR2017}}
Homogeneous locally nilpotent derivations $\partial_1,\ldots,\partial_n$ of the Cox ring $\KK[x_1,\ldots, x_m]$ corresponding to Demazure roots $e_1,\ldots,e_n$ define a normalized
additive action on $X$ if and only if $e_1,\ldots,e_n$ form a complete collection.
\end{lemma}

\begin{proof}
Assume first that the derivations $\partial_1,\ldots,\partial_n$ give rise to an additive action ${\GG_a^n\times X\to X.}$ If some roots $e_i$ and $e_j$ with $i\ne j$ correspond to the same ray of the fan $\Sigma_X$, then the $\GG_a^n$-action changes at most $n-1$ coordinates of the ring $\KK[x_1,\ldots, x_m]$, and any $\GG_a^n$-orbit on $X$ can not be $n$-dimensional. Then Lemma~\ref{com} implies that $\langle p_i,e_j\rangle~=~0$ for $i\ne j$. By definition, we have $\langle p_i,e_i\rangle~=-1$, and thus $e_1,\ldots,e_n$ form a complete collection.

Conversely, if $e_1,\ldots,e_n$ is a complete collection, then the corresponding homogeneous locally nilpotent derivations commute. They define a $\GG_a^n$-action on $\KK[x_1,\ldots, x_m]$,
and hence on $\KK^m$. This action descends to $X$. We have to show that the $\GG_a^n$-action on $X$ has an open orbit. For this purpose, it suffices to check that the group $\GG_a^n\times H_X$ acts on $\KK^m$ with an open orbit.

By construction, the group $\GG_a^n$ changes exactly $n$ of the coordinates $x_1,\ldots,x_m$, while the weights of the remaining $m-n$ coordinates with respect to the $\Cl(X)$-grading form a basis of the lattice of characters of the torus $H_X$. This shows that the stabilizer of the point $(1,\ldots,1)\in\KK^m$ in the group
$\GG_a^n\times H_X$ is trivial. Since $\dim(\GG_a^n\times H_X)=n+m-n=m$, we conclude that the $(\GG_a^n\times H_X)$-orbit of the point $(1,\ldots,1)$ is open in $\KK^m$.
\end{proof}

\begin{corollary} \label{ncc}
A toric variety $X$ admits a normalized additive action if and only if there is a complete collection of Demazure roots of the fan $\Sigma_X$.
\end{corollary}

\begin{remark}
In~\cite{AK2015}, one more application of Demazure roots to the theory of equivariant completions of commutative linear algebraic groups is given.
\end{remark}

The following theorem shows that a normalized additive action on a toric variety is essentially unique. 

\begin{theorem} {\cite[Theorem~2]{AR2017}} \label{unique}
Any two normalized additive actions on a toric variety are equivalent.
\end{theorem}

Let $X$ be a complete toric variety with an acting torus $T$. It is well known that the automorphism group $\Aut(X)$ is a linear algebraic group with $T$ as a maximal torus, see~\cite{De1970}, \cite{Cox1995}. In particular, $\Aut(X)$ contains a maximal unipotent subgroup $U$, and all such subgroups are conjugate in $\Aut(X)$.

\begin{theorem} {\cite[Theorem~3]{AR2017}} \label{3con}
Let $X$ be a complete toric variety with an acting torus~$T$. The following conditions are equivalent:
\begin{itemize}
\item[(1)]
there exists an additive action on $X$ normalized by the acting torus $T$;
\item[(2)]
there exists an additive action on $X$;
\item[(3)]
a maximal unipotent subgroup $U$ of the automorphism group $\Aut(X)$ acts on $X$ with an open orbit.
\end{itemize}
\end{theorem}

We conclude that Corollary~\ref{ncc} characterizes complete toric varieties admitting some additive action. 

\begin{corollary}\label{cba}
A complete toric variety $X$ admits an additive action if and only if there is a complete collection of Demazure roots of the fan $\Sigma_X$.
\end{corollary}

In terms of the fan $\Sigma_X$, this condition means that up to renumbering first $n$ primitive vectors $p_1,\ldots,p_n$ on the rays of $\Sigma_X$ form a basis of the lattice $N$, and the remaining primitive vectors $p_{n+1},\ldots,p_m$ have non-positive coordinates in this basis. 

\subsection{Projective toric varieties and polytopes}

It is well known that there is a correspondence between convex lattice polytopes and projective toric varieties. This subsection aims to characterize polytopes corresponding to toric varieties that admit an additive action.

We begin with preliminary results; see \cite[Chapter~2]{CLS2011} and \cite[Section~1.5]{Fu1993} for more details. Let $M$ be a lattice of rank $n$ and $P$ be a full dimensional convex polytope in the space $M_{\QQ}$. We say that $P$ is a {\it lattice polytope} if all its vertices are in $M$.

A subsemigroup $S\subseteq M$ is called {\it saturated} if $S$ coincides with the intersection of the group $\ZZ S$ and the cone $\QQ_{\ge 0}S$ it generates. A lattice polytope $P$ is {\it very ample} if for every vertex $v\in P$, the semigroup $S_{P,v}:=\ZZ_{\ge 0}(P\cap M - v)$ is saturated. It is known that for every lattice polytope $P$ and every $k\ge n-1$ the polytope $kP$ is very ample, see \cite[Corollary~2.2.19]{CLS2011}.

Let us consider $M$ as a lattice of characters of a torus $T$. Let $P\subseteq M_{\QQ}$ be a very ample polytope and $P\cap M=\{m_1,\ldots,m_s\}$. We consider a map
$$
T\to\PP^{s-1}, \quad t\mapsto [\chi^{m_1}(t):\ldots:\chi^{m_s}(t)]
$$
and define a variety $X_P$ as the closure of the image of this map in $\PP^{s-1}$. It is known that $X_P$ is a projective toric variety with the acting torus $T$, and any projective toric variety appears this way.

\begin{definition} \label{iir}
We say that a very ample polytope $P$ is \emph{inscribed in a rectangle} if there is a vertex $v_0\in P$ such that
\begin{enumerate}
\item[(1)]
the primitive vectors on the edges of $P$ containing $v_0$ form a basis $e_1,\ldots,e_n$
of the lattice $M$;
\item[(2)]
for every inequality $\langle p,x\rangle\le a$ on $P$ that corresponds to a facet of $P$ not passing through $v_0$ we have $\langle p,e_i\rangle\ge 0$ for all $i=1,\ldots,n$.
\end{enumerate}
\end{definition}

\begin{figure}[h]\label{rectangle-polytope}
\centerline{\includegraphics[scale = 0.12]{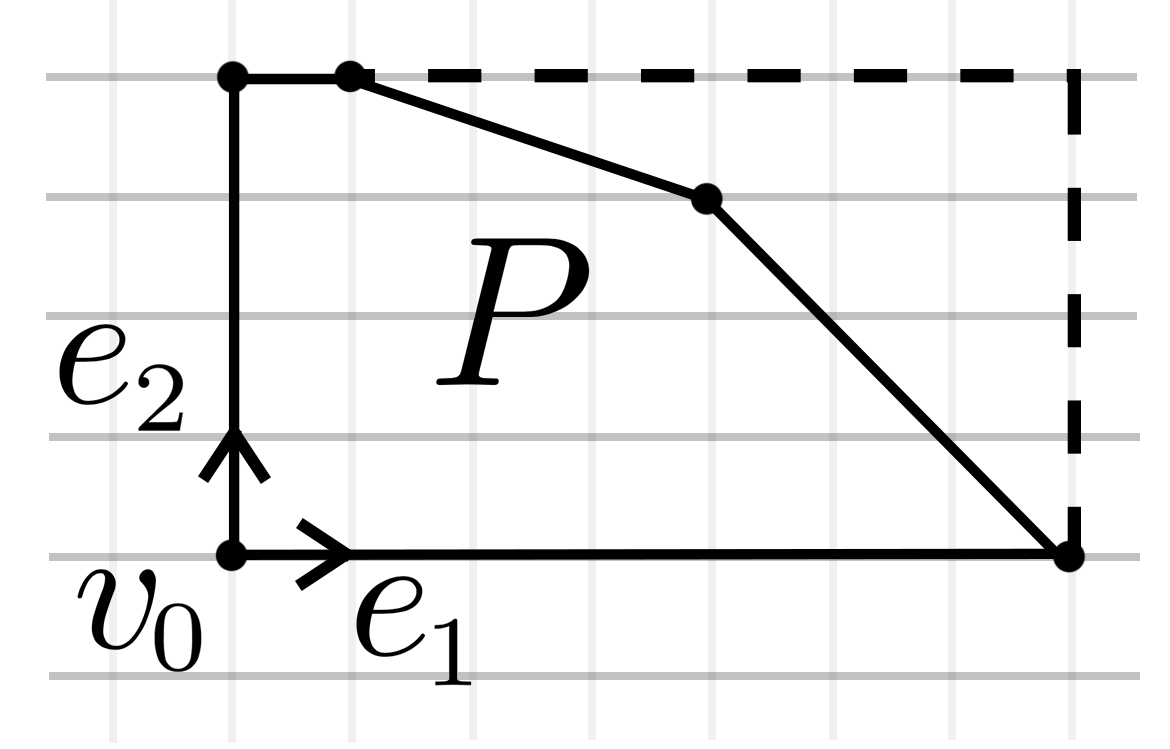}}
\end{figure}

\begin{theorem} {\cite[Theorem~4]{AR2017}} \label{ppoo}
Let $P$ be a very ample polytope and $X_P$ be the corresponding projective toric variety. Then $X_P$ admits an additive action if and only if the polytope $P$ is inscribed in a rectangle.
\end{theorem}

\begin{proof}
By Corollary~\ref{cba}, a toric variety $X$ admits an additive action if and only if the fan~$\Sigma_X$ admits a complete collection of Demazure roots. By \cite[Proposition~3.1.6]{CLS2011}, the fan~$\Sigma_{X_P}$ of the toric variety $X_P$ corresponding to the polytope~$P$ coincides with the normal fan~$\Sigma_P$ of the polytope~$P$. It is straightforward to check that two conditions of Definition~\ref{iir} are precisely the conditions on $-e_1,\ldots,-e_n$ to be a complete collection of Demazure roots of the fan $\Sigma_P$.
\end{proof}

\begin{remark}
The result of Theorem~\ref{ppoo} can also be obtained using the language of polytopal linear groups developed in~\cite{BrGu1999}. 
\end{remark}

Let us illustrate this approach with two examples.

\smallskip

The segment $P=[0,d]$ in $\QQ^1$ with $d\in\ZZ_{\ge 1}$ is a polytope inscribed in a rectangle, and the variety
$$
X_P=\overline{\{[1:a:\ldots:a^d] \mid a\in\KK\}}\subseteq\PP^d
$$
is a rational normal curve of degree $d$.

\smallskip

Further, the polytope

\begin{figure}[h]\label{hirtzebruch-polytope}
\centerline{\includegraphics[scale = 0.25]{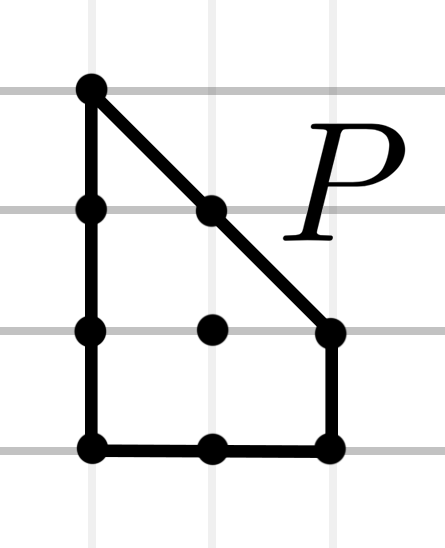}}
\end{figure}

\noindent defines the surface
$$
X_P=\overline{\{[1:a:a^2:b:ab:a^2b:b^2:ab^2:b^3] \mid a,b\in\KK\}}\subseteq\PP^8
$$
isomorphic to the Hirzebruch surface $\FF_1$.

\smallskip

Now we give some explicit formulas for additive actions on toric varieties in terms of Cox rings.

\begin{example} \label{Pn_Demaz_ex}
The fan of the projective space $X = \PP^n$ is generated by a basis $p_1, \ldots, p_n$ of the lattice~$\ZZ^n$ and the vector $p_0=-p_1-\ldots-p_n$. The complete collection of Demazure roots $e_i$, ${1 \le i \le n}$, which consists of vectors opposite to the dual basis of $p_1, \ldots, p_n$, corresponds to pairwise commuting locally nilpotent derivations $\pa_{e_i}=x_0\frac{\partial}{\partial x_i}$, $1\le i \le n$, on the Cox ring $R(X) = \KK[x_0, \ldots, x_n]$. They define a $\GG_a^n$-action 
$$
(x_0,x_1,\ldots,x_n)\mapsto (x_0,x_1+s_1x_0,\ldots,x_n+s_nx_0)
$$
on the total coordinate space $\overline{X}=\AA^{n+1}$, see Example~\ref{ex1}. This action is normalized by the diagonal torus $(\KK^\times)^{n+1}$ and commutes with the action of one-dimensional torus $H_X = \GG_m$ since locally nilpotent derivations are homogeneous with respect to the standard grading by~$\ZZ^{n+1}$ and have degree zero with respect to the grading by $\Cl(X) = \ZZ$, $\deg x_i=1$, $0\le i \le n$. Thus, this action induces the normalized additive action on the projective space~$\PP^n$:
\begin{equation} \label{Pn_normact_eq}
[z_0:z_1:\ldots:z_n]\mapsto [z_0:z_1+s_1z_0:\ldots:z_n+s_nz_0].
\end{equation}
The hyperplane $\{z_0=0\}$ consists of $\GG_a^n$-fixed points and thus for $n\ge 2$ the number of $\GG_a^n$-orbits on $\PP^n$ is infinite.

Consider the case $n=2$. A maximal unipotent subgroup of the automorphism group $\Aut(\PP^2)$ is isomorphic to the unitriangular matrix subgroup of $\GL_3(\KK)$ and consists of automorphisms
$$
[z_0:z_1:z_2]\mapsto [z_0:z_1+a_{12}z_0:z_2+a_{23}z_1+a_{13}z_0], \; a_{12}, a_{23}, a_{13} \in \KK.
$$
Two-dimensional (commutative) subgroups of this group are 
\[
H_{[\alpha:\beta]} = \left\{\begin{pmatrix}1 & \alpha a & b \\ 0 & 1 & \beta a \\ 0 & 0 & 1\end{pmatrix}, \; a,b \in \KK\right\}\]
for $[\alpha:\beta] \in \PP^1$. If $[\alpha:\beta] = [0:1]$ the corresponding action of the group~$\GG_a^2$ has no open orbit, if $[\alpha:\beta] = [1:0]$ we obtain normalized additive action~\eqref{Pn_normact_eq}, and points $[\alpha:\beta] \in \PP^1 \setminus \{0,\infty\}$ determine pairwise isomorphic non-normalized additive actions with three orbits, see Example~\ref{eexx}. 
\end{example}

\begin{example} \label{P1P1_Demaz_ex}
In the same way as in Example~\ref{Pn_Demaz_ex} one can check that the normalized additive action on the product $\PP^1\times\ldots\times\PP^1$ and the corresponding action on the total coordinate space $\AA^{2n}$ are given by
\begin{gather*}
([z_1:z_2],\ldots,[z_{2n-1}:z_{2n}])\mapsto ([z_1:z_2+s_1z_1],\ldots,[z_{2n-1}:z_{2n}+s_nz_{2n-1}]), \\
(x_1,x_2,\ldots,x_{2n-1},x_{2n})\mapsto (x_1,x_2+s_1x_1,\ldots,x_{2n-1},x_{2n}+s_nx_{2n-1}).
\end{gather*}
This shows that the number of $\GG_a^n$-orbits on $\PP^1\times\ldots\times\PP^1$ is $2^n$.
\end{example}

\begin{example} \label{Fd_Demaz_ex}
Let $X$ be the Hirzebruch surface $\FF_d$. Its fan is generated by the vectors
$$
p_1=(1,0), \quad p_2=(0,1), \quad p_3=(-1,d), \quad p_4=(0,-1)
$$
The Cox ring $\KK[x_1,x_2,x_3,x_4]$ carries a $\ZZ^2$-grading given by $$
\dg(x_1)=(1,0), \quad \dg(x_2)=(0,1), \quad \dg(x_3)=(1,0), \quad \dg(x_4)=(d,1).
$$
Moreover, $X$ is obtained as a geometric quotient of 
$$
\widehat{X}=\KK^4\setminus(\{x_1=x_3=0\}\cup\{x_2=x_4=0\})
$$
by the action of the torus $H_X=(\KK^{\times})^2$, which is given by the $\ZZ^2$-grading.

In this case the Demazure roots are $(1,0)$, $(-1,0)$ and $(k,1)$ with $0\le k\le d$, and the corresponding homogeneous locally nilpotent derivations are given as
$$
x_1\frac{\partial}{\partial x_3}, \quad x_3\frac{\partial}{\partial x_1}, \quad
x_1^kx_2x_3^{d-k}\frac{\partial}{\partial x_4}.
$$
There are two complete collections of Demazure roots, namely $(1,0), (d,1)$ and $(-1,0),(0,1)$.
They define two normalized additive actions on $X$ defined on $\widehat{X}$ as
$$
(x_1,x_2,x_3,x_4)\mapsto (x_1,x_2,x_3+s_1x_1,x_4+s_2x_1^dx_2)
$$
and
\begin{equation} \label{Fd_normact_eq}
(x_1,x_2,x_3,x_4)\mapsto (x_1+s_1x_3,x_2,x_3,x_4+s_2x_2x_3^d),
\end{equation}
which are rearranged by the automorphism $(x_1,x_2,x_3,x_4)\mapsto (x_3,x_2,x_1,x_4)$.

\smallskip

According to the results of~\cite{De1970} or~\cite{Cox1995}, the automorphism group $\Aut(X)$ is isomorphic to $\KK^{\times}\cdot\PSL(2)\rightthreetimes\,\GG_a^{d+1}$. Consider the case $d=1$. Then a maximal unipotent subgroup of $\Aut(X)$ is isomorphic to unitriangular matrix subgroup of $\GL_3(\KK)$ and acts in the following way:
$$
(x_1,x_2,x_3,x_4)\mapsto (x_1+a_{12}x_3,x_2,x_3,x_4+a_{23}x_1x_2+a_{13}x_2x_3), \; a_{12}, a_{23}, a_{13} \in \KK.
$$
Its two-dimensional (commutative) subgroups have the form $H_{[\alpha:\beta]}$, see Example~\ref{Pn_Demaz_ex}. If $[\alpha:\beta]=[0:1]$ we obtain the action of~$\GG_a^2$ on~$\FF_1$ with one-parameter family of one-dimensional orbits, if $[\alpha:\beta]=[1:0]$ we get normalized additive action~\eqref{Fd_normact_eq}, and points $[\alpha:\beta] \in \PP^1\setminus\{0, \infty\}$ define pairwise isomorphic non-normalized additive actions on~$\FF_1$. Thus, there are exactly two equivalence classes of additive actions on~$\FF_1$. This result is obtained in~\cite[Proposition~5.5]{HaTs1999} by geometric arguments. 
\end{example}

Let us mention one more recent result on additive actions on toric varieties. In~\cite{Shaf2021}, Shafarevich classifies toric projective hypersurfaces admitting an additive action. Every toric hypersurface of dimension~$n$ can be represented by a lattice polytope in the lattice~$M$ such that the number of lattice points inside this polytope is~$n+2$. So the question is when such a polytope is inscribed in a rectangle. In~\cite[Proposition~1]{Shaf2021}, all such polytopes are found. It turns out that apart from the projective space there are two toric projective hypersurfaces admitting an additive action in every dimension $n\ge 2$; they are the quadrics of rank~3 and~4, see \cite[Theorem~2]{Shaf2021}. 

Using the results of~\cite{Cox1995}, Shafarevich computes the automorphism groups of these quadrics. Then he applies the correspondence between additive actions on projective hypersurfaces and pairs $(A,U)$, where $A$ is a local algebra and $U$ is a hyperplane in the maximal ideal of~$A$ which generates~$A$ (see Theorem~\ref{hypHaTsch_theorem}), and finds in~\cite[Section~5]{Shaf2021} the number of non-equivalent additive actions on quadrics of rank~3 and~4 in dimension from 2 to~4, see Remark~\ref{rsha}. 

\subsection{Additive actions on complete toric surfaces}
\label{subcts}

We observe that blowing up repeatedly fixed points, one may obtain infinitely many different (smooth) complete toric surfaces that admit an additive action. 
In this subsection we discuss a result of Dzhunusov~\cite{Dz2019} which clarifies how many additive actions we may have on a complete toric surface. 

Let $X_{\Sigma}$ be a complete toric variety of dimension~$n$ admitting an additive action and $\Sigma$ be the corresponding fan. We begin with some results on the structure of the set of Demazure roots of the cone~$\Sigma$ following \cite[Section~5]{Dz2019}. 

Denote primitive vectors on the rays of the fan $\Sigma$ by~$p_i$, where $1\le i\le m$. It follows from Theorem~\ref{cc} that we can order the vectors~$p_i$ in such a way that the first $n$
vectors form a basis of the lattice~$N$ and the remaining vectors $p_j$ $(n < j \le m)$ are equal to $\sum_{i=1}^n -\alpha_{ji}p_i$ for some non-negative integers~$\alpha_{ji}$.

Let us denote the dual basis of the basis $p_1,\ldots, p_n$ by $p_1^*,\ldots,p_n^*$ and let $\RRR_i=\RRR_{\rho_i}$.

\begin{lemma} {\cite[Lemma~2]{Dz2019}} \label{lr1}
Consider $1\le i\le n$. The set $\RRR_i$ is a subset of the set $-p_i^*+\sum_{l=1,l\ne i}^n\ZZ_{\ge 0} p_j^*$ and the vector~$-p_i^*$ belongs to the set~$\RRR_i$. 
\end{lemma}

Now let us divide the set of Demazure roots $\RRR$ into two classes:
$$
\SSS=\RRR\cap -\RRR, \quad \UUU=\RRR\setminus\SSS. 
$$
Roots in $\SSS$ and $\UUU$ are called \emph{semisimple} and \emph{unipotent}, respectively. 

Consider the set $\text{Reg}(\SSS) = \{u\in N \, : \, \langle u,e\rangle\ne 0 \ \forall e\in\SSS\}$. Any element $u$ from $\text{Reg}(\SSS)$ divides the set of semisimple roots $\SSS$ into two classes as follows:
$$
\SSS^+_u = \{e\in \SSS \, : \, v=\langle u,e\rangle>0\}, \quad \SSS^-_u = \{e\in \SSS \, : \, v=\langle u,e\rangle<0\}.
$$
Any element in $\SSS^+_u$ is called \emph{positive} and any element in $\SSS^-_u$ is called \emph{negative}.

\begin{lemma} {\cite[Proposition~2]{Dz2019}} \label{lr2}
Let $X_{\Sigma}$ be a complete toric variety admitting an additive action. Then
\begin{itemize}
\item[(1)]
any element in $\RRR_j, j > n$, is equal to $p_{i_j}^*$ for some $1\le i_j\le n$;
\item[(2)]
all unipotent Demazure roots lie in the set $\bigcup_{i=1}^n\RRR_i$;
\item[(3)] 
there exists a vector $u\in\text{Reg}(\SSS)$ such that $\SSS_u^+\subseteq\bigcup_{i=1}^n\RRR_i$. 
\end{itemize}
\end{lemma} 

\smallskip

Now we consider a complete toric surface $X_{\Sigma}$ with the fan $\Sigma$ admitting an additive action. We still assume that $p_1, p_2$ form a basis of the lattice $N$ and remaining vectors $p_3,\ldots,p_m$ are non-positive integer combinations of $p_1,p_2$. We say that a fan $\Sigma$ is \emph{wide} if there exist indices $2 < i_1, i_2\le m$ such that $\alpha_{i_11} > \alpha_{i_12}$ and $\alpha_{i_21} < \alpha_{i_22}$. One can check that this condition means $\RRR_1=\{-p_1^*\}$ and $\RRR_2=\{-p_2^*\}$. 

\smallskip

Now we are ready to formulate the main result.

\begin{theorem} {\cite[Theorem~3]{Dz2019}} \label{tdh1}
Let $X_{\Sigma}$ be a complete toric surface admitting an additive action. Then there is only one additive action on $X_{\Sigma}$ if and only if the fan $\Sigma$ is wide; otherwise there
exist exactly two non-equivalent additive actions, one is normalized and the other is not.
\end{theorem}
Lemmas~\ref{lr1}-\ref{lr2} show that for a wide fan $\Sigma$ the maximal unipotent subgroup of the linear algebraic group $\Aut(X_{\Sigma})$ has dimension two, and so it is the only candidate for a commutative unipotent group acting on $X_{\Sigma}$ with an open orbit. To treat the case of a non-wide fan, Dzhunusov classifies pairs of commuting homogeneous LNDs on the Cox ring of the variety $X_{\Sigma}$ and shows that precisely one equivalence class of such pairs corresponds to non-normalized additive actions on $X_{\Sigma}$. 

\subsection{Uniqueness criteria}
\label{subuc}

This subsection contains a criterion of uniqueness for an additive action on a complete toric variety of arbitrary dimension proved by Dzhunusov~\cite{Dz2020}. This result is also based in Lemmas~\ref{lr1}-\ref{lr2}. We keep notation of the previous subsection. 

\begin{theorem} {\cite[Theorem~4]{Dz2020}} \label{tdh2}
Let $X_{\Sigma}$ be a complete toric variety admitting an additive action. Then any additive action on variety $X$ is equivalent to the normalized additive action if and only if the set $\RRR_i$ is equal to $\{-p_i^*\}$ for every $1\le i\le n$.
\end{theorem}

In the proof, Dzhunusov shows that the second claim of the theorem means that the maximal unipotent subgroup of the linear algebraic group $\Aut(X_{\Sigma})$ is a commutative group of dimension $n$, and it is again the only candidate for a commutative unipotent group acting on $X_{\Sigma}$ with an open orbit. If this condition does not hold, an $n$-tuple of pairwise commuting homogeneous LNDs of the Cox ring of the variety $X_{\Sigma}$ is constructed in \cite[Section~6]{Dz2020} and it is proved that this $n$-tuple defines an additive action on $X_{\Sigma}$ that is not equivalent to the normalized one. More precisely, if the normalized additive action corresponds to the tuple of homogeneous LNDs 
$$
(\partial_{-p_1^*}, \partial_{-p_2^*}, \partial_{-p_3^*}, \ldots,\partial_{-p_n^*}), 
$$
then the second tuple is given as 
$$
(\partial_{-p_1^*}, \partial_{-p_2^*}+\partial_{-p_1^*+dp_2^*}, \partial_{-p_3^*}, \ldots,\partial_{-p_n^*})
$$
with some positive integer $d$. 

\begin{corollary} \label{ccl1}
Let $X_{\Sigma}$ be a complete toric variety admitting an additive action. If the rank of the divisor class group $\Cl(X_{\Sigma})$ equals one, then there are at least two non-equivalent
additive actions on $X_{\Sigma}$.
\end{corollary} 

Corollary~\ref{ccl1} covers the case of weighted projective spaces. By \cite[Proposition~2]{AR2017}, the weighted projective space $\PP(a_0,a_1,\ldots, a_n)$, $a_0\le a_1\le\ldots\le a_n$ admits an additive action if and only if $a_0=1$. So there are at least two non-equivalent additive actions on the weighted projective space $\PP(1, a_1,\ldots,a_n)$.

An explicit description of additive actions on weighted projective planes is given in~\cite[Proposition~7]{ABZ2018}. It turns out that, as in the case of the projective plane $\PP^2$, every weighted projective plane $\PP(1,a_1,a_2)$ admits precisely two non-equivalent additive actions. 
%
%
\section{Further results and questions on equivariant completions} 
\label{frec}

The aim of subsection~\ref{sub51} is to collect recent geometric and classificational results on additive actions on Fano varieties. We begin with the surface case and list all singular and generalized del Pezzo surfaces admitting an additive action. In dimension three, we recall a classification of smooth projective 3-folds with an irreducible boundary divisor that admit an additive action due to Hassett and Tschinkel. The next step is a classification of smooth Fano 3-folds of Picard number at least two that admit an additive action. There are 17 varieties satisfying all these conditions. In dimensions starting from four the corresponding classifications are available only under a restriction on the index of a Fano variety. These results are due to Fu and Montero. 

In subsection~\ref{sub-5-1-1} we give a short discussion on so-called Euler-symmetric varieties. Such a variety is defined by the condition that for a generic point $P$ there is a one-dimensional torus $\GG_m$ in the automorphism group such that $P$ is an isolated fixed point for $\GG_m$ and the torus $\GG_m$ acts on the tangent space at the point $P$ by scalar multiplication. It is known that any Euler-symmetric variety admits an additive action, and there is a conjecture that this is a way to describe smooth Fano varieties admitting an additive action. 

In subsection~\ref{sub52} we formulate several open problems and conjectures on equivariant completions of affine spaces. They concern all the subjects discussed in this paper. 

\subsection{Classification results on additive actions on Fano varieties}
\label{sub51}
Let us begin this subsection with the case of surfaces. 
In~\cite{DL2010}, a classification of del Pezzo surfaces that are equivariant compactifications of the group $\GG_a^2$ is given. Recall that del Pezzo surfaces are defined as smooth projective surfaces $X$ whose anticanonical class $-K_X$ is ample. A singular del Pezzo surface is a normal singular projective surface with only ADE-singularities whose anticanonical class is ample. A generalized del Pezzo surface is either a smooth del Pezzo surface or a minimal desingularization of a singular del Pezzo surface. The main result of~\cite{DL2010} claims that if $S$ is a (possibly singular or generalized) del Pezzo surface of degree $d$ defined over a field $K$ of characteristic zero, then $S$ admits an additive action precisely in the following cases:

\begin{enumerate}
\item[(i)]
$S$ has a nonsingular rational $K$-point and is of the form of $\PP^2$, $\PP^1\times\PP^1$, the Hirzebruch surface $\FF_2$, or the corresponding singular del Pezzo surface; 
\item[(ii)]
$S$ is of the form of $\text{Bl}_1(\PP^2)$ or $\text{Bl}_2(\PP^2)$;
\item[(iii)]
$d=7$ and $S$ is of type $A_1$;
\item[(iv)]
$d=6$ and $S$ is of type $A_1$ (with three lines), $2A_1$, $A_2$, or $A_2+A_1$;
\item[(v)]
$d=5$ and $S$ is of type $A_3$ or $A_4$;
\item[(vi)]
$d=4$ and $S$ is of type $D_5$.
\end{enumerate}

\smallskip 

More generally, in \cite{DL2015} the authors determine all (possibly singular) del Pezzo surfaces that are equivariant compactifications of homogeneous spaces of two-dimensional linear algebraic groups. It is well known that besides the torus $\GG_m^2$ and the vector group $\GG_a^2$ the only connected two-dimensional linear algebraic groups are semidirect products $\GG_m\rightthreetimes\GG_a$. The classification result claims that a del Pezzo surface $S$ of degree $d$, possibly singular with rational double points, is an equivariant compactification of some semi-direct product $\GG_m\rightthreetimes\GG_a$ if and only if it has one of the following types:
\begin{enumerate}
\item[(i)]
$d\ge 7$: all types; 
\item[(ii)]
$d=6$: types $A_2+A_1$, $A_2$, $2A_1$, or $A_1$ (with three or four lines);
\item[(iii)]
$d=5$: types $A_3$, $A_2+A_1$, or $A_2$;
\item[(iv)]
$d=4$: types $A_3+2A_1$, $D_4$, or $A_3+A_1$.
\end{enumerate}

Additionally, precisely the following types are equivariant compactifications of a homogeneous space for some semi-direct product $\GG_m\rightthreetimes\GG_a$:
\begin{enumerate}
\item[(i)]
$d=5$: type $A_4$; 
\item[(ii)]
$d=4$: types $D_5$ or $A_4$;
\item[(iii)]
$d=3$: types $E_6$ or $A_5+A_1$.
\end{enumerate}

As we already know, the structure of a torus compactification on a toric variety is unique up to isomorphism, while already the projective plane $\PP^2$ admits two different additive actions. It is proved in \cite[Theorem~3.3]{DL2015} that if $\GG_m\rightthreetimes\GG_a$ is not the direct product $\GG_m\times\GG_a$ then up to equivalence $\PP^2$ admits precisely two different structures of an equivariant compactification of $\GG_m\rightthreetimes\GG_a$. Moreover, it is shown that $\PP^2$ admits infinitely many different structures of an equivariant compactification of a homogeneous space for each $\GG_m\rightthreetimes\GG_a$. 

\smallskip

A characterization of complete $\GG_m$-surfaces admitting an additive action is obtained in~\cite[Proposition~13.17]{HaHu2020}. 

\smallskip

Now we come from surfaces to dimension three.  In~\cite{HaTs1999}, a classification of smooth projective 3-folds of Picard number one admitting an additive action is given. 

\begin{theorem} {\cite[Theorem 6.1]{HaTs1999}}
Let $X$ be a smooth projective 3-fold admitting an additive action with an irreducible boundary divisor $D$. 
Then $X$ is one of the following:
\begin{enumerate}
\item[(i)]
$\PP^3$ with $D$ a hyperplane;
\item[(ii)]
$Q_3\subseteq\PP^4$ a smooth quadric with $D$ a tangent hyperplane section. 
\end{enumerate}
\end{theorem} 

Let us recall that for a Fano variety $X$ of dimension $n$, its \emph{index} $i_X$ is the greatest integer such that $-K_X=i_X H$ for some divisor $H$ on $X$. In the proof of~\cite[Theorem 6.1]{HaTs1999} the authors observe that $-K_X = rD$, where $r\ge 2$. Therefore, $X$ is a rational Fano variety of index $r\ge 2$. They consider the cases $r>2$ and $r=2$ separately and use Furushima's classification of non-equivariant compactifications of affine 3-space. 

\smallskip

A classification of all smooth Fano 3-folds of Picard number at least two that admit additive action is given in~\cite{HM2018}. This classification includes 17~varieties. The authors consider firstly the case of smooth toric Fano 3-folds. They use the classification due to Batyrev and Watanabe-Watanabe and apply to them the criterion of the existence of an additive action on a toric variety from Theorem~\ref{cc}. This way they obtain 13 smooth toric Fano 3-folds admitting an additive action. In the non-toric case, they go through the classification of Mori and Mukai and check the existence of an additive action in each case. This analysis results in 4~smooth non-toric Fano 3-folds with an additive action. 

\smallskip

In higher dimensions, a classification of smooth Fano varieties admitting an additive action is available only for varieties with high index. It is well known that the index $i_X$ of a smooth Fano variety of dimension $n$ does not exceed $n+1$. Classification of smooth Fano varieties of index $i_X\ge n-2$ is obtained in the works by Fujita, Mella, Mukai, and Wisniewski. Based on this classification, a complete list of smooth Fano varieties of dimension $n$ and index $i_X\ge n-2$ that admit an additive action is obtained in~\cite{FM2018}. 

\smallskip

Let us begin with the case of Picard number one.

\begin{theorem} {\cite[Theorem~1.1]{FM2018}}
Let $X$ be an $n$-dimensional smooth projective variety of Picard number one admitting an additive action. Assume that $i_X\ge n-2$. Then $X$ is isomorphic to one of the following varieties:
\begin{enumerate}
\item[(1)]
6 homogeneous varieties of algebraic groups: $\PP^n$, $Q_n$, $\Gr(2, 5)$, $\Gr(2, 6)$, $\SAS_5$, $\Lag(6)$.
\item[(2)]
5 non-homogeneous varieties:
\item[(2-a)]
smooth linear sections of $\Gr(2, 5)$ of codimension $1$ or $2$.
\item[(2-b)]
$\PP^4$-general linear sections of $\SAS_5$ of codimension $1$, $2$ or $3$.
\end{enumerate}
Here $\SAS_5$ and $\Lag(6)$ denote the $10$-dimensional spinor variety and the $6$-dimensional Lagrangian Grassmannian, respectively. 
\end{theorem} 

A classification of smooth $n$-dimensional Fano varieties with $i_X \ge n-2$ of Picard number at least two that admit an additive action is given in~\cite[Section~3]{FM2018}. Here the result is based on Wisniewski's classifications of smooth Fano $n$-dimensional varieties of index $\ge \frac{n+1}{2}$ with Picard number at least two and of Mukai 4-folds with Picard number at least two. 

\subsection{Euler-symmetric varieties} 
\label{sub-5-1-1}
In this subsection we discuss a general construction of varieties with an additive action due to Fu and Hwang~\cite{FH2017}. We work over the field of complex numbers. 

\begin{definition}
\label{ESDef}
Let $Z\subseteq\PP(V)$ be a projective variety. For a nonsingular point $x\in Z$, a $\GG_m$-action on~$Z$ coming from a multiplicative subgroup of $\GL(V)$ is said to be of \emph{Euler type} at~$x$ if $x$ is an isolated fixed point of the restricted $\GG_m$-action on~$Z$ and the induced $\GG_m$-action on the tangent space $T_x(Z)$ is by scalar multiplication. A nonsingular point $x\in Z$ is called \emph{Euler} if there is a $\GG_m$-action on~$Z$ which is of Euler type at~$x$. We say that $Z\subseteq\PP(V)$ is \emph{Euler-symmetric} if there is an open dense subset $W$ in~$Z$ consisting of Euler points.
\end{definition}

\begin{remark}
The condition on the action of~$\GG_m$ on the tangent space~$T_x(Z)$ implies that the $\GG_m$-fixed point~$x$ is isolated. Indeed, since the action of~$\GG_m$ on~$\PP(V)$ is diagonalizable, the point $x$ is contained in a $\GG_m$-invariant open affine chart~$X$ on~$Z$. Since the point~$x$ is nonsingular in~$X$, by~\cite[Theorem~6.4]{PV1994} it is also nonsingular in the subvariety of $\GG_m$-fixed points $X^{\GG_m}$ and $T_x(X^{\GG_m})=T_x(X)^{\GG_m}$. But the action of~$\GG_m$ on~$T_x(X)$ is by scalar multiplication, so the point $x$ is an isolated fixed point in~$X$, and hence in~$Z$. 
\end{remark}

It is proved in~\cite[Proposition~2.3]{FH2017} that for an Euler-symmetric projective variety ${Z\subseteq\PP(V)}$ the group $\Aut_l(Z)\subseteq\PGL(V)$ of linear automorphisms preserving~$Z$ acts on~$Z$ with an open orbit. In fact, Theorem~\ref{FHES} below provides a more concrete version of this result. 

In~\cite[Theorem~3.7]{FH2017}, the authors show that Euler-symmetric varieties are classified by certain algebraic data called symbol systems. Such a description makes this class of varieties accessible for investigation. Our interest in these varieties is explained by the following result; it provides a way for the systematic study of equivariant completions of affine space; see, e.g., Conjecture~\ref{confh} below.

\begin{theorem}\cite[Theorem~3.7(i)]{FH2017}
\label{FHES}
Every Euler-symmetric variety admits an additive action. 
\end{theorem}

Let us sketch the proof of Theorem~\ref{FHES}. Let $Z \subseteq\PP(V)$ be a Euler-symmetric variety of dimension $n$ and $x \in Z$ be an Euler point. Choose homogeneous coordinates in $\PP(V) = \PP^m$ in such a way that $x$ has coordinates $[1:0:\ldots:0]$. Since the $\GG_m$-action of Euler type at $x$ is linear on $\PP^m$, we may assume that $\GG_m$ acts in these coordinates diagonally. 

Let $y_i = \frac{z_i}{z_0}$, $1 \le i \le m$, be coordinates on the affine chart $U_0 = \{z_0 \ne 0\}$ in $\PP^m$. We may assume that the tangent space $T_x(Z)$ is given by the equations $y_{n+1} = \ldots = y_m = 0$. It follows that the torus $\GG_m$ acts on $y_1,\ldots,y_n$ by scalar multiplication. 

By the analytic implicit function theorem, there exists some neighbourghood of $x$, in which $y_1, \ldots, y_n$ are coordinates on $Z$ and $Z$ is given by the system of equations
\begin{equation}\label{ES_eq}
y_{n+i} = h_i(y_1, \ldots, y_n), \quad 1 \le i \le k = m-n
\end{equation}
with some holomorphic functions $h_i$. Any function $h_i$ has a Taylor series in $y_1, \ldots, y_n$ at~$x$. Denote by $h_1^0, \ldots, h_k^0$ the sums of nonzero summands of minimal degrees $d_1, \ldots, d_k$ in Taylor series of $h_1, \ldots, h_k$, respectively. Since the functions $y_{n+1},\ldots,y_m$ are homogeneous with respect to the torus $\GG_m$, we conclude that $h_i=h_i^0$ for all $1\le i\le k$. In particular, the functions $h_i$ are homogeneous polynomials.

\begin{remark}
It is shown in \cite[Proposition~5]{Sha2021} that the variety $Z$ is toric if and only if the functions $h_i$ are monomials corresponding to lattice points in some very ample inscribed in a rectangle polytope; see Definition~\ref{iir}.
\end{remark}

Since $Z$ is an irreducible variety, the intersection $Z \cap U_0$ is an irreducible affine variety in~$U_0$. On the other hand, the system of equations~\eqref{ES_eq} defines an irreducible affine variety $Z'$ in $U_0$ that is isomorphic to the affine space $\AA^n$ with coordinates $y_1, \ldots, y_n$. Since the irreducible varieties $Z'$ and $Z \cap U_0$ coincide in some neighborhood, they are equal to each other, so $Z \cap U_0$ is given by the system of equations $y_{n+i} = h_i(y_1, \ldots, y_n)$, $1 \le i \le k$. In particular, the variety $Z \cap U_0$ is isomorphic to~$\AA^n$. 

Let us consider the functions $h_i$ as elements of $\Sym^{d_i} T_x(Z)^*$. Then the space \[F_x = \KK \oplus T_x(Z)^* \oplus \langle h_1^0, \ldots, h_k^0\rangle\] is called the \emph{fundamental form} of $Z$ at the point $x$. It is a subspace of the direct sum $\bigoplus_{l \ge 0} \Sym^lT_x(Z)^*$. 

Below we will need the following classical result; see e.g. \cite[Theorem~3.3]{FH2017}.

\begin{theorem}[Cartan]
\label{ThCar}
Let $Z$ be a projective variety. Then there exists an open subset $W' \subseteq Z$ such that for any point $x \in W'$ the fundamental form $F_x$ is a symbol system, i.e. for any $h \in F_x$ and any $v \in T_x(Z)$ the derivation of $h$ along $v$ belongs to $F_x$. 
\end{theorem}

Thus, for an Euler-symmetric variety $Z$ we have two open subsets $W$ and $W'$ in $Z$, see Definition~\ref{ESDef}. Such subsets have a non-empty intersection. So we may assume that $x$ is an Euler point and the fundamental form $F_x$ is a symbol system. 

Since $Z \cap U_0$ is isomorphic to $\AA^n$, we have an additive action on $Z \cap U_0$ by parallel translations. This action can be extended to an action on $\PP^m$ provided it can be extended to an action on $U_0$ by affine transformations. 

Let us show that any $\GG_a$-subgroup $H$ of this action of $\GG_a^n$ on $Z\cap U_0$ can be extended to a $\GG_a$-subgroup of affine transformations of $U_0$. Let $\pa$ be a locally nilpotent derivation on $\KK[Z \cap U_0]$ corresponding to $H$. Since $F_x$ is a symbol system, the derivation $\pa$ applied to $h_i$ belongs to $F_x$ as well. On the other hand, $F_x = \KK \oplus T_x(Z)^* \oplus \langle h_1, \ldots, h_k\rangle = \langle 1, y_1, \ldots, y_m\rangle$ since $x$ is an Euler point. Then for any $1 \le i \le k$ we have
\[
\pa(y_{n+i}) = \pa(h_i(y_1, \ldots, y_n)) = \ell_i(1, y_1, \ldots, y_m),
\]
where $\ell_i$ is a linear form. So the action of $s \in \GG_a$ given by $\exp s\pa$ is an action by affine transformations:
\[
y_{n+i} \mapsto y_{n+i} + s \ell_{i,1}(1, y_1, \ldots, y_m) + \frac{s^2}{2} \ell_{i,2}(1, y_1, \ldots, y_m) + \ldots, 
\]
and all $\ell_{i,j}$ are linear forms. Finally, the group $H$ acts on $y_1, \ldots, y_n$ by shifts, hence by affine transformations as well. 

We conclude that the additive action on $Z\cap U_0$ extends to an action on $\PP^m$ and so it induces an additive action on $Z$ since $Z$ is the closure of $Z\cap U_0$. This completes the proof of Theorem~\ref{FHES}.  

\begin{remark}
There are many arguments showing that $\GG_m$- and $\GG_a$-actions are of completely different nature. At the same time, one may prove that the existence of $\GG_m$-actions of certain type implies the existence of $\GG_a$-actions. For example, if an affine variety $X$ admits two actions of the torus $\GG_m$ that do not commute, then $X$ admits a non-trivial $\GG_a$-action; see \cite[Section~3]{FZ} and~\cite[Proof of Theorem~2.1]{AG}. Further, it is shown in \cite[Theorem~1]{Ar2021} that the existence of a $\GG_m$-action of parabolic type on a normal affine variety implies the existence of a non-trivial $\GG_a$-action. Theorem~\ref{FHES} also may be regarded as a result of this form. 
\end{remark}

It turns out that the condition to be Euler-symmetric is a criterion of the existence of an additive action for wide classes of projective varieties. Let us start with the toric case. 

\begin{theorem}\cite[Theorem~3]{Sha2021}
\label{Shtor}
Let $X$ be a projective toric variety. The following conditions are equivalent: 
\begin{itemize}
\item[i)]
the variety $X$ is Euler-symmetric with respect to some embedding into a projective space;
\item[ii)]
the variety $X$ is Euler-symmetric with respect to any linearly non-degenerate linearly normal embedding into a projective space;
\item[iii)]
the variety $X$ admits an additive action.
\end{itemize}
\end{theorem}

Let us sketch the proof of the theorem. In order to obtain implication $i)\to ii)$, one uses linearizability of a (very ample) line bundle on a normal variety with respect to a torus action; see \cite[Proposition~2.4]{KKLV}. This allows to extend a $\GG_m$-action of Euler type from $X$ to an ambient projective space. This implication does not use that $X$ is toric. Implication $ii)\to i)$ is trivial. 

Implication $i)\to iii)$ follows from Theorem~\ref{FHES}. An alternative proof of this implication that uses the specifics of the toric case~--- a description of orbits of the automorphism group on a complete toric variety due to Bazhov~\cite{Ba2013-1} and Corollary~\ref{ncc}~--- is given in~\cite[Proposition~4]{Sha2021}. 

The proof of implication $iii)\to i)$ is divided into three steps. At the first step it is checked in~\cite[Proposition~2]{Sha2021} that every nonsingular $T$-fixed point $x_0$ on a projective toric variety $X$ is Euler with respect to some linearly non-degenerate linearly normal projective embedding. The second step is the claim that a point $x\in X$ is Euler if and only if $x$ can be moved to a nonsingular $T$-fixed point $x_0$ on $X$ by an automorphism of $X$~\cite[Proposition~3]{Sha2021}. Finally, if $X$ admits an additive action then $X$ admits an additive action normalized by the acting torus $T$; see Theorem~\ref{3con}. The open orbit $\Uf$ of the latter additive action is $T$-invariant and isomorphic to an affine space. This implies that $\Uf$ contains a (nonsingular) $T$-fixed point $x_0$. We know that $x_0$ is Euler with respect to some linearly non-degenerate linearly normal projective embedding. Again using~\cite[Proposition~2.4]{KKLV} one may extend the additive action on $X$ to a $\GG_a^n$-action on the ambient projective space. This implies that all points in $\Uf$ are Euler on $X$, and so $X$ is Euler-symmetric. 

\begin{remark}
The proof of implication $i)\to ii)$ shows that the property to be Euler-symmetric for a normal projective variety $Z$ may be defined in interior terms, without involving an embedding into a projective space. Namely, for a nonsingular point $x\in Z$, a $\GG_m$-action on $Z$ is said to be of \emph{Euler type} at $x$ if $x$ is an isolated fixed point of the given action and the induced $\GG_m$-action on the tangent space $T_x(Z)$ is by scalar multiplication. A nonsingular point $x\in Z$ is called \emph{Euler} if there is a $\GG_m$-action on $Z$ which is of Euler type at $x$. We say that a normal projective variety $Z$ is \emph{Euler-symmetric} if there is an open dense subset in $Z$ consisting of Euler points.

In this case, for any linearly non-degenerate linearly normal embedding $Z\subseteq \PP(V)$ we may extend $\GG_m$-actions of Euler type at all Euler points on $Z$ to the action on $\PP(V)$, and so any such embedding is Euler-symmetric in the sense of Definition~\ref{ESDef}. 

It is an interesting problem to prove Theorem~\ref{FHES} without involving an embedding of $Z$ into a projective space. At the moment we do not know such a proof. 
\end{remark}

\begin{remark}
Shafarevich gives some examples illustrating properties of the set of Euler points on a projective toric variety $X$. In particular, such points may not form one orbit of the group~$\Aut(X)$ \cite[Example~1]{Sha2021} and not every point on a smooth Euler-symmetric projective variety is Euler~\cite[Example~2]{Sha2021}.
\end{remark} 

The next result concerns flag varieties. It is mentioned in~\cite[Example~3.13]{FH2017}; below we give a direct proof. 

\begin{theorem} 
A flag variety $G/P$ is Euler-symmetric if and only if it admits an additive action. 
\end{theorem} 

\begin{proof}
We use notation and results of subsection~\ref{sub32}. We assume that $G$ is a connected simple linear algebraic group, $P$ is a maximal parabolic subgroup of $G$ corresponding to a simple root $\alpha_i\in\Delta$, and $G$ is the connected component of the automorphism group $\Aut(G/P)$. In view of Theorem~\ref{gpex} it suffices to prove that $G/P$ is Euler-symmetric if and only if the unipotent radical $P_u^-$ is commutative. The latter is equivalent to commutativity of the tangent algebra $\fp_u^-$.   

We know that $\fp_u^-=\bigoplus_{\alpha\in\Phi^-_i}\fg_{\alpha}$, where $\Phi^-_i$ is the set of negative roots whose decomposition as a linear combination of simple roots contains the root $\alpha_i$. Since the variety $G/P$ is homogeneous, it is Euler-symmetric if and only if there is a $\GG_m$-action of Euler type at the point $x=eP$. The subgroup $\GG_m$ is contained
in $P$ and, up to conjugation, we may assume that $\GG_m$ is a subgroup of the maximal torus $T$. 

Since the action of $T$ on $G$ by conjugation descends to $G/P$ as the action by left translations, its differential acts on the tangent space $T_x(G/P)=\fp_u^-$ by endomorphisms of the Lie algebra structure. 

If all operators of scalar multiplication preserve the Lie bracket, then the Lie bracket is zero. 
Conversely, assume that the Lie algebra $\fp_u^-$ is commutative. Then any root in $\Phi^-_i$ contains $\alpha_i$ with coefficient $-1$; otherwise it is a sum of two roots from $\Phi^-_i$, and so $\fp_u^-$ is not commutative. In this case the $\GG_m$-subgroup in $T$ given by equations $\alpha_j(t)=1$ for all $\alpha_j\in\Delta$, $j\ne i$, acts on $\fp_u^-$ by scalar multiplication, and so $G/P$ is Euler-symmetric. 
\end{proof} 

On the other hand, if a projective hypersurface $X$ admits an additive action, then $X$ need not be Euler-symmetric. Indeed, it is proved in \cite[Example~3.14]{FH2018} that a non-degenerate hypersurface is Euler-symmetric if and only if it is a smooth quadric. By Proposition~\ref{newprop}, these are the only smooth hypersurfaces admitting an additive action. 
But we have non-degenerate singular hypersurfaces admitting an additive action; see e.g. Theorem~\ref{tbazhov}. 

At the same time, in~\cite{Shaf2021} one can find examples of additive actions on degenerate toric quadrics. By Theorem~\ref{Shtor}, such quadrics are Euler-symmetric. 

\smallskip 

The question whether Euler-symmetric varieties are complete intersections in projective spaces is considered in a recent preprint~\cite{Lu2022}. 

\smallskip 

We finish this subsection with more examples of Euler-symmetric varieties. A smooth Euler-symmetric projective surface is successive blow-ups of the projective plane $\PP^2$ or a Hirzebruch surface $\FF_n$ along fixed points of the $\GG_a^2$-action~\cite[Example~3.12]{FH2017}. In higher dimensions, it is shown in~\cite[Example~2.2]{FH2017} that scalar multiplications
of $\AA^n$ can be extended to $\GG_m$-actions of Euler type on the blowup of a smooth subvariety in $\PP^n\setminus\AA^n$. This proves that such blowups are Euler-symmetric, compare with Proposition~\ref{proprr}. 

Finally, in \cite{FH2018} a complete classification of Euler-symmetric varieties of rank~2 is obtained; here the rank is defined in terms of fundamental forms. Such varieties are also called quadratically symmetric. Fu and Hwang observe that if one considers Euler-symmetric varieties as quasi-homogeneous generalizations of Hermitian symmetric spaces, then quadratically symmetric varieties are quasi-homogeneous generalizations of Hermitian symmetric spaces of rank~2. 

\subsection{Open problems}
\label{sub52} 
In this subsection we formulate some open problems and conjectures on additive actions. We hope that they will stimulate further progress in this area. 

\smallskip 

It is well known that a complete normal algebraic variety $X$ is toric if and only if the Cox ring $R(X)$ is a polynomial ring. 

\begin{problem}
Characterize complete normal algebraic varieties admitting an additive action in terms of their Cox rings.
\end{problem}

Applying results of Section~\ref{aatv} it is easy to construct two complete toric varieties $X_{\Sigma_1}$ and $X_{\Sigma_2}$ such that the fans $\Sigma_1$ and $\Sigma_2$ have the same number of rays, and $X_{\Sigma_1}$ admits an additive action, but $X_{\Sigma_2}$ does not. This shows that the existence of an additive action can not be characterized in terms of $R(X)$ as an abstract ring. But the ring $R(X)$ is graded by the group $\Cl(X)$, and we believe that there is a characterization in terms of this grading. 

\smallskip

If we are going to study additive actions via lifting the action to the total coordinate space, the solution of the following problem may be very helpful.

\begin{problem}
Fix positive integers $r$ and $n$, and let $d=r+n$. Describe all affine factorial varieties~$X$ of dimension~$d$ equipped with an effective action of the group $\GG_m^r\times\GG_a^n$ with an open orbit. 
\end{problem}

The case $n=0$ corresponds to affine factorial toric varieties that are known to be direct products of a torus and an affine space. The case $n=1$ also corresponds to affine factorial toric varieties, see~\cite{AK2015}. In turn, in the case $r=0$ we have that $X$ is an affine space with the transitive action of the group $\GG_a^n$. All other cases remain open. 

\smallskip 

As we know from Corollary~\ref{caafin}, the projective space $\PP^n$ with $n\ge 6$ admits infinitely many non-equivalent additive actions. At the same time, the results on the uniqueness of additive actions on smooth projective quadrics and, more generally, on flag varieties suggest that the situation with projective spaces may be in some sense exceptional. This motivates the following problem. 

\begin{problem}
Describe all complete toric varieties that admit infinitely many additive actions. 
\end{problem}

\smallskip

It is natural to ask for a description of all additive action on concrete complete toric varieties. 

\begin{problem}
Describe all additive actions on the weighted projective space $\PP(1,a_1,\ldots,a_n)$. 
\end{problem}

The case of weighted projective planes suggests an idea that the number of additive actions on $\PP(1,a_1,\ldots,a_n)$ depends only on~$n$, but not on the values of
$a_1,\ldots,a_n$. One may expect that there is a description of additive actions on $\PP(1,a_1,\ldots,a_n)$ in terms of some ``weighted Hassett-Tschinkel correspondence''. 

\smallskip

In Proposition~\ref{popor} we observed a connection between additive actions on a given complete variety $X$ and maximal commutative unipotent subgroups in the automorphism group $\Aut(X)$. It is interesting to make this correspondence more precise. 

\begin{problem}
Is it true that any maximal commutative unipotent subgroup of dimension~$n$ in the group $\GL_{n+1}(\KK)$ acts on the projective space $\PP^n$ with an open orbit?
\end{problem} 

Let us show that the same question in the case of the action of the group $\SO_{n+2}(\KK)$ on the quadric $Q_n\subseteq\PP^{n+1}$ has negative answer. By Theorem~\ref{hyp_quadr_nondeg_theor}, there is just one conjugacy class of maximal commutative unipotent subgroups of dimension~$n$ in $\SO_{n+2}(\KK)$ that acts on the quadric~$Q_n$ with an open orbit. At the same time in \cite[Section~6]{Sh2009} an example of a maximal commutative unipotent subgroup of dimension $n$ from another conjugacy class in $\SO_{n+2}(\KK)$ is given. Such a subgroup corresponds to a so-called \emph{free-rowed} maximal commutative nilpotent subalgebra in the Lie algebra $\mathfrak{so}_{n+2}(\KK)$ for $n\ge 6$, see \cite{HWZ1990} for details. 

\begin{problem}
Let $G$ be a connected linear algebraic group and $H$ be a commutative unipotent subgroup of $G$ which is maximal among commutative subgroups in $G$. Does there exist a $G$-variety $X$ such that the induced action of $H$ on $X$ has an open orbit? 
\end{problem}

The next conjecture is about additive actions on projective hypersurfaces. It can be a good complement to the result of Theorem~\ref{tnew}.

\begin{conjecture} 
Let $X \subseteq \PP^{n+1}$ be a degenerate hypersurface admitting an induced additive action. Then there are at least two induced additive actions on~$X$ up to equivalence.
\end{conjecture} 

Finally, let us give the following conjectural characterization of a class of varieties with an additive action. In \cite[Conjecture~5.1]{FH2017}, the following is formulated.

\begin{conjecture} \label{confh}
Let $X$ be a smooth Fano variety of Picard number one which is an equivariant compactification of a vector group. Then $X$ can be realized as an Euler-symmetric projective variety under a suitable projective embedding.
\end{conjecture} 

Partial positive results on this conjecture can be found in~\cite[Section~5]{FH2017}. 


%
\end{document}